\documentclass[a4paper,12pt]{article}
\usepackage{amsmath,amsthm,amssymb,amsfonts,mathrsfs}
\usepackage{bbm}
\usepackage{enumerate}
\usepackage{color}
\usepackage[usenames,dvipsnames]{xcolor}

\usepackage[]{amsmath}
\usepackage{amsthm}
\usepackage{amssymb}
\usepackage{bbm}
\usepackage{microtype}

\PassOptionsToPackage{colorlinks}{hyperref}
\usepackage[textsize=small]{todonotes}

\usepackage[colorlinks]{hyperref}

\usepackage{xr}
\newcommand{\E}{\mathbb{E}}

\newcommand{\N}{\mathbb{N}}

\renewcommand{\P}{\mathbb{P}}

\newcommand{\F}{\mathscr{F}}

\newcommand{\cA}{\mathcal{A}}

\newcommand{\cL}{\mathcal{L}}
\newcommand{\cM}{\mathcal{M}}
\newcommand{\cP}{\mathcal{P}}
\newcommand{\cR}{\mathcal{R}}
\newcommand{\cS}{\mathcal{S}}

\newcommand{\whp}{whp}
\newcommand{\whpdot}{whp.} % for use at end of sentence; allows \whp to be replaced by whp if desired

\newcommand{\condparentheses}[2]{\left(\left.#1\,\right\vert#2\right)}
\newcommand{\condparenthesesreversed}[2]{\left(#1\left\vert\,#2\right.\right)}
\newcommand{\condP}[2]{\mathbb{P}\condparentheses{#1}{#2}}
\newcommand{\condE}[2]{\mathbb{E}\condparentheses{#1}{#2}}

\newcommand{\boundary}{\partial}

\newcommand{\set}[1]{\left\{#1\right\}}
\newcommand{\shortset}[1]{\{#1\}}

\newcommand{\abs}[1]{\left\vert#1\right\vert}
\newcommand{\shortabs}[1]{\vert#1\vert}
\newcommand{\bigabs}[1]{\big\vert#1\big\vert}

\newcommand{\floor}[1]{\lfloor#1\rfloor}
\newcommand{\ceiling}[1]{\lceil#1\rceil}

\newcommand{\ocinterval}[1]{\left(#1\right]} % to help with bracket matching
\newcommand{\cointerval}[1]{\left[#1\right)}

\newcommand{\indicatorofset}[1]{{\mathbbm{1}}_{#1}}
\newcommand{\indicator}[1]{\indicatorofset{\set{#1}}}

\newcommand{\union}{\cup}
\newcommand{\bigunion}{\bigcup}
\newcommand{\intersect}{\cap}

\newcommand{\equalsd}{\overset{d}{=}}

\newcommand{\decreasesto}{\downarrow}

\newcommand{\st}{\text{st}}

\renewcommand{\th}{\text{th}}

\newcommand{\graphG}{\mathcal{G}}

\renewcommand{\emptyset}{\varnothing}

\newcommand{\basicparent}{p}
\newcommand{\parent}[1]{\basicparent\left(#1\right)}
\newcommand{\ancestor}[2]{\basicparent^{#1}\!\left(#2\right)}

\newcommand{\tree}{\mathcal{T}}
\newcommand{\twoPWITs}{\tree}

\newcommand{\SWT}{{\sf SWT}}

\newcommand{\BP}{{\sf BP}}
\newcommand{\fr}{{\rm fr}}
\newcommand{\unfr}{{\rm unfr}}

\newcommand{\lucky}{{\rm lucky}}

\newcommand{\BB}{{\sf BB}}

\newcommand{\cluster}{\mathcal{B}}
\newcommand{\thinnedcluster}{\widetilde{\cluster}}
\newcommand{\thinnedBP}{\widetilde{\BP}}
\newcommand{\fstpond}{\mathscr{V}}

\newcommand{\explore}{\mathcal{E}}
\newcommand{\thinnedExplore}{\widetilde{\explore}}

\newcommand{\epsilonCondition}{\epsilon_0} % in case we decide to change this notation
\newcommand{\deltaCondition}{\delta_0} % in case we decide to change this notation
\newcommand{\epsilonConditiona}{\epsilon_1}% in case we decide to change this notation it will be easier

\newtheorem{theorem}{Theorem}[section]
\newtheorem{prop}[theorem]{Proposition}
\newtheorem{lemma}[theorem]{Lemma}
\newtheorem{coro}[theorem]{Corollary}
\newtheorem{cond}[theorem]{Condition}

\theoremstyle{definition}
\newtheorem{defn}[theorem]{Definition}

\newtheorem{example}[theorem]{Example}

\newcommand{\textandreference}[2]{\texorpdfstring{\hyperref[PartI:#2]{#1\refPartI{#2}}}{#1\refstarPartI{#2}}}
\newcommand{\labelPartI}[1]{\label{PartI:#1}}
\newcommand{\refPartI}[1]{\ref{PartI:#1}}
\newcommand{\refstarPartI}[1]{\ref*{PartI:#1}}
\newcommand{\eqrefPartI}[1]{\eqref{PartI:#1}}

\newcommand{\lbsect}[1]{\labelPartI{s:#1}}
\newcommand{\refsect}[1]{\textandreference{Section~}{s:#1}}

\newcommand{\lbsubsect}[1]{\labelPartI{ss:#1}}
\newcommand{\refsubsect}[1]{\textandreference{Section~}{ss:#1}}

\newcommand{\lbthm}[1]{\labelPartI{t:#1}}
\newcommand{\refthm}[1]{\textandreference{Theorem~}{t:#1}}

\newcommand{\lbprop}[1]{\labelPartI{p:#1}}
\newcommand{\refprop}[1]{\textandreference{Proposition~}{p:#1}}

\newcommand{\lblemma}[1]{\labelPartI{l:#1}}
\newcommand{\reflemma}[1]{\textandreference{Lemma~}{l:#1}}

\newcommand{\lbcoro}[1]{\labelPartI{c:#1}}
\newcommand{\refcoro}[1]{\textandreference{Corollary~}{c:#1}}

\newcommand{\lbcond}[1]{\labelPartI{cond:#1}}
\newcommand{\refcond}[1]{\textandreference{Condition~}{cond:#1}}

\newcommand{\lbdefn}[1]{\labelPartI{d:#1}}
\newcommand{\refdefn}[1]{\textandreference{Definition~}{d:#1}}

\newcommand{\lbexample}[1]{\labelPartI{ex:#1}}
\newcommand{\refexample}[1]{\textandreference{Example~}{ex:#1}}

\newcommand{\lbitem}[1]{\labelPartI{item:#1}}
\newcommand{\refitem}[1]{\refPartI{item:#1}}

\newcommand{\textandreferencePartII}[2]{#1\ref*{PartII:#2}} % Note: no hyperlink created
\newcommand{\refother}[1]{[Part II, #1]}

\newcommand{\refstarPartII}[1]{\ref*{PartII:#1}}

\newcommand{\refsectPartII}[1]{\textandreferencePartII{Section~}{s:#1}}

\newcommand{\refthmPartII}[1]{\textandreferencePartII{Theorem~}{t:#1}}

\definecolor{MyDarkBlue}{rgb}{0,0.08,0.50}
\definecolor{BrickRed}{rgb}{0.65,0.08,0}
\hypersetup{
colorlinks=true,       % false: boxed links; true: colored links
    linkcolor=MyDarkBlue,          % color of internal links
    citecolor=BrickRed,        % color of links to bibliography
    filecolor=red,      % color of file links
    urlcolor=cyan           % color of external links
}

\newcommand{\blank}[1]{}

\numberwithin{equation}{section}

\renewcommand{\epsilon}{\varepsilon}

\newcommand{\e}{{\mathrm e}}
\newcommand{\sss}{\scriptscriptstyle}
\newcommand {\convd}{\stackrel{d}{\longrightarrow}}
\newcommand {\convp}{\stackrel{\sss {\mathbb P}}{\longrightarrow}}
\newcommand {\convas}{\stackrel{a.s.}{\longrightarrow}}

\newcommand{\eqn}[1]{\begin{equation} #1 \end{equation}}

\newcommand{\prob}{{\mathbb P}}

\newcommand{\Poi}{{\rm Poi}}
\newcommand{\IP}{{\rm IP}}
\newcommand{\FPP}{\mathrm{FPP}}
\newcommand{\Op}{O_{\sss \prob}}
\newcommand{\sop}{o_{\sss \prob}}
\newcommand{\FY}{F_{\sss Y}}

\newcommand{\eps}{\epsilon}

\title{Long paths in first passage percolation\\
on the complete graph I. Local PWIT dynamics}
\author{
Maren Eckhoff\thanks{Department of Mathematical Sciences, University of Bath, Bath, BA2 7AY, United Kingdom. Email: {\tt eckhoff.maren@gmail.com}}
\and
Jesse Goodman\thanks{Department of Statistics, University of Auckland, Private Bag 92019, Auckland 1142, New Zealand. Email: {\tt jesse.goodman@auckland.ac.nz}}
\and
Remco van der Hofstad\thanks{Department of Mathematics and Computer Science,
Eindhoven University of Technology, P.O.\ Box 513,
5600~MB Eindhoven, The Netherlands. Email:
{\tt rhofstad@win.tue.nl, f.r.nardi@tue.nl}}
\and
Francesca R.\ Nardi\footnotemark[3]% should give same mark as third distinct institution
}

\begin{document}

\maketitle

\begin{abstract}
We study the random geometry of first passage percolation on the complete graph equipped with independent and identically distributed edge weights, continuing the program initiated by Bhamidi and van der Hofstad~\cite{BhaHof12}.
We describe our results in terms of a sequence of parameters $(s_n)_{n\geq 1}$ that quantifies the extreme-value behavior of small weights, and that describes different universality classes for first passage percolation on the complete graph.
We consider both $n$-independent as well as $n$-dependent edge weights.
The simplest example consists of edge weights of the form $E^{s_n}$, where $E$ is an exponential random variable with mean 1.

In this paper, we investigate the case where $s_n\rightarrow \infty$, and focus on the local neighborhood of a vertex. We establish that the smallest-weight tree of a vertex locally converges to the invasion percolation cluster on the Poisson weighted infinite tree.
In addition, we identify the scaling limit of the weight of the smallest-weight path between two uniform vertices.
\end{abstract}

\section{Model and results}
\lbsect{IntroFPP}
In this paper, we continue the program of studying first passage percolation on the complete graph initiated in \cite{BhaHof12}.
We start by introducing first passage percolation (FPP).
Given a graph $\graphG=(V(\graphG),E(\graphG))$, let $(Y_e^{\sss (\graphG)})_{e\in E(\graphG)}$ denote a collection of positive edge weights.
Thinking of $Y_e^{\sss (\graphG)}$ as the cost of crossing an edge $e$, we can define a metric on $V(\graphG)$ by setting
	\begin{equation}\labelPartI{FPPdistance}
	d_{\graphG,Y^{(\graphG)}}(i,j)=\inf_{\pi \colon i\to j} \sum_{e\in\pi} Y_e^{\sss (\graphG)},
	\end{equation}
where the infimum is over all paths $\pi$ in $\graphG$ that join $i$ to $j$, and $Y^{\sss(\graphG)}$ represents the edge weights $(Y_e^{\sss (\graphG)})_{e\in E(\graphG)}$.
We will always assume that the infimum in \eqrefPartI{FPPdistance} is attained uniquely, by some (finite) path $\pi_{i,j}$.
We are interested in the situation where the edge weights $Y_e^{\sss (\graphG)}$ are \emph{random}, so that $d_{\graphG,Y^{(\graphG)}}$ is a random metric.
In particular, when the graph $\graphG$ is very large, with $\abs{V(\graphG)}=n$ say, we wish to understand the scaling behavior of the following quantities for fixed $i,j \in V(\graphG)$:
\begin{enumerate}
\item
The \emph{distance} $W_n=d_{\graphG,Y^{(\graphG)}}(i,j)$ -- the total edge cost of the optimal path $\pi_{i,j}$;
\item
The \emph{hopcount} $H_n$ -- the number of edges in the optimal path $\pi_{i,j}$;
\item
The \emph{topological structure} -- the shape of the random neighborhood of a point.
\end{enumerate}

In this paper, we consider FPP on the complete graph and focus on problem (c).
In the companion paper \cite{EckGooHofNar14b}, we will use these results to investigate problems (a) and (b). We will often refer to results in \cite{EckGooHofNar14b}, and write, e.g., \refother{Section~\refstarPartII{ss:ConvRWThm}} to refer to \cite[Section \refstarPartII{ss:ConvRWThm}]{EckGooHofNar14b}. We also refer to  \refother{\refsectPartII{DiscExt}} for an extended discussion of the results in these two papers and their relations to the literature.

In \cite{BhaHof12}, the question was raised what the {\it universality classes} are for this model.
We bring the discussion substantially further by describing a way to distinguish several universality classes and by identifying the limiting behavior of first passage percolation in one of these classes.
The cost regime introduced in \eqrefPartI{FPPdistance} uses the information from all edges along the path and is known as the {\it weak disorder} regime.
By contrast, in the {\it strong disorder} regime the cost of a path $\pi$ is given by $\max_{e \in \pi} Y_e^{\sss (\graphG)}$.
We establish a firm connection between the weak and strong disorder regimes in first passage percolation.
Interestingly, this connection also establishes a strong relation to invasion percolation (IP) on the Poisson-weighted infinite tree (PWIT), which is the scaling limit of IP on the complete graph.
This process also arises as the scaling limit of the minimal spanning tree on $K_n$.

Our main interest is in the case $\graphG=K_n$, the complete graph on $n$ vertices $V(K_n)=[n]:=\set{1,\ldots,n}$, equipped with independent and identically distributed (i.i.d.) edge weights $(Y_e^{\sss(K_n)})_{e \in E(K_n)}$.
We write $Y$ for a random variable with $Y\overset{d}{=}Y_e^{\sss(\graphG)}$, and
assume that the distribution function $\FY$ of $Y$ is continuous.
For definiteness, we study the optimal path $\pi_{1,2}$ between vertices $1$ and $2$; by exchangeability, $\pi_{1,2}$ has the same law as $\pi_{u,v}$ for any other $u,v \in [n]$, $u\neq v$.
In \cite{BhaHof12} and \cite{EckGooHofNar13} this setup was studied for the case that $Y_e^{\sss(K_n)}\equalsd E^{s}$ where $E$ is an exponential mean $1$ random variable, and $s>0$ constant and $s=s_n>0$ a null-sequence, respectively.
We start by stating our main theorem for this situation where $s=s_n$ tends to infinity.
First, we introduce some notation:

\paragraph{Notation.}
All limits in this paper are taken as $n$ tends to infinity unless stated otherwise.
A sequence of events $(\mathcal{A}_n)_n$ happens \emph{with high probability (\whp)} if $\P(\mathcal{A}_n) \to 1$.
For random variables $(X_n)_n, X$, we write $X_n \convd X$, $X_n \convp X$ and $X_n \convas X$ to denote convergence in distribution, in probability and almost surely, respectively.
For real-valued sequences $(a_n)_n$, $(b_n)_n$, we write $a_n=O(b_n)$ if the sequence $(a_n/b_n)_n$ is bounded; $a_n=o(b_n)$ if $a_n/b_n \to 0$; $a_n =\Theta(b_n)$ if the sequences $(a_n/b_n)_n$ and $(b_n/a_n)_n$ are both bounded; and $a_n \sim b_n$ if $a_n/b_n \to 1$.
Similarly, for sequences $(X_n)_n$, $(Y_n)_n$ of random variables, we write $X_n=\Op(Y_n)$ if the sequence $(X_n/Y_n)_n$ is tight; $X_n=\sop(Y_n)$ if $X_n/ Y_n \convp 0$; and $X_n =\Theta_{\P}(Y_n)$ if the sequences $(X_n/Y_n)_n$ and $(Y_n/X_n)_n$ are both tight.
Moreover, $E$ always denotes an exponentially distributed random variable with mean $1$.

\subsection{\texorpdfstring{First passage percolation with $n$-dependent edge weights}{First passage percolation with n-dependent edge weights}}
\lbsubsect{FPPn-dependent}

We start by investigating the case where $Y=E^{s_n}$ where $s_n\rightarrow \infty$:

\begin{theorem}[Weight -- $n$-dependent edge weights]\lbthm{IPWeight_Exp}
 Let $Y_e^{\sss(K_n)}\equalsd E^{s_n}$, where $(s_n)_n$ is a positive sequence with $s_n/\log\log n\to \infty$.
Let $M^{\sss(1)},M^{\sss(2)}$ be i.i.d.\ random variables for which $\P(M^{\sss(j)}\leq x)$ is the survival probability of a Poisson Galton--Watson branching process with mean $x$.
Then
	\begin{equation}\labelPartI{WeightRescalesToIP_Exp}
	nW_n^{1/s_n} \convd M^{\sss(1)} \vee M^{\sss(2)}.
	\end{equation}
\end{theorem}

When $s_n\rightarrow \infty$ the values of the random weights $E_e^{s_n}$ depend strongly on the disorder $(E_e)_{e\in E(\graphG)}$, making small values increasingly more favorable and large values increasingly less favorable, thus the weak disorder problem with weights $E_e^{s_n}$ approaches the strong disorder problem.
Mathematically, the elementary limit
	\begin{equation}\labelPartI{PowerToMax}
	\lim_{s\to\infty}(x_1^s+x_2^s)^{1/s}=x_1\vee x_2
	\end{equation}
expresses the convergence of the $\ell^s$ norm towards the $\ell^\infty$ norm and establishes a relationship between the weak disorder regime and the strong disorder regime of FPP.

We continue by discussing the result in \refthm{IPWeight_Exp} in more detail.
Any sequence $(u_n(x))_n$ for which $n\FY(u_n(x))\rightarrow x$ is such that, for i.i.d.\ random variables $(Y_i)_{i \in \N}$ with distribution function $\FY$,
	\eqn{
	\labelPartI{unx-res}
	\prob\Bigl( \min_{i\in [n]} Y_{i}\leq u_n(x) \Bigr) \rightarrow 1-\e^{-x}.
	}
As the distribution function $\FY$ is continuous, we will choose $u_n(x)=\FY^{-1}(x/n)$, so that $n\FY(u_n(x))$  is $x$.
The value $u_n(1)$ is denoted by $u_n$.
In view of \eqrefPartI{unx-res}, the family $(u_n(x))_{x\in(0,\infty)}$ are the \emph{characteristic values} for $\min_{i\in [n]} Y_i$.
See \cite{EmbKluMik97} for a detailed discussion of extreme value theory.
(In the strong disorder regime, $u_n(x)$ varies heavily in $x$ such that the phrase characteristic values can be misleading.)
In the setting of \refthm{IPWeight_Exp}, $\FY(y)=1-\e^{-y^{1/s_n}} \approx y^{1/s_n}$ when $y=y_n$ tends to zero fast enough, so that $u_n(x)$ can be taken as $u_n(x)\approx(x/n)^{s_n}$ (where $\approx$ indicates approximation with uncontrolled error).
Then we see in \eqrefPartI{WeightRescalesToIP_Exp} that $W_n \approx u_n(M^{\sss(1)}\vee M^{\sss(2)})$, which means that the random fluctuations of the weight of the smallest-weight path are of the same order of magnitude as some of the typical values for the minimal edge weight adjacent to vertices $1$ and $2$.
\smallskip

To explain the appearance of the random variables $M^{\sss(1)}$ and $M^{\sss(2)}$, we now informally state that the local neighborhoods of the smallest-weight tree for FPP from a single source converges to invasion percolation on the so-called Poisson-weighted infinite tree.
We start by introducing these notions informally; for more details see \refsubsect{PWITFPPIP}.
In \emph{invasion percolation} (IP) on a weighted graph, we grow the invasion percolation cluster by starting from a single vertex and sequentially adding the edge attached to the cluster with minimal edge weight.
The {\it Poisson-weighted infinite tree} (PWIT), which serves as the large-$n$ limit of the complete graph with i.i.d.\ exponential edge weights, is the infinite weighted tree for
which the edge weight between a vertex and its $i$th child, jointly in $i$, has the same distribution as $E_1+\cdots+E_i$, where $(E_i)_{i\geq 1}$ are i.i.d.\ exponential random variables with mean 1.
The weights of edges between different vertices and their children in the tree are independent.
When performing IP on the PWIT, the largest edge weight ever to be accepted has distribution $M^{\sss(1)}$ as in \refthm{IPWeight_Exp}.
When performing IP on the complete graph started from the two sources $1$ and $2$, the largest edge weight accepted converges in distribution to $M^{\sss(1)} \vee M^{\sss(2)}$.
This explains how \refthm{IPWeight_Exp} can be interpreted in terms of IP on the PWIT and on the complete graph.
The next theorem shows that this relation also holds for local neighborhoods of vertices:

\begin{theorem}\lbthm{IPLocConvForFPP_exp}
Let $Y_e^{\sss(K_n)}\equalsd E^{s_n}$, where $(s_n)_n$ is a positive sequence with $s_n\to \infty$.
For each fixed $m\in \N$, the topology of the FPP smallest-weight tree  to the nearest $m$ vertices, as well as the weights along its edges, converge in distribution to invasion percolation on the Poisson-weighted infinite tree.
\end{theorem}

In \refsubsect{RelToIP}, all  notions needed in \refthm{IPLocConvForFPP_exp} are formally defined. Moreover, the evolution of two smallest-weight trees started from vertices $1$ and $2$ is explored in detail.

\subsection{The universal picture}
\lbsubsect{univ-class}

Our results are applicable not only to the edge-weight distribution $Y_e^{\sss(K_n)}\overset{d}{=}E^{s_n}$ with $s_n\rightarrow \infty$, but to a large family of distributions which we now characterize.
Interestingly, this family includes edge weights whose distribution is independent of $n$.
For fixed $n$, the edge weights $(Y_e^{\sss(K_n)})_{e\in E(K_n)}$ are independent for different $e$.
However, there is no requirement that they are independent over $n$, and in fact in \refsect{Coupling}, we will produce $Y_e^{\sss(K_n)}$ using a fixed source of randomness not depending on $n$.
Therefore, it will be useful to describe the randomness on the edge weights $((Y_e^{\sss(K_n)})_{e\in E(K_n)}\colon n \in \N)$ uniformly across the sequence.
It will be most useful to give this description in terms of exponential random variables.
\begin{description}
\item[Distribution function:]
Choose a distribution function $\FY(y)$ having no atoms, and draw the edge weights $Y_e^{\sss(K_n)}$ independently according to $\FY(y)$:
	\begin{equation}\labelPartI{EdgesByDistrFunct}
	\P(Y_e^{\sss(K_n)} \leq y)= \FY(y).
	\end{equation}

\item[Parametrization by Exponential variables:]
Fix independent exponential mean $1$ variables $(E_e^{\sss (K_n)})_{e\in E(K_n)}$, choose a strictly increasing function $g\colon(0,\infty)\to(0,\infty)$, and define
	\begin{equation}\labelPartI{EdgesByIncrFunct}
	Y_e^{\sss(K_n)}=g(E_e^{\sss (K_n)}).
	\end{equation}
\end{description}
The relations between these parametrizations are given by
\begin{equation}\labelPartI{FYandg}
	\FY(y)=1-\e^{-g^{-1}(y)}\quad \text{and}
	\quad
	g(x)=\FY^{-1} \left( 1-\e^{-x} \right)
	.
\end{equation}
We define
	\eqn{
	\labelPartI{fnFromParamEdges}
	f_n(x)=g(x/n)=\FY^{-1} \left( 1-\e^{-x/n} \right).
	}
Let $Y_1,\dotsc,Y_n$ be i.i.d.\ with $Y_i=g(E_i)$ as in \eqrefPartI{EdgesByIncrFunct}.
Since $g$ is increasing,
	\begin{equation}\labelPartI{minYifn}
	\min_{i\in[n]}Y_i=g\bigl( \min_{i\in[n]} E_i \bigr) \equalsd g(E/n)=f_n(E).
	\end{equation}
Because of this convenient relation between the edge weights $Y_e^{\sss(K_n)}$ and exponential random variables, we will express our hypotheses about the distribution of the edge weights in terms of conditions on the functions $f_n(x)$ as $n\to\infty$.
In \refsubsect{UniversalityClassExamples}, we use this formulation for increasing $g$ or analogous for decreasing function, to explore the universality class in which our results hold in more detail.

Consider first the case $Y_e^{\sss(K_n)} \equalsd E^{s_n}$ from \refthm{IPWeight_Exp}.
From \eqrefPartI{EdgesByIncrFunct}, we have $g(x)=g_n(x)=x^{s_n}$, so that \eqrefPartI{fnFromParamEdges} yields
	\begin{equation}\labelPartI{fnEsnCase}
	f_n(x)=(x/n)^{s_n}=f_n(1)x^{s_n}, \qquad \text{for }\quad Y_e^{\sss(K_n)} \equalsd E^{s_n}.
	\end{equation}
Thus, \eqrefPartI{minYifn}--\eqrefPartI{fnEsnCase} show that the parameter $s_n$ measures the relative \emph{sensitivity} of $\min_{i\in[n]}Y_i$ to fluctuations in the variable $E$.
In general, we will have $f_n(x)\approx f_n(1)x^{s_n}$ if $x$ is appropriately close to 1 and $s_n\approx f_n'(1)/f_n(1)$.
These observations motivate the following conditions on the functions $(f_n)_n$, which we will use to relate the distributions of the edge weights $Y_e^{\sss(K_n)}$, $n\in\N$, to a sequence $(s_n)_n$:

\begin{cond}[Scaling of $f_n$]
\lbcond{scalingfn}
For every $x\geq 0$,
	\eqn{\labelPartI{fnx1oversn}
	\frac{f_n(x^{1/s_n})}{f_n(1)} \to x.
	}
\end{cond}

\begin{cond}[Density bound for small weights]\lbcond{LowerBoundfn}
There exist $\epsilonCondition>0$, $\deltaCondition \in \ocinterval{0,1}$ and $n_0 \in \N$ such that
	\begin{equation}\labelPartI{BoundfnSmall}
	\epsilonCondition s_n \leq x \frac{d}{dx} \log f_n(x) \leq s_n/\epsilonCondition,
	\qquad \text{whenever } \quad 1-\deltaCondition\leq x\leq 1,\text{ and }n\geq n_0.
	\end{equation}
\end{cond}
\begin{cond}[Density bound for large weights]
\lbcond{boundfn}\hfill
\begin{enumerate}
\item \lbitem{boundfnR}
For all $R>1$, there exist $\epsilonConditiona>0$ and $n_1\in \N$ such that for every $1\leq x\leq R$ and $n \ge n_1$,
	\eqn{
	\labelPartI{fn-lb}
	x\frac{d}{dx} \log f_n(x)\geq {\epsilonConditiona} s_n.
	}
\item \lbitem{boundfnlog}
For all $C>1$, there exist $\epsilonConditiona>0$ and $n_1\in\N$ such that \eqrefPartI{fn-lb} holds for every $n\geq n_1$ and every $x\geq 1$ satisfying $f_n(x)\leq C f_n(1) \log n$.
\end{enumerate}
\end{cond}

Notice that \refcond{scalingfn} implies that $f_n(1) \sim u_n$ whenever $s_n=o(n)$.
Indeed, by \eqrefPartI{fnFromParamEdges} we can write $u_n=f_n(x_n^{1/s_n})$ for $x_n = (-n \log(1-1/n))^{s_n}$.
Since $s_n=o(n)$, we have $x_n=1-o(1)$ and the monotonicity of $f_n$ implies that $f_n(x_n^{1/s_n})/f_n(1)\to 1$.
We remark also that \eqrefPartI{unx-res} remains valid if $u_n(x)$ is replaced by $f_n(x)$.

In \refsubsect{UniversalityClassExamples} the universality class is explored in more detail, and  several examples that satisfy Conditions~\refPartI{cond:scalingfn}--\refPartI{cond:boundfn} are discussed.
The results from \refsubsect{FPPn-dependent} generalize as follows:

\begin{theorem}\lbthm{IPWeightForFPP}
Suppose that \refcond{boundfn}~\refitem{boundfnR} is satisfied for a positive sequence $(s_n)_n$ with $s_n/\log\log n\to \infty$.
Let $M^{\sss(1)},M^{\sss(2)}$ be i.i.d.\ random variables for which $\P(M^{\sss(j)}\leq x)$ is the survival probability of a Poisson Galton--Watson branching process with mean $x$.
Then
	\begin{equation}\labelPartI{WeightRescalesToIP_gen}
	f_n^{-1}(W_n) \convd M^{\sss(1)} \vee M^{\sss(2)}.
	\end{equation}
\end{theorem}

\refthm{IPWeightForFPP} is proved in \refsect{IPWeightProof}.

\begin{theorem}\lbthm{IPLocConvForFPP_gen}
Let $(s_n)_n$ be a positive sequence with $s_n\to \infty$.
Suppose that \refcond{boundfn}~\refitem{boundfnR} holds and in addition \refcond{LowerBoundfn} holds for any value $\deltaCondition\in(0,1)$.
Then, for each fixed $m\in \N$, the topology of the FPP smallest-weight tree to the nearest $m$ vertices, as well as the weights along its edges, converges in distribution to invasion percolation on the Poisson-weighted infinite tree.
\end{theorem}

\refthm{IPLocConvForFPP_gen} will be formalized in \refthm{CoupIP-PWIT}.

\refthm{IPWeight_Exp} follows from \refthm{IPWeightForFPP}.  \refthm{IPLocConvForFPP_exp} follows from \refthm{IPLocConvForFPP_gen}
because in the case $Y_e^{\sss (K_n)}\equalsd E^{s_n}$, \eqrefPartI{BoundfnSmall} and \eqrefPartI{fn-lb} hold identically with $\eps_1=1 = \eps_0$.
Note that neither \refthm{IPWeightForFPP} nor \refthm{IPLocConvForFPP_gen} requires \refcond{scalingfn}.

\subsection{Regularity: Sufficient condition for the universality class}
\lbsubsect{UniversalityClassExamples}

Theorems~\refPartI{t:IPWeightForFPP} and \refPartI{t:IPLocConvForFPP_gen} apply also to $n$-independent distributions.
The following example collects some edge-weight distributions that satisfy Conditions~\refPartI{cond:scalingfn}--\refPartI{cond:boundfn}:

\begin{example}[Examples of weight distributions]
\lbexample{AllExamples}\hfill
\begin{enumerate}
\item\lbitem{EsnExample}
Let $(s_n)_{n \in \N} \in (0,\infty)^{\N}$ with $s_n \to \infty$.
Take $Y_e^{\sss(K_n)} \equalsd E^{s_n}$, i.e., $\FY(y)=1-\e^{-y^{1/s_n}}$.

\item\lbitem{PowerOfLogExampleFY}
Let $\rho>0, \kappa \in (0,1)$.
Take $Y_e^{\sss(K_n)} \equalsd \exp(-(E/\rho)^{1/\kappa})$, for which $\;\FY(y)=\exp(-\rho(\log(1/y))^\kappa)$, and define $s_n=\frac{(\log n)^{1/\kappa-1}}{\kappa \rho^{1/\kappa}}$.

%\item\lbitem{PowerOfLogExampleg}
%Let $\rho>0, \kappa \in (0,1)$.
%Take $Y_e^{\sss(K_n)} \equalsd \e^{-(\rho^{-1}\log(1/E))^{1/\kappa}}$, i.e., $\FY(y)=1-\exp(-\e^{-\rho(\log 1/y)^\kappa})$, and define $s_n=...$.
%\JG{Actually this example gets too complicated.}

\item \lbitem{PowerOfnExampleFY}
Let $\rho>0, \alpha \in (0,1)$.
Take $Y_e^{\sss(K_n)} \equalsd \exp\left( - \rho \e^{\alpha E}/\alpha \right)$, for which
	\[
	\FY(y)=\exp(-\frac{1}{\alpha} \log(\frac{\alpha}{\rho} \log(1/y))),
	\]
and define $s_n=\rho n^{\alpha}$.

\item \lbitem{PowerOfnExampleg}
Let $\rho>0,\alpha \in (0,1)$.
Take $Y_e^{\sss(K_n)} \equalsd \exp(-\rho E^{-\alpha}/\alpha)$, for which
	\[
	\FY(y)=1-\exp\bigl( -(\frac{\alpha}{\rho} \log(1/y))^{-1/\alpha} \bigr),
	\]
and define $s_n=\rho n^{\alpha}$.
\end{enumerate}

For \refexample{AllExamples}~\refitem{EsnExample} and \refitem{PowerOfnExampleg}, the edge weights distributions are expressed in terms of strictly increasing functions of $E$, namely $g(x)=x^{s_n}$ and  $g(x)=\exp(-\rho x^{-\alpha}/\alpha)$ respectively.
We therefore get $\FY(y)$ by \eqrefPartI{FYandg}.
For \refexample{AllExamples}~\refitem{PowerOfLogExampleFY} and \refitem{PowerOfnExampleFY}, the function is strictly decreasing instead, so we get $\FY(y)$ by the following analogue of \eqrefPartI{FYandg} for the case $Y=h(E)$ and $h$ strictly decreasing:

\begin{equation}\labelPartI{FYandgdecreasing}
	\FY(y)=\e^{-h^{-1}(y)}\quad \text{and}
	\quad
	h(x)=\FY^{-1} \left( \e^{-x} \right)
	.
\end{equation}

Indeed,
	\begin{equation}\labelPartI{FYandgdecreasingexpla}
	\FY(y)= \P(Y_e^{\sss(K_n)} \leq y)
	=\P(E_e^{\sss(K_n)} \geq h^{-1}(y))
	=\e^{-h^{-1}(y)}.
	\end{equation}
For \refexample{AllExamples}~\refitem{PowerOfLogExampleFY}
$h(x)=\exp(-(x/\rho)^{1/\kappa})$ and $h^{-1}(y)=\rho(\log(1/y))^{\kappa}$. By \eqrefPartI{FYandgdecreasing} $\;\FY(y)=\exp(-\rho(\log(1/y))^\kappa)$.
Analogously, for \refexample{AllExamples}~\refitem{PowerOfnExampleFY}
we have $h(x)=\exp\left( - \rho \e^{\alpha x}/\alpha \right)$ and $h^{-1}(y)=\frac{1}{\alpha} \log(\frac{\alpha}{\rho} \log(1/y)))$.
By \eqrefPartI{FYandgdecreasing} we get $\;\FY(y)=\exp(-\frac{1}{\alpha} \log(\frac{\alpha}{\rho} \log(1/y)))$.
 In \refcoro{examplesconditions} below, we derive the values $s_n$ that were stated in \refexample{AllExamples}.
\end{example}

\refexample{AllExamples}~\refitem{PowerOfLogExampleFY}--\refitem{PowerOfnExampleg} are instances of the following family of edge-weight distributions:
Suppose that the edge weights $Y_e^{\sss(K_n)}$ follow an $n$-independent distribution $\FY$ with no atoms and $u\mapsto u(\FY^{-1})'(u)/\FY^{-1}(u)$ is regularly varying with index $-\alpha$ as $u\decreasesto 0$, i.e.,
\begin{equation}
\labelPartI{SlowVarCondFY}
u \frac{d}{du}\log \FY^{-1}(u)=u^{-\alpha}L(1/u) \qquad \text{for all } u \in (0,1),
\end{equation}
where $t \mapsto L(t)$ is slowly varying as $t\to \infty$.
The following proposition shows that edge weights of this type satisfy our conditions with $s_n=n^\alpha L(n)$:

\begin{prop}\lbprop{RegVarImplies}
For an $n$-independent distribution $Y_e^{\sss(K_n)} \equalsd g(E)$ with distribution function $\FY$, the function $u \mapsto u(\FY^{-1})'(u)/\FY^{-1}(u)$ is regularly varying with index $-\alpha$ as $u\decreasesto 0$ if and only if $x \mapsto xg'(x)/g(x)$ is regularly varying with index $-\alpha$ as $x\decreasesto 0$.
If either of these equivalent conditions holds then, writing for all $x>0$, $u \in (0,1)$,
	\begin{equation}\labelPartI{SlowVarCondg}
	u \frac{d}{du} \log \FY^{-1}(u)=u^{-\alpha}L(1/u) \quad  \text{and } \quad  x \frac{d}{dx} \log g(x)=x^{-\alpha}\widetilde{L}(1/x)
	\end{equation}
it holds that
	\begin{equation}%\labelPartI{}
	L(t)\sim\widetilde{L}(t) \qquad \text{as }t\to\infty.
	\end{equation}
Furthermore if either of the asymptotically equivalent sequences
	\begin{equation}%\labelPartI{}
	s_n=n^\alpha L(n) \qquad\text{or}\qquad s_n=n^\alpha \widetilde{L}(n)
	\end{equation}
satisfies $s_n\to\infty$, then Conditions~\refPartI{cond:scalingfn}, \refPartI{cond:LowerBoundfn} and \refPartI{cond:boundfn}~\refitem{boundfnR} hold, while if in addition $s_n/\log\log n\to\infty$ then \refcond{boundfn}~\refitem{boundfnlog} holds as well.
\end{prop}

Note that replacing $(s_n)_n$ by an asymptotically equivalent sequence makes no difference in Conditions~\refPartI{cond:scalingfn}--\refPartI{cond:boundfn}.
Moreover, every sequence $(s_n)_n$ of the form $s_n=n^{\alpha} L(n)$, for $\alpha \ge 0$ and $L$ slowly varying at infinity, can be obtained from a $n$-independent distribution by taking $\log \FY^{-1}(u)=\int u^{-1-\alpha}L(1/u)du$, i.e., the indefinite integral of the function $u\mapsto u^{-1-\alpha}L(1/u)$.
For a given $n$-independent distribution, \refprop{RegVarImplies} allows us to define the sequence $s_n$ using either $\FY$ or $g$, whichever is more convenient.
See the proof of  \refcoro{examplesconditions} for details.

\begin{proof}[Proof of \refprop{RegVarImplies}]
The equivalence follows by noting that $\widetilde{L}(1/x)=\e^{-x}(\frac{x}{1-\e^{-x}})^{1+\alpha} L(1/(1-\e^{-x}))$ (see \eqrefPartI{FYandg} and \eqrefPartI{SlowVarCondg}), so that $\widetilde{L}(t)$ is slowly varying as $t\to \infty$ if and only if $L(t)$ is.
Conditions~\refPartI{cond:LowerBoundfn} and \refPartI{cond:boundfn}~\refitem{boundfnR} follow from the observation that (recall \eqrefPartI{fnFromParamEdges})
\begin{equation}%\labelPartI{}
\frac{xf_n'(x)}{f_n(x)}=\frac{\frac{x}{n}g'(\frac{x}{n})}{g(\frac{x}{n})}=n^{\alpha} x^{-\alpha} \widetilde{L}(n/x)
\end{equation}
and the definition of slowly varying, while \refcond{scalingfn} follows from the computation
\begin{align}
\log\frac{f_n(x^{1/s_n})}{f_n(1)}
&=
\int_1^{x^{1/s_n}}\frac{f_n'(\xi)}{f_n(\xi)}d\xi
=
\int_1^x \frac{f'_n(u^{1/s_n})}{f_n(u^{1/s_n})} \frac{1}{s_n} u^{1/s_n-1} du
\notag\\
&=
\int_1^x \frac{n^\alpha u^{-\alpha/s_n} \widetilde{L}(n/u^{1/s_n})}{n^\alpha \widetilde{L}(n)}\frac{du}{u}
\to \log x,
\labelPartI{nIndepScaling}
\end{align}
since $\widetilde{L}(n/u^{1/s_n})/\widetilde{L}(n)\to 1$ and $u^{-\alpha/s_n} \to 1$, uniformly over $u$ in a compact subset of $(0,\infty)$.

Finally, assume $s_n/\log\log n\to\infty$ and fix any $C,R \in (1,\infty)$.
By \refcond{boundfn}~\refitem{boundfnR}, there is $\epsilonConditiona=\eps(R)>0$ such that for sufficiently large $n$, $\log(f_n(R)/f_n(1))\ge \epsilonConditiona s_n \int_1^R 1/\tilde{x}\, d\tilde{x}$.
Therefore, for sufficiently large $n$, $f_n(R) \ge f_n(1) R^{\epsilonConditiona s_n} \ge Cf_n(1) \log n$ by Condition \refPartI{cond:boundfn}~\refitem{boundfnR} and the inequality \eqrefPartI{fn-lb} holds for any value $x\geq 1$ such that $f_n(x)\leq Cf_n(1)\log n$, as required by \refcond{boundfn}~\refitem{boundfnlog}.
\end{proof}

\begin{coro}
\lbcoro{examplesconditions}
The edge-weight distributions in \refexample{AllExamples}~\refitem{EsnExample}--\refitem{PowerOfnExampleg}, with the associated sequences $(s_n)_n$, satisfy conditions Conditions~\refPartI{cond:scalingfn}--\refPartI{cond:boundfn}.
\end{coro}

\begin{proof}
For \refexample{AllExamples}~\refitem{EsnExample} is immediate. For \refexample{AllExamples}~\refitem{PowerOfnExampleg},  \refprop{RegVarImplies} allows us to derive the sequence $s_n$.
In this case it is more convenient to use the second equation in \eqrefPartI{SlowVarCondg} with $g(x)=\exp(-\rho x^{-\alpha}/\alpha)$, thus we get
\begin{equation}\labelPartI{Findsn(d)}
	x \frac{d}{dx} \log g(x)=\rho x^{-\alpha} \;\;\hbox{ and }\;\; \widetilde{L}(1/x)=\rho
	.
	\end{equation}
We conclude that the Conditions~\refPartI{cond:scalingfn}, \refPartI{cond:LowerBoundfn} and \refPartI{cond:boundfn}~\refitem{boundfnR} hold with $s_n=\rho n^\alpha$.
Since $s_n/\log\log n\to\infty$, \refcond{boundfn}~\refitem{boundfnlog} holds as well.

For \refexample{AllExamples}~\refitem{PowerOfnExampleFY} it is more convenient to use the first equation in \eqrefPartI{SlowVarCondg} with
$\FY(y)=\exp(-\frac{1}{\alpha} \log(\frac{\alpha}{\rho} \log(1/y)))$.
Thus $\FY^{-1}(u)=\exp(-\frac{\alpha}{\rho} u^{-\alpha})$,
\begin{equation}\labelPartI{Findsn(c)}
	u \frac{d}{du} \log \FY^{-1}(u)=\rho u^{-\alpha} \;\;\hbox{ and }\;\; {L}(1/u)=\rho
	.
	\end{equation}
By \refprop{RegVarImplies}, we can set $s_n=\rho n^\alpha$, and Conditions~\refPartI{cond:scalingfn}--\refPartI{cond:boundfn} hold since $s_n/\log\log n\to\infty$.

For \refexample{AllExamples}~\refitem{PowerOfLogExampleFY}, we first observe that $U \equalsd 1-U$ and $U \equalsd \e^{-E}$, so $\e^{-E} \equalsd 1-\e^{-E}$ and $E \equalsd - \log(1-\e^{-E})$.
Thus we get
\begin{equation}
Y_e^{\sss(K_n)} \equalsd \exp(-(E/\rho)^{1/\kappa})\equalsd \exp\left[-\left(- \log(1-\e^{-E})/\rho\right)^{1/\kappa}\right].
\end{equation}
In this case is more convenient to use the second equation in \eqrefPartI{SlowVarCondg} with
\begin{equation}
g(x)= \exp\left[-\left(- \log(1-\e^{-x})/\rho\right)^{1/\kappa}\right],
\end{equation}
thus we get
	\begin{equation}\labelPartI{Findsn(b)}
	x \frac{d}{dx} \log g(x)=
	\frac{x}{\kappa \rho^{1/\kappa}} \left(-\log(1-\e^{-x})\right)^{1/\kappa-1}\frac{\e^{-x}}{1-\e^{-x}}
	\sim \frac{1}{\kappa \rho^{1/\kappa}} \left(\log(1/x)\right)^{1/\kappa-1}
	,
	\end{equation}
so that $L(n)\sim \frac{1}{\kappa \rho^{1/\kappa}} \left(\log n\right)^{1/\kappa-1}$.
Asymptotic equivalence does not affect Conditions~\refPartI{cond:scalingfn}--\refPartI{cond:boundfn}, so we can set $s_n=\frac{ (\log n)^{1/\kappa-1}}{\kappa \rho^{1/\kappa}}$.
By \refprop{RegVarImplies} and the observation that $s_n/\log\log n\to\infty$ since $\kappa<1$, Conditions~\refPartI{cond:scalingfn}--\refPartI{cond:boundfn} hold.
\end{proof}

An example of an edge-weight distribution that is $n$-dependent but not $Y_e\equalsd E^{s_n}$ is the following:

\begin{example}
\lbexample{ExZtosn}
Let $(s_n)_n$ be a positive sequence with $s_n \to \infty$, $s_n=o(n^{1/3})$.
Let $U$ be a positive, continuous random variable with distribution function $G$ and
	\begin{equation}\labelPartI{asympExd}
	\lim_{u \downarrow 0} u \frac{d}{du} \log G^{-1}(u) =1.
	\end{equation}
Take $Y_e^{\sss(K_n)} \equalsd U^{s_n}$, i.e., $\FY(y)=G(y^{1/s_n})$.
Examples for $G$ are the uniform distribution on an interval $(0,b)$, for any $b>0$, or the exponential distribution with any parameter.
\end{example}

\begin{lemma}
\lblemma{ExZtosnCond}
The edge weights of \refexample{ExZtosn} satisfy Conditions~\refPartI{cond:LowerBoundfn}--\refPartI{cond:boundfn}, and Condition \refPartI{cond:scalingfn} when $s_n/\log\log{n}\rightarrow \infty$.
\end{lemma}

\begin{proof}
Write $U \equalsd h(E)$ for an increasing function $h\colon(0,\infty)\to(0,\infty)$, so that $f_n(x)=h(x/n)^{s_n}$ and
\begin{equation}%\labelPartI{}
\frac{xf_n'(x)}{f_n(x)} = s_n \frac{\frac{x}{n}h'(\frac{x}{n})}{h(\tfrac{x}{n})}.
\end{equation}
By \eqrefPartI{FYandg}, the assumption \eqrefPartI{asympExd} on the distribution function of $U$ is equivalent to
\begin{equation}
\lim_{x\decreasesto 0} x \frac{d}{dx} \log h(x) = 1.
\end{equation}
Conditions~\refPartI{cond:LowerBoundfn}--\refPartI{cond:boundfn} follow immediately, and \refcond{scalingfn} follows as in \refprop{RegVarImplies}.
\end{proof}

\section{Fine results on the IP part of FPP}\lbsect{IPpartFPP}

In this section, we argue that the optimal path between two vertices can be divided into two parts:
The local neighbourhoods of the two endpoints that follow IP dynamics by \refthm{IPLocConvForFPP_gen},
and the main part of the path which is characterized in terms of a branching process.
The main results of this paper connect the maximal weight $M^{\sss(1)}$ in IP to the transition time between these two regimes, and give a detailed description of the topology of the neighbourhood contained in the IP part.

\subsection{Coupling FPP to a continuous-time branching process}\lbsubsect{PWITFPPIP}

To understand the random neighbourhood of a vertex in the complete graph, we study the first passage \emph{exploration process}.
Recall from \eqrefPartI{FPPdistance} that $d_{K_n,Y^{\sss(K_n)}}(i,j)$ denotes the total cost of the optimal path $\pi_{i,j}$ between vertex $i$ and $j$.
For a vertex $j\in V(K_n)$, let the \emph{smallest-weight tree} $\SWT_t^{\sss(j)}$ be the connected subgraph of $K_n$ defined by
	\begin{align}
	\begin{split}
	\labelPartI{OneSourceSWT}
	V(\SWT_t^{\sss(j)})
	&=
	\set{i\in V(K_n)\colon d_{K_n,Y^{\sss(K_n)}}(i,j)\leq t} \! ,
	\\
	E(\SWT_t^{\sss(j)})
	&=
	\set{e\in E(K_n)\colon e\in\pi_{j,i}\text{ for some }i\in V(\SWT_t^{\sss(j)})}.
	\end{split}
	\end{align}
Note that $\SWT_t^{\sss(j)}$ is indeed a tree: if two optimal paths $\pi_{j,k},\pi_{j,k'}$ pass through a common vertex $i$, both paths must contain $\pi_{j,i}$ since the minimizers of \eqrefPartI{FPPdistance} are unique.

To visualize the process $(\SWT_t^{\sss(j)})_{t\geq 0}$, think of the edge weight $Y_e^{\sss (K_n)}$ as the time required for fluid to flow across the edge $e$.
Place a source of fluid at $j$ and allow it to spread through the graph.
Then $V(\SWT_t^{\sss(j)})$ is precisely the set of vertices that have been wetted by time $t$, while $E(\SWT_t^{\sss(j)})$ is the set of edges along which, at any time up to $t$, fluid has flowed from a wet vertex to a previously dry vertex.
Equivalently, an edge is added to $\SWT_t^{\sss(j)}$ whenever it becomes completely wet, with the additional rule that an edge is not added if it would create a cycle.

Because fluid begins to flow across an edge only after one of its endpoints has been wetted, the \emph{age} of a vertex -- the length of time that a vertex has been wet -- determines how far fluid has traveled along the adjoining edges.
Given $\SWT_t^{\sss(j)}$, the future of the exploration process will therefore be influenced by the current ages of vertices in $\SWT_t^{\sss(j)}$, and the nature of this effect depends on the probability law of the edge weights $(Y_e^{\sss (K_n)})_e$.
In the sequel, for a subgraph $\graphG=(V(\graphG),E(\graphG))$ of $K_n$, we write $\graphG$ instead of $V(\graphG)$ for the vertex set when there is no risk of ambiguity.

To study the smallest-weight tree from a vertex, say vertex $1$, let us consider the time until the first vertex is added.
By construction, $\min_{i\in [n]\setminus \set{1}} Y_{\set{1,i}}^{\sss(K_n)} \equalsd f_n(\frac{n}{n-1}E)$ (cf.\ \eqrefPartI{minYifn}), where $E$ is an exponential random variable of mean $1$.
We next extend this to describe the distribution of the order statistics of the weights of edges from vertex 1 to \emph{all} other vertices.

Denote by $Y_{(k)}^{\sss(K_n)}$ the $k^\th$ smallest weight from $(Y_{\set{1,i}}^{\sss(K_n)})_{i\in [n]\setminus \set{1}}$.
Then $(Y_{(k)}^{\sss(K_n)})_{k\in [n-1]}\equalsd (f_n(S_{k,n}))_{k\in [n-1]}$, where $S_{k,n}=\sum_{j=1}^k \frac{n}{n-j} E_j$ and $(E_j)_{j\in [n-1]}$ are i.i.d.\ exponential random variables with mean $1$.
The fact that the distribution of $S_{k,n}$ depends on $n$ is awkward, and can be changed by using a \emph{thinned} Poisson point process.
Let $X_1<X_2<\dotsb$ be the points of a Poisson point process with intensity 1, so that $X_k\equalsd\sum_{j=1}^kE_j=\lim_{n\to \infty} S_{k,n}$.
To each $k\in\N$, we associate a \emph{mark} $M_k$ which is chosen uniformly at random from $[n]$, different marks being independent.
We \emph{thin} a point $X_k$ when $M_k=1$ (since 1 is the initial vertex) or when $M_k=M_{k'}$ for some $k'<k$.
Then
\begin{equation}\labelPartI{SingleVertexCoupling}
(Y_{(k)}^{\sss(K_n)})_{k\in [n-1]}\equalsd (f_n(X_k))_{k\in\N, \, X_k\text{~unthinned}}.
\end{equation}
In the next step, we extend this result to the smallest-weight tree $\SWT^{\sss(1)}$ using a relation to FPP on the Poisson-weighted infinite tree.
Before giving the definitions, we recall the Ulam--Harris notation for describing trees.

Define the tree $\tree^{\sss(1)}$ as follows.
The vertices of $\tree^{\sss(1)}$ are given by finite sequences of natural numbers headed by the symbol $\emptyset_1$, which we write as $\emptyset_1 j_1 j_2\dotsb j_k$.
The sequence $\emptyset_1$ denotes the root vertex of $\tree^{\sss(1)}$.
We concatenate sequences $v=\emptyset_1 i_1\dotsb i_k$ and $w=\emptyset_1 j_1\dotsb j_m$ to form the sequence $vw=\emptyset_1 i_1\dotsb i_k j_1\dotsb j_m$ of length $\abs{vw}=\abs{v}+\abs{w}=k+m$.
Identifying a natural number $j$ with the corresponding sequence of length 1, the $j^\th$ child of a vertex $v$ is $vj$, and we say that $v$ is the parent of $vj$.
Write $\parent{v}$ for the (unique) parent of $v\neq\emptyset_1$, and $\ancestor{k}{v}$ for the ancestor $k$ generations before, for $k\leq \abs{v}$.

We can place an edge (which we could consider to be directed) between every $v\neq\emptyset_1$ and its parent; this turns $\tree^{\sss(1)}$ into a tree with root $\emptyset_1$.
With a slight abuse of notation, we will use $\tree^{\sss(1)}$ to mean both the set of vertices and the associated graph, with the edges given implicitly according to the above discussion, and we will extend this convention to any subset $\tau\subset\tree^{\sss(1)}$.
We also write $\boundary \tau=\set{v\notin\tau\colon \parent{v}\in\tau}$ for the set of children one generation away from $\tau$.

The Poisson-weighted infinite tree is an infinite edge-weighted tree in which every vertex has infinitely many (ordered) children.
To describe it formally, we associate weights to the edges of $\tree^{\sss(1)}$.
By construction, we can index these edge weights by non-root vertices, writing the weights as $X=(X_v)_{v\neq\emptyset_1}$, where the weight $X_v$ is associated to the edge between $v$ and its parent $p(v)$.
We make the convention that $X_{v0}=0$.

\begin{defn}[Poisson-weighted infinite tree]\lbdefn{PWITdef}
The \emph{Poisson-weighted infinite tree} (PWIT) is the random tree $(\tree^{\sss(1)},X)$ for which $X_{vk}-X_{v(k-1)}$ is exponentially distributed with mean 1, independently for each
$v\in \tree^{\sss(1)}$ and each $k\in\N$.
Equivalently, the weights $(X_{v1},X_{v2},\dotsc)$ are the (ordered) points of a Poisson point process of intensity 1 on $(0,\infty)$, independently for each $v$.
\end{defn}

Motivated by \eqrefPartI{SingleVertexCoupling}, we study FPP on $\tree^{\sss(1)}$ with edge weights $(f_n(X_v))_v$:

\begin{defn}[First passage percolation on the Poisson-weighted infinite tree]\lbdefn{FPPonPWITdef}
For FPP on $\tree^{\sss(1)}$ with edge weights $(f_n(X_v))_v$, let
the FPP edge weight between $v\in\tree^{\sss(1)}\setminus\set{\emptyset_1}$ and $\parent{v}$ be $f_n(X_v)$. The FPP distance from $\emptyset_1$ to $v\in\tree^{\sss(1)}$ is
\begin{equation}\labelPartI{TvDefinition}
T_v = \sum_{k=0}^{\abs{v}-1} f_n(X_{\ancestor{k}{v}})
\end{equation}
and the FPP exploration process $\BP^{\sss(1)}=(\BP^{\sss(1)}_t)_{t\geq 0}$ on $\tree^{\sss(1)}$ is defined by $\BP^{\sss(1)}_t=\set{v\in\tree^{\sss(1)}\colon T_v\leq t}$.
\end{defn}

Note that the FPP edge weights $(f_n(X_{vk}))_{k\in\N}$ are themselves the points of a Poisson point process on $(0,\infty)$, independently for each $v\in\tree^{\sss(1)}$.
The intensity measure of this Poisson point process, which we denote by $\mu_n$, is the image of Lebesgue measure on $(0,\infty)$ under $f_n$.
Since $f_n$ is strictly increasing by assumption, $\mu_n$ has no atoms and we may abbreviate $\mu_n(\ocinterval{a,b})$ as $\mu_n(a,b)$ for simplicity.
Thus $\mu_n$ is characterized by
\begin{equation}\labelPartI{munCharacterization}
\mu_n(a,b) = f_n^{-1}(b) - f_n^{-1}(a),
\qquad
\int_0^\infty h(y) d\mu_n(y) = \int_0^\infty h(f_n(x)) dx,
\end{equation}
for any measurable function $h\colon \cointerval{0,\infty}\to\cointerval{0,\infty}$.

Clearly, and as suggested by the notation, the FPP exploration process $\BP^{\sss(1)}$ is a continuous-time branching process:

\begin{prop}\lbprop{FPPisCTBP}
The process $\BP^{\sss(1)}$ is a continuous time branching process (CTBP), started from a single individual $\emptyset_1$, where the ages at childbearing of an individual form a Poisson point process with intensity $\mu_n$, independently for each individual.
The time $T_v$ is the birth time $T_v=\inf\set{t\geq 0\colon v\in\BP^{\sss(1)}_t}$ of the individual $v\in\tree^{\sss(1)}$.
\end{prop}

Similar to the analysis of the weights of the edges containing vertex $1$, we now introduce a thinning procedure.
Define $M_{\emptyset_1}=1$ and to each other $v\in \tree^{\sss(1)}\setminus\set{\emptyset_1}$ associate a mark $M_v$ chosen independently and uniformly from $[n]$.

\begin{defn}[Thinning]\lbdefn{ThinningBP}
The root $\emptyset_1 \in \tree^{\sss(1)}$ is not thinned, i.e.\ \emph{unthinned}.
The vertex $v\in\tree^{\sss(1)}\setminus\set{\emptyset_1}$ is \emph{thinned} if it has an ancestor $v_0=\ancestor{k}{v}$ (possibly $v$ itself) such that $M_{v_0}=M_w$ for some unthinned vertex $w$ with $T_w<T_{v_0}$.
\end{defn}

Note that whether or not a vertex $v$ is thinned can be assessed recursively in terms of earlier-born vertices\footnote{At least up until the time $t=\sup_{x\geq 0} f_n(x)$ when $\BP^{\sss(1)}_t$ ceases to be finite a.s.
However, before time $t$, the root $\emptyset_1$ has had infinitely many children, a.s., so all available marks have been used and all vertices born after time $t$ are thinned, a.s.} and therefore \refdefn{ThinningBP} is not circular.

Write $\thinnedBP_t^{\sss(1)}$ for the subgraph of $\BP_t^{\sss(1)}$ consisting of unthinned vertices.
If a vertex $v\in\tree^{\sss(1)}$ is thinned, then so are all its descendants, and this implies that $\thinnedBP_t^{\sss(1)}$ is a tree for all $t$.

\begin{defn}\lbdefn{InducedGraph}
Given a subset $\tau\subset\tree^{\sss(1)}$ and marks $M=(M_v \colon v \in \tau)$ with $M_v\in[n]$, define $\pi_M(\tau)$ to be the subgraph of $K_n$ induced by the mapping $\tau\to [n]$, $v\mapsto M_v$.
That is, $\pi_M(\tau)$ has vertex set $\set{M_v\colon v\in\tau}$, with an edge between $M_v$ and $M_{\parent{v}}$ whenever $v,\parent{v}\in\tau$.
\end{defn}
Note that if the marks $(M_v)_{v\in\tau}$ are distinct then $\pi_M(\tau)$ and $\tau$ are isomorphic graphs.

The following theorem establishes a close connection between FPP on $K_n$ and FPP on the PWIT with edge weights $(f_n(X_v))_v$:

\begin{theorem}[Coupling to FPP on PWIT]\lbthm{Coupling1Source}
The law of $(\SWT_t^{\sss(1)})_{t\ge 0}$ is the same as the law of $\Bigl( \pi_M\bigl( \thinnedBP_t^{\sss(1)} \bigr) \Bigr)_{t\ge 0}$.
\end{theorem}

\refthm{Coupling1Source} is based on an explicit coupling between the edge weights $(Y_e^{\sss(K_n)})_e$ on $K_n$ and $(X_v)_v$ on $\tree^{\sss(1)}$.
A general form of those couplings and the proof of \refthm{Coupling1Source} are given in \refsect{Coupling}.

\subsection{Relation to invasion percolation on the PWIT}\lbsubsect{RelToIP}
Under our scaling assumptions, FPP on the PWIT is closely related to \emph{invasion percolation} (IP) on the PWIT which is defined as follows.
Set $\IP^{\sss(1)}(0)$ to be the subgraph consisting of $\emptyset_1$ only.
For $k\in\N$, form $\IP^{\sss(1)}(k)$ inductively by adjoining to $\IP^{\sss(1)}(k-1)$ the boundary vertex $v\in\boundary\IP^{\sss(1)}(k-1)$ of minimal weight.
We note that, since we consider only the relative ordering of the various edge weights, we can use either the PWIT edge weights $(X_v)_v$ or the FPP edge weights $(f_n(X_v))_v$.

Write $\IP^{\sss(1)}(\infty)=\bigunion_{k=1}^\infty\IP^{\sss(1)}(k)$ for the limiting subgraph.
We remark that $\IP^{\sss(1)}(\infty)$ is a strict subgraph of $\tree^{\sss(1)}$ a.s.\ (in contrast to FPP, which eventually explores every vertex).
Indeed, define
\begin{equation}\labelPartI{MDefinition}
M^{\sss(1)}=\sup\set{X_{v}\colon v\in\IP^{\sss(1)}(\infty)\setminus \set{\emptyset_1}},
\end{equation}
the largest weight of an invaded edge.
Then $\P(M^{\sss(1)}<x)$ is the survival probability of a Poisson Galton--Watson branching process with mean $x$, as in \refthm{IPWeightForFPP}.
Indeed, the event $\set{M^{\sss(1)}<x}$ is the event that, if we remove from $\tree^{\sss(1)}$ all edges of weight $X_{vk}>x$, the component of $\emptyset_1$ in the resulting subgraph is infinite.
We remark that, a.s., the supremum in \eqrefPartI{MDefinition} is attained, and the unique edge of weight $M^{\sss(1)}$ is invaded after a finite number of steps (see \refprop{IPstructure} below).

The value $x=1$ acts as a \emph{critical value} for the PWIT.
Indeed, if we remove all edges of weight $X_{vk}> x$, then the subtree containing the roots is a branching process with $\Poi(x)$ offspring distribution.
Hence for $x\leq 1$ the tree is finite a.s., while for $x>1$ the tree is infinite with positive probability.
As a result, IP on the PWIT will have to accept edges of weight $X_{vk}>1$ infinitely often, and we have $M^{\sss (1)}>1$ a.s.

We next explain the connection between FPP and IP on the PWIT, under our scaling assumptions for the edge weights.
We emphasize that FPP depends on $n$ via the edge weights $(f_n(X_v))_v$, whereas IP is independent of $n$.

By \refcond{scalingfn}, we can approximate
	\begin{equation}\labelPartI{ApproximateWeights}
	f_n(X_v)\approx f_n(1) (X_v)^{s_n}.
	\end{equation}
(Note that for the FPP exploration process, the only effect of multiplying by $f_n(1)$ is to rescale time.)
Since the $\ell^s$ norm converges towards the $\ell^\infty$ norm (recall \eqrefPartI{PowerToMax}), we see that, when $s_n\to\infty$, a small edge weight is almost negligible when added to a larger edge weight.
Therefore, the sequence of edges added under the FPP dynamics with weights $f_n(X_v)\approx f_n(1) (X_v)^{s_n}$ can be well approximated by adding the boundary edge having the smallest weight, that is, by the IP dynamics.
Moreover, the time for the FPP exploration is dominated by the time $f_n(M^{\sss(1)})$ spent exploring the edge of \emph{largest} weight, and until this edge has been explored only a finite number of other edges will be explored.

To formalize this discussion, consider the smallest-weight tree $\SWT^{\sss(1)}$ on $K_n$ started from vertex $1$, as defined in \eqrefPartI{OneSourceSWT}.
Write $\tau_k=\inf\shortset{t\geq 0\colon \shortabs{E(\SWT^{\sss(1)}_t)}=k}$ for the time when the $k^\th$ edge is added to $\SWT^{\sss(1)}$.
Then we have the following local weak convergence result, formalizing \refthm{IPLocConvForFPP_exp} and \refthm{IPLocConvForFPP_gen}:

\begin{theorem}[Coupling to IP on the PWIT]
\lbthm{CoupIP-PWIT}
Suppose $\lim_{n\to\infty} f_n(x+\delta)/f_n(x)=\infty$ for each $x,\delta>0$.
Then the smallest-weight tree $\SWT^{\sss(1)}$ on $K_n$ can be coupled to invasion percolation $\IP^{\sss(1)}$ on one copy of the PWIT such that, for any fixed $m\in\N$,
\begin{equation}
\labelPartI{locWeakConvSWTIP}
\prob \big(\SWT^{\sss(1)}_{\tau_k}=\pi_M(\IP^{\sss(1)}(k)) \text{ for all }k \le m\big)
	=1-o(1).
\end{equation}
\end{theorem}

\refthm{CoupIP-PWIT} is proved in \refsubsect{ProofCoupIPPWIT}.
The convergence in \refthm{CoupIP-PWIT} is the local weak convergence in the sense of Benjamini and Schramm~\cite{BenSch01} for appropriately chosen metrics.
\refthm{IPLocConvForFPP_gen} follows from \refthm{CoupIP-PWIT} because given $x,\delta>0$ we can apply \refcond{LowerBoundfn} (with $1-\deltaCondition=x$, if $x<1$) or \refcond{boundfn}~\refitem{boundfnR} (with $R=x+\delta$, if $x+\delta>1$) to find a value $\epsilon>0$ such that for sufficiently large $n$,
\begin{equation}
\log f_n(x+\delta) - \log f_n(x) = \int_x^{x + \delta} t \frac{d}{dt} \log f_n(t)\, dt/t \ge \epsilon s_n \log(1+ \delta/x).
\end{equation}
Applying the exponential function on both sides, $f_n(x+\delta)/f_n(x) \to \infty$ follows from $s_n\rightarrow \infty$.

The heuristic comparison between FPP and IP ceases to be valid when a smaller edge weight $f_n(X_v\wedge X_w)$ is no longer negligible when added to a larger edge weight $f_n(X_v \vee X_w)$.
By \eqrefPartI{ApproximateWeights}, this requires $\abs{X_v-X_w}=\Theta(1/s_n)$.
By our discussion of the critical value, only edge weights $X_{v}\approx 1$ will be relevant to the large-scale behavior.
It follows that once edge weights belonging to a \emph{critical window} $[1-\Theta(1/s_n),1+\Theta(1/s_n)]$ become numerous, the heuristic fails and the connection to IP on the PWIT ceases to hold.
Since we are mainly interested in the case that $1/s_n \gg n^{-1/3}$ (see in particular \cite{EckGooHofNar14b}), the critical window observed here is wider than the critical window for the Erd\H{o}s-R\'enyi random graph; cf.\ \cite{Ald97}.

For IP on the PWIT, the weight of the maximal weight edge that IP uses after time $k^2t$ scales as $1+\mathcal{U}_t/k,$ where $(\mathcal{U}_t)_{t\geq 0}$ is a limiting stochastic process (see \cite[Proposition~3.3 and Theorem~1.6]{AGdHS2008} and the remarks following \cite[Theorem 31]{AddGriKan12}).
In particular, these weights become of the order $1+\Theta(1/s_n)$ when the size of the IP cluster is $\Theta(s_n^2)$.
This suggests that the maximal size of the smallest-weight tree that allows its dynamics to be coupled to IP on the PWIT is $o(s_n^2)$.
However, we do not need and will not prove such a strong result.

\subsection{Exploration from two sources}\lbsubsect{explorationfromtwosources}

So far we have studied the smallest-weight tree $\SWT^{\sss(1)}$ from one vertex and its coupling to a suitably thinned version of a CTBP starting from one root.
To study the optimal path between vertices $1$ and $2$,
a standard approach would be to place sources of fluid on both vertices and wait for the two smallest-weight trees to merge.
Appealing to the coupling in \refthm{Coupling1Source}, equally we could study the evolution of two independent CTBPs $\BP^{\sss(1)}$ and $\BP^{\sss(2)}$ with original ancestors $\emptyset_1$ and $\emptyset_2$.
For this method it is important that the two trees grow in a similar fashion.
However, the heuristics in \refsubsect{RelToIP} show that $\BP^{\sss(j)}$ spends a considerable amount of time, namely $f_n(M^{\sss(j)})\gg f_n(1)$, waiting for a single edge to be explored, during which time only finitely many other edges are explored.
Since $M^{\sss(1)}\neq M^{\sss(2)}$ a.s., our scaling assumptions on $f_n$ mean that the times $f_n(M^{\sss(1)}),f_n(M^{\sss(2)})$ will be quite different.

For this reason, we will not grow the two CTBPs at the same speed.
When one of them becomes large enough (what this means precisely will be explained later), it has to wait for the other one to catch up. We call this procedure {\em freezing.}

To formalize this, let $\twoPWITs$ be the disjoint union of two independent copies $(\tree^{\sss(j)},X^{\sss(j)})$, $j \in \set{1,2}$, of the PWIT.
We shall assume that the copies $\tree^{\sss(j)}$ are vertex-disjoint, with roots $\emptyset_j$, so that we can unambiguously write $X_v$ instead of $X^{\sss(j)}_v$ for $v\in \tree^{\sss(j)}$, $v\neq\emptyset_j$.
The notation introduced for $\tree^{\sss(1)}$ is used verbatim on $\tree$.
For example, for any subset $\tau \subseteq \tree$, we write $\boundary \tau=\set{v \not\in \tau\colon \parent{v} \in \tau}$ for the boundary vertices of $\tau$.

The FPP process on $\twoPWITs$ with edge weights $(f_n(X_v))_v$ starting from $\varnothing_1$ and $\varnothing_2$ is equivalent to the union $\BP=\BP^{\sss(1)} \cup \BP^{\sss(2)}$ of two CTBPs.
Let $T_\fr^{\sss(j)}$ be a stopping time with respect to the filtration induced by $\BP^{\sss(j)}$, $j \in\set{1,2}$.
We call $T_\fr^{\sss(1)}$ and $T_\fr^{\sss(2)}$ \emph{freezing times} and run $\BP$ until $T_\fr^{\sss(1)}\wedge T_\fr^{\sss(2)}$, the time when one of the two CTBPs is large enough (see \refdefn{Freezing} for the precise definition of what large enough means).
Then we freeze the larger CTBP and allow the smaller one to evolve normally until it is large enough, at time $T_\fr^{\sss(1)}\vee T_\fr^{\sss(2)}$.
At this time, which we call the \emph{unfreezing} time $T_\unfr=T_\fr^{\sss(1)}\vee T_\fr^{\sss(2)}$, both CTBPs resume their usual evolution.
We denote by $R_j(t)$ the on-off processes describing this behavior: that is, for $j \in \set{1,2}$,
	\begin{equation}
	\labelPartI{OnOff}
	R_j(t)=(t\wedge T_\fr^{\sss(j)}) + ((t-T_\unfr)\vee 0)
  .
	\end{equation}
The version of $\BP=(\BP_t)_{t\ge 0}$ including freezing is then given by
\begin{equation}\labelPartI{clusterDefinition}
\cluster_t=\bigunion_{j=1}^2 \cluster_t^{\sss(j)},  \quad  \cluster_t^{\sss(j)}=\set{v\in\tree^{\sss(j)} \colon T_v \leq R_j(t)} = \BP_{R_j(t)}^{\sss(j)} \quad \text{for all }t\ge 0.
\end{equation}
As with $\tree$, we can consider $\cluster_t$ to be the union of two trees by placing an edge between each non-root vertex $v\notin\set{\emptyset_1,\emptyset_2}$ and its parent.
We denote by $T_v^\cluster=\inf\set{t\ge 0 \colon  v\in\cluster_t}$ the arrival time of the individual $v\in\tree$ in $\cluster=(\cluster_t)_{t\ge 0}$.
Using the left-continuous inverse of $R_j(t)$, defined by
\begin{equation}\labelPartI{OnOffInverse}
R_j^{-1}(y) = \inf\set{t\geq 0\colon R_j(t)\geq y}
=
\begin{cases}
t & \text{if }t\leq T_\fr^{\sss(j)},\\
T_\unfr-T_\fr^{\sss(j)}+t & \text{if }t>T_\fr^{\sss(j)},
\end{cases}
\end{equation}
we obtain $T_v^\cluster=R_j^{-1}(T_v)$ for $v\in\tree^{\sss(j)}$.

\subsection{Freezing a CTBP}
\lbsubsect{FreezingDisc}

For very fine results about the branching process and the freezing times, we strengthen \refcond{boundfn} to the following condition:

\begin{cond}[Density bound for large weights]
\lbcond{boundfnExtended}
There exist $\epsilonConditiona>0$ and $n_1 \in \N$ such that
	\eqn{
	\labelPartI{fn-bound}
	x \frac{d}{dx} \log f_n(x)\geq \epsilonConditiona s_n \qquad\text{for every }x\geq 1, n \ge n_1.
	}
\end{cond}

As explained, the purpose of the freezing is to guarantee a comparable growth of the two CTBPs.
One requirement is therefore that the edge of weight $f_n(M^{\sss(j)})$ is explored before freezing, for each $j=1,2$ (and they are instantaneously unfrozen after the last of these times).
The second requirement is that the two CTBPs exhibit typical branching process dynamics.
To make this precise, recall that a typical branching process grows exponentially where the growth rate is given by its Matlthusian parameter $\lambda_n$:
Writing $\hat{\mu}_n(\lambda)=\int\e^{-\lambda y}d\mu_n(y)$ for the Laplace transform of the intensity measure $\mu_n$,
$\lambda_n>0$ is the unique solution to
	\begin{equation}\labelPartI{lambdanDefn}
	\hat{\mu}_n(\lambda_n)=1.
	\end{equation}
Asymptotically, $\lambda_n$ scales like $1/f_n(1)$:

\begin{lemma}\lblemma{lambdanAsymp}
Suppose Conditions~\refPartI{cond:scalingfn}, \refPartI{cond:LowerBoundfn} and \refPartI{cond:boundfnExtended} hold for a positive sequence $(s_n)_n$ with $s_n \to \infty$. Then $\lambda_n f_n(1) \to \e^{-\gamma}$, where $\gamma$ is Euler's constant.
\end{lemma}
\reflemma{lambdanAsymp} is proved in \refsect{densityBoundsfn}.
The same reasoning is used to prove a more general statement in \refother{\refthmPartII{ConvRW}} (see \refother{Sections~\refstarPartII{ss:ConvRWThm} and \refstarPartII{ss:ConvRWPf}}) but we include the proof here to avoid circularity.

Hence, we expect $\BP^{\sss(j)}=(\BP_t^{\sss(j)})_{t\ge 0}$ to grow exponentially at rate $\lambda_n$ for large times $t$.
However, initially, $\BP^{\sss(j)}$ does not grow in this way.
The reason is that the exponential growth typical of a branching process arises primarily from rare individuals that have an unusually large number of offspring in a very short time.
Formally, for $v \in \tree^{\sss (1)}$, write $\BP^{\sss(v)}$ for the branching process of descendants of $v$, re-rooted and time-shifted to start at $t=0$.
That is,
\begin{equation}\labelPartI{BPvDefinition}
\BP^{\sss(v)}_t=\set{w\in\tree^{\sss(1)}\colon vw\in\BP^{\sss(1)}_{T_v+t}}.
\end{equation}
For instance, $\BP^{\sss (1)}=\BP^{\sss (\emptyset_1)}$, and  $(\BP^{\sss(v)})_{\parent{v}=\emptyset_1}$ are independent of each other and of $(T_v)_{\parent{v}=\emptyset_1}$.

\begin{defn}\lbdefn{DefRLucky}
Given $R<\infty$, a vertex $v\in \tree^{\sss(j)}$ is \emph{$R$-lucky} if $\bigabs{\BP^{\sss(v)}_{f_n(1)}}\geq Rs_n^2$.
\end{defn}

That is, an $R$-lucky vertex has $Rs_n^2$ descendants by the time it reaches age $f_n(1)$. The following proposition states that this happens with probability at least a constant times $1/s_n$:

\begin{prop}\lbprop{LuckyProb}
Suppose Conditions~\refPartI{cond:scalingfn}, \refPartI{cond:LowerBoundfn} and \refPartI{cond:boundfnExtended} hold for a positive sequence $(s_n)_n$ with $s_n \to \infty$.
Fix $R\in (0,\infty)$.
There exists $\delta>0$ such that for every $r \in \ocinterval{0,R}$ there is some $n_0 \in \N$ such that $\P\left( v\text{ is $r$-lucky} \right)\geq \delta/(s_n \sqrt{r})$ for all $v \in \tree$ and all $n\ge n_0$.
\end{prop}

\refprop{LuckyProb} is proved in \refsubsect{VolClusterFr}.

Once an $R$-lucky vertex $v$ is born (for some $R>0$), another $R$-lucky vertex (for some $R>0$) is likely to be born soon thereafter.
Indeed, \refcond{scalingfn} implies that between ages $f_n(1)$ and $2f_n(1)$, the number of new children of $v$ will be Poisson with mean of order $1/s_n$.
The same is true for the order $s_n^2$ initial descendants of $v$, so that a total of order $s_n$ children is expected to be born during this time.
Among these, of order $1$ can be expected to repeat the unlikely event performed by $v$ and thereby perpetuate the growth.

We therefore expect that $\BP^{\sss(j)}$ will exhibit typical branching dynamics, with exponential growth on the time scale $f_n(1)$, only once $\BP^{\sss(j)}$ is large enough so that the intensity of new births, in the time scale $f_n(1)$, is at least $s_n$.

This motivates the following definition of the freezing times:

\begin{defn}[Freezing]
\lbdefn{Freezing}
Define, for $j=1,2$, the \emph{freezing times}
	\begin{equation}\labelPartI{TfrDefn}
	T_\fr^{\sss(j)} = \inf\bigg\{ t\geq 0\colon \sum_{v\in\BP_t^{\sss(j)}} \int_{t-T_v}^\infty \e^{-\lambda_n \left(y-(t-T_v)\right)}
	d\mu_n(y) \geq s_n \bigg\},
	\end{equation}
and the \emph{unfreezing time} $T_\unfr=T_\fr^{\sss(1)} \vee T_\fr^{\sss(2)}$.
The \emph{frozen cluster} is given by
\begin{equation}
\cluster_\fr=\cluster_{T_\unfr}=\cluster_\fr^{\sss(1)}\union\cluster_\fr^{\sss(2)} \quad \text{where} \quad \cluster_\fr^{\sss(j)}=\cluster_{T_\fr^{(j)}}^{\sss(j)}.
\end{equation}
\end{defn}
The random variable $\int_{t-T_v}^\infty \e^{-\lambda_n \left(y-(t-T_v)\right)} d\mu_n(y)$ represents the expected number of future offspring of vertex $v\in\BP_t^{\sss(j)}$, exponentially time-discounted at rate $\lambda_n$.
Recall from \eqrefPartI{OnOff} that $R_j(t)=(t\wedge T_\fr^{\sss(j)}) + ((t-T_\unfr)\vee 0)$.
Thus, each CTBP evolves at rate $1$ until its expected number of future offspring, exponentially time-discounted at rate $\lambda_n$, first exceeds $s_n$.
At that time the configuration is ``frozen'' and ceases to evolve until both sides have been frozen.
The two sides, which are now of a comparably large size, are then simultaneously unfrozen and thereafter evolve at rate 1.
Henceforth we will always assume this choice of $T_\fr^{\sss(1)},T_\fr^{\sss(2)}$.

We next investigate the asymptotics of the freezing times:
\begin{theorem}[Scaling of the freezing times]
\lbthm{TfrScaling}
Suppose Conditions~\refPartI{cond:scalingfn}, \refPartI{cond:LowerBoundfn} and \refPartI{cond:boundfnExtended} hold for a positive sequence $(s_n)_n$ with $s_n \to \infty$.
The freezing times satisfy $f_n^{-1}(T_\fr^{\sss(j)}) \convp M^{\sss (j)}$ for $j=1,2$.
\end{theorem}

\refthm{TfrScaling}, which is proved in \refsubsect{TfrScalingPf}, confirms that the
two CTBPs $\BP^{\sss(1)}$ and $\BP^{\sss(2)}$ are indeed large enough relatively soon after exploring the edges of weight $f_n(M^{\sss(1)})$ and $f_n(M^{\sss(2)})$, respectively.

The next result states that soon after unfreezing, an $R$-lucky vertex is born.

\begin{lemma}\lblemma{RluckyBornSoon}
Suppose Conditions~\refPartI{cond:scalingfn}, \refPartI{cond:LowerBoundfn} and \refPartI{cond:boundfnExtended} hold for a positive sequence $(s_n)_n$ with $s_n \to \infty$.
Fix $R<\infty$ and let $v_{j,R}$ denote the first $R$-lucky vertex in $\cluster^{\sss(j)}$ born after time $T_\unfr$.
Then $T_{v_{j,R}}^\cluster=T_\unfr+O_\P(f_n(1))$.
\end{lemma}

\reflemma{RluckyBornSoon} is proved in \refsubsect{RLuckyBirthProof}.

\subsection{Decomposing FPP into IP and branching dynamics}\lbsubsect{DecomposeFPP}

Similarly to the definition of the CTBPs with freezing $\cluster$, we can study the smallest-weight trees on $K_n$ with freezing.
However, in contrast to the definition of $\cluster$, there is competition between the two clusters and we have to actively forbid their merger.
We do not need or prove any statements about the smallest-weight trees from two sources in this paper, and, therefore, give only the following informal definition. We refer to Part II for extensive results in this direction.

The two smallest-weight trees with freezing on $K_n$ started from vertices $1$ and $2$ are two disjoint trees $\cS^{\sss(j)}=(\cS_t^{\sss(j)})_{t\ge 0}$, $j\in \set{1,2}$, on $K_n$ with root $1$ and $2$, respectively.
Initially $\cS^{\sss(j)}$ contains only vertex $j$ which can be viewed as the source of a fluid that differs from the fluid originating at $j'$ with $\set{j,j'}=\set{1,2}$.
When a vertex becomes wet, it is added to the tree of the corresponding fluid and all adjacent edges between the new vertex and the other tree can no longer transport fluid.
Moreover, fluid originating from $j$ cannot travel between times $T_\fr^{\sss(j)}$ and $T_\unfr$.
We write $\cS_t=\cS_t^{\sss(1)} \cup \cS_t^{\sss(2)}$ for all $t \ge 0$ \and call $\cS_t$ the smallest-weight tree from two sources.

Theorems~\refPartI{t:Coupling1Source} and \refPartI{t:TfrScaling} imply that the smallest-weight trees without freezing, i.e., the choice $R_1(t)=R_2(t)=t$, would behave almost like the asymmetric choices $R_1(t)=t$, $R_2(t)=0$ or $R_1(t)=0$, $R_2(t)=t$.
Write $\cS^{\sss({\rm id})}=(\cS_t^{\sss({\rm id})})_{t\ge 0}$ for the smallest-weight tree starting from vertices $1$ and $2$, constructed as $\cS=(\cS_t)_{t\ge 0}$ but with $R_j$ replaced by the identity map, for $j\in \set{1,2}$ (that is, $\cS^{\sss(\text{id})}$ is the smallest-weight tree from two sources without freezing).
Denote by $\cS^{\sss({\rm id},1)}$ the connected component of $\cS^{\sss({\rm id})}$ containing vertex $1$.
Our results suggest that
\begin{equation}\labelPartI{ImportanceOfFreezing}
\frac{\abs{\cS_{\infty}^{\sss ({\rm id},1)}}}{n} \convd \theta_{{\rm id}} \qquad \text{and } \qquad \frac{\abs{\cS_{\infty}^{\sss (1)}}}{n} \convd \theta_\fr,
\end{equation}
where $\theta_{{\rm id}}$ and $\theta_\fr$ are random variables with $\P(\theta_{{\rm id}}=1)=\P(\theta_{{\rm id}}=0)=1/2$ and $\P(\theta_\fr \in (0,1))=1$. This means that freezing is necessary to guarantee that both trees are asymptotically of comparable size.

Similarly to \refthm{Coupling1Source}, we can couple $\cluster_t$ and $\cS_t$.
To this end, we introduce a thinning procedure for $\cluster=(\cluster_t)_{t\ge 0}$ analogous to \refdefn{ThinningBP}.
Define $M_{\emptyset_j}=j$, for $j=1,2$.
To each other $v\in \twoPWITs\setminus\set{\emptyset_1,\emptyset_2}$, we associate a mark $M_v$ chosen uniformly and independently from $[n]$.

\begin{defn}\lbdefn{Thinningcluster}
The vertex $v\in\tree\setminus\set{\emptyset_1,\emptyset_2}$ is \emph{thinned} if it has an ancestor $v_0=\ancestor{k}{v}$ (possibly $v$ itself) such that $M_{v_0}=M_w$ for some unthinned vertex $w$ with $T^\cluster_w<T^\cluster_{v_0}$.
\end{defn}
As in the remarks below \refdefn{ThinningBP}, this definition is not circular, and we write $\thinnedcluster_t$ for the subgraph of $\cluster_t$ consisting of unthinned vertices.

Recall the subgraph $\pi_M(\tau)$ of $K_n$ introduced in \refdefn{InducedGraph}, which we extend to the case where $\tau\subset\tree$.

\begin{theorem}\lbthm{CouplingFPP}
There exists a coupling such that $\cS_t^{\sss(1)}\cup \cS_t^{\sss(2)} =\pi_M(\thinnedcluster_t)$ for all $t\ge 0$ almost surely.
\end{theorem}

\refthm{CouplingFPP} is formalized and proved as \refother{\refthmPartII{CouplingFPP}}.

We will not use the result of \refthm{CouplingFPP} in this paper.
Nevertheless, \refthm{CouplingFPP} motivates our study of $\cluster$ and freezing times because it relates FPP on the complete graph ($n$-independent dynamics run on an $n$-dependent weighted graph) with an exploration defined in terms of a pair of Poisson-weighted infinite trees ($n$-dependent dynamics run on an $n$-independent weighted graph).
By analyzing the dynamics of $\cluster$ when $n$ and $s_n$ are large, we obtain a fruitful \emph{dual picture}:
a \emph{static} approximation by IP, valid when the number of explored vertices is small and independent of $n$; followed by a \emph{dynamic} rescaled branching process approximation, valid when the number of explored vertices is large, that is also essentially independent of $n$.

An important goal of this paper is to understand the dynamics of the IP part which is formalized as the frozen cluster.
Our main results are collected in the following theorem:

\begin{theorem}[Properties of the frozen cluster]\lbthm{FrozenCluster}Suppose Conditions~\refPartI{cond:scalingfn}, \refPartI{cond:LowerBoundfn} and \refPartI{cond:boundfnExtended} hold for a positive sequence $(s_n)_n$ with $s_n \to \infty$.
\begin{enumerate}
\item \lbitem{FrozenVolume}
The volume $\abs{\cluster_\fr}$ of the frozen cluster is $O_\P(s_n^2)$.
\item \lbitem{FrozenDiameter}
The diameter $\max\set{\abs{v}\colon v\in\cluster_\fr}$ of the frozen cluster is $O_\P(s_n)$.
\end{enumerate}
\end{theorem}
\refthm{FrozenCluster}~\refitem{FrozenVolume} and \refitem{FrozenDiameter} are proved in Sections~\refPartI{ss:VolClusterFr} and \refPartI{ss:FrozenDiameterPf}, respectively.
\refthm{FrozenCluster} will allow us to ignore the elements coming from the frozen cluster in analysis of the smallest-weight path.
For instance, part \refitem{FrozenDiameter} will be used to show that path lengths within the frozen cluster are negligible compared to overall path lengths. We believe that the diameter $\max\set{\abs{v}\colon v\in\cluster_\fr}$ of the frozen cluster really is of order $s_n$, but we will not need this and therefore also not prove this. Since $s_n\rightarrow\infty$, this intuitively explained `long paths' in our title. This analysis will be carried out in \cite{EckGooHofNar14b} where we are interested in global properties of the FPP process. There also the title will be substantiated, since it will be shown that
$H_n$ is of order $s_n\log(n/s_n^3)$ when $s_n=o(n^{1/3})$.

\subsection{Discussion of our results}
\lbsect{DiscExt}

In this section we briefly discuss our results and state open problems. For a more detailed discussion of the results in this paper and in our companion paper \cite{EckGooHofNar14b}, as well as an extensive discussion of the relations to the literature, we refer to  \refother{\refsectPartII{DiscExt}}.

First passage percolation (FPP) on the complete graph is closely approximated by invasion percolation (IP) on the Poisson-weighted infinite tree (PWIT), studied in \cite{AddGriKan12}, whenever $s_n\to\infty$.
See \refthm{IPLocConvForFPP_gen} and the discussion in \refsubsect{PWITFPPIP}.
However, this relationship is a {\em local} one, and the scaling of $s_n$ relative to $n$ controls whether the two objects are {\em globally} comparable.
\refthm{IPWeightForFPP} shows that the weights are globally comparable provided $s_n/\log\log n\to\infty$.
For the hopcount, the appropriate comparison is to the minimal spanning tree (MST) on the complete graph, obtained from running IP with a simple no-loops constraint.
Path lengths in the MST scale as $n^{1/3}$ (see \cite{AddBroRee08} and \cite{AddBroGolMie12pre}).
We conjecture that, for $s_n^3/n \to \infty$, FPP on the complete graph is in the same universality class as IP.
It would be of great interest to make this connection precise by showing, for example, that $H_n/n^{1/3}$ converges in distribution, and that the scaling limit of $H_n$ is the same as the scaling limit of the graph distance between two uniform vertices in the MST.

The local graph convergence from \refthm{IPLocConvForFPP_gen} and weight convergence from \refthm{IPWeightForFPP} are the first two in a hierarchy of possible comparisons between FPP and the MST.
Strengthening the previous statement about the scaling limit of $H_n$, we can ask whether the optimal path between vertices $i,j\in [n]$ equals (under a suitable coupling) the unique path in the MST from $i$ to $j$; whether the union of the optimal paths\footnote{This union of optimal paths will be the smallest-weight tree $\SWT^{(i)}_t$ from \refsect{IPpartFPP} in the limit $t\to\infty$.} from vertex $i$ to every other vertex $j\neq i$ equals the entire MST; and whether these unions agree simultaneously for every $i\in [n]$.
Assuming hypotheses similar to Conditions~\refPartI{cond:scalingfn}--\refPartI{cond:boundfn}, it would be of interest to know how $s_n$ must grow relative to $n$ in order for each of these events to occur.

\section{\texorpdfstring{Growth and density bounds for $f_n$ and $\mu_n$}{Growth and density bounds for f\_n and mu\_n}}\lbsect{densityBoundsfn}

Throughout this section, we assume Conditions~\refPartI{cond:LowerBoundfn} and \refPartI{cond:boundfnExtended}.
Further assumptions will be stated explicitly.
We will reserve the notation $\epsilonCondition, \deltaCondition$ for some fixed choice of the constants in Conditions~\refPartI{cond:LowerBoundfn} and \refPartI{cond:boundfnExtended}, with $\epsilonCondition$ chosen small enough to satisfy both conditions.

The aim of the section is to explore the key implications of Conditions~\refPartI{cond:LowerBoundfn} and \refPartI{cond:boundfnExtended} on $f_n$ and on the intensity measure $\mu_n$.

\begin{lemma}\lblemma{ExtendedImpliesWeak}
There exists $n_0 \in \N$ such that
\begin{equation}%\labelPartI{}
f_n(x)\leq \left( \frac{x}{x'} \right)^{\epsilonCondition s_n} f_n(x') \qquad\text{whenever }1-\deltaCondition\leq x\leq x', n \ge n_0
.
\end{equation}
\end{lemma}

\begin{proof}
Divide \eqrefPartI{BoundfnSmall} or \eqrefPartI{fn-bound} by $x$ and integrate between $x$ and $x'$ to obtain $\log{f_n(x')}-\log{f_n(x)}\geq \epsilonCondition s_n \left( \log x' - \log x \right)$ whenever $1-\deltaCondition\leq x\leq x'$, $n \ge n_0$, as claimed.
\end{proof}

We call \refcond{boundfnExtended} a density bound because it implies the following lemma:
\begin{lemma}\lblemma{munDensityBound}
For $n$ sufficiently large, on the interval $(f_n(1-\deltaCondition),\infty)$, the measure $\mu_n$ is absolutely continuous with respect to Lebesgue measure and
\begin{equation}%\labelPartI{}
\indicator{y>f_n(1-\deltaCondition)} d\mu_n(y) \leq \frac{1}{\epsilonCondition s_n} \frac{f_n^{-1}(y)}{y} dy.
\end{equation}
\end{lemma}
\begin{proof}
By Conditions~\refPartI{cond:LowerBoundfn} and \refPartI{cond:boundfnExtended}, $f_n$ is strictly increasing on $(1-\deltaCondition,\infty)$, so $y=f_n(\mu_n(0,y))$ for $y>f_n(1-\deltaCondition)$.
Differentiating and again applying Conditions~\refPartI{cond:LowerBoundfn} and \refPartI{cond:boundfnExtended}, we get
\begin{equation*}%\labelPartI{}
1=f'_n(\mu_n(0,y)) \frac{d}{dy}\mu_n(0,y) \geq \epsilonCondition s_n \frac{f_n(\mu_n(0,y))}{\mu_n(0,y)} \frac{d}{dy}\mu_n(0,y) , \qquad y>f_n(1-\deltaCondition).
\qedhere
\end{equation*}
\end{proof}

\begin{lemma}\lblemma{munDensityBounded}
For $n$ sufficiently large, the density of $\mu_n$ with respect to Lebesgue measure is at most $1/(\epsilonCondition s_n f_n(1))$ on the interval $(f_n(1),\infty)$.
\end{lemma}
\begin{proof}
From \reflemma{ExtendedImpliesWeak} it follows immediately that $f_n^{-1}(y)\leq (y/f_n(1))^{1/\epsilonCondition s_n}\leq y/f_n(1)$  for all $y> f_n(1)$ and sufficiently large $n$.
The result now follows from \reflemma{munDensityBound}.
\end{proof}

In the lemma below, the notation $\mu(t+dy)$ denotes the translation of the measure $\mu$ by $t$.
\begin{lemma}\lblemma{BoundOnContribution}
Given $\epsilon, \bar{\epsilon}>0$, there exist $n_0 \in \N$ and $K<\infty$ such that, for all $n \ge n_0$ and $t\geq 0$,
\begin{equation}%\labelPartI{}
\int \e^{-\epsilon y/f_n(1)}\indicator{y\geq Kf_n(1)} \mu_n(t+dy)\leq \bar{\epsilon}/s_n.
\end{equation}
\end{lemma}
\begin{proof}
By \reflemma{munDensityBounded}, for large $n$, the density of $\mu_n$ with respect to Lebesgue measure is bounded from above by $1/(\epsilonCondition s_n f_n(1))$ on $(f_n(1),\infty)$.
Hence, for $K>1$,
\begin{equation*}
\int \e^{-\epsilon y/f_n(1)}\indicator{y\geq Kf_n(1)} \mu_n(t+dy)\le  \int_t^{\infty}  \e^{-\epsilon (y-t) /f_n(1)}\indicator{y-t\geq Kf_n(1)}  \frac{dy}{\epsilonCondition s_n f_n(1)}
= \frac{\e^{-\epsilon K}}{\epsilonCondition s_n \epsilon}.
\qedhere
\end{equation*}
\end{proof}

We are now in the position to prove \reflemma{lambdanAsymp}.
Recall from the discussion around \eqrefPartI{lambdanDefn} that $\hat{\mu}_n(\lambda)=\int\e^{-\lambda y}d\mu_n(y)$ denotes the Laplace transform of $\mu_n$.

\begin{proof}[Proof of \reflemma{lambdanAsymp}]
We begin by proving
\begin{equation}\labelPartI{munhatUpTo1oversn}
\hat{\mu}_n\left(\frac{a\e^{-\gamma}}{f_n(1)}\right)
= 1 - \frac{\log a}{s_n} + o(1/s_n).
\end{equation}
Recalling \eqrefPartI{munCharacterization}, we have
\begin{equation}\labelPartI{munhattau}
\hat{\mu}_n(\lambda)=\int_0^\infty \e^{-\lambda f_n(\tilde{x})}d\tilde{x}.
\end{equation}
Write $\tilde{f}_n(\tilde{x})=f_n(\tilde{x})\indicator{\tilde{x}\geq 1-\delta_0}$.
Then
\begin{equation}%\labelPartI{}
	\hat{\mu}_n(\lambda)=\int_0^{\infty} \e^{-\lambda \tilde{f}_n(\tilde{x})} \, d\tilde{x}
	- \int_0^{1-\delta_0} 	\big(1-\e^{-\lambda f_n(\tilde{x})}\big) \, d\tilde{x}
\end{equation}
and take $\lambda=a\e^{-\gamma}/f_n(1)$ to estimate $\int_0^{1-\delta_0} (1-\e^{-a\e^{-\gamma}f_n(\tilde{x})/f_n(1)})\,d\tilde{x} = O(f_n(1-\delta_0)/f_n(1))=o(1/s_n)$ by \reflemma{ExtendedImpliesWeak}.
Hence, for the purposes of proving \eqrefPartI{munhatUpTo1oversn}, it is no loss of generality to assume that $f_n(x^{1/s_n})\leq f_n(1)x^{\epsilonCondition}$ for all $x\leq 1$.

Inspired by the example $Y_e\equalsd E^{s_n}$ in \eqrefPartI{fnEsnCase}, where
$f_n(\tilde{x})=f_n(1)\tilde{x}^{s_n}$,  we compute
\begin{equation}\labelPartI{PowerOfExpIntegral}
\int_0^\infty \e^{-\lambda f_n(1)\tilde{x}^{s_n}} d\tilde{x}
= \int_0^\infty \frac{1}{s_n} x^{1/s_n - 1} \e^{-\lambda f_n(1) x} dx
= \left( \frac{\Gamma(1+1/s_n)^{s_n}}{\lambda f_n(1)} \right)^{1/s_n}.
\end{equation}
In particular, setting $\lambda=a\Gamma(1+1/s_n)^{s_n}/f_n(1)$ gives $\int_0^\infty \exp\left( -a\Gamma(1+1/s_n)^{s_n}\tilde{x}^{s_n} \right) d\tilde{x}=a^{-1/s_n}$, which is $1-(\log a)/s_n+o(1/s_n)$.
Subtracting this from \eqrefPartI{munhattau}, we can therefore prove \eqrefPartI{munhatUpTo1oversn} if we show that
\begin{equation}%\labelPartI{}
s_n\int_0^\infty \left( \e^{-a \e^{-\gamma} f_n(\tilde{x})/f_n(1)} - \e^{-a \Gamma(1+1/s_n)^{s_n} \tilde{x}^{s_n}} \right)d\tilde{x}
\to 0,
\end{equation}
or equivalently, by the substitution $\tilde{x}=x^{1/s_n}$, if we show that
\begin{equation}\labelPartI{DifferenceOfIntegrals}
\int_0^\infty x^{1/s_n - 1} \left( \e^{-a \e^{-\gamma} f_n(x^{1/s_n})/f_n(1)} - \e^{-a \Gamma(1+1/s_n)^{s_n} x} \right) dx \to 0.
\end{equation}
 Note that $\Gamma(1+1/s)^s\to\e^{-\gamma}$ as $s\to\infty$.
Together with \refcond{scalingfn}, this implies that the integrand in \eqrefPartI{DifferenceOfIntegrals} converges pointwise to $0$.
 For $x\leq 1$, $f_n(x^{1/s_n}) \le f_n(1)x^{\epsilonCondition}$ means that the integrand is bounded by $O(x^{\epsilonCondition-1}+1)$.
 For $x\geq 1$, \reflemma{ExtendedImpliesWeak} implies that the integrand is bounded by $\e^{-\delta x^{\epsilonCondition}}$ for some $\delta>0$.
Dominated convergence therefore completes the proof of \eqrefPartI{munhatUpTo1oversn}.

To conclude \reflemma{lambdanAsymp} from \eqrefPartI{munhatUpTo1oversn}, we use the monotonicity of $\mu_n$. Taking $a>1$ shows that $\hat{\mu}_n\bigl(a\e^{-\gamma}/f_n(1)\bigr)<1$ for all $n$ large enough, implying $\lambda_n< a\e^{-\gamma}/f_n(1)$ and $\limsup_{n\to\infty} \lambda_n f_n(1)\leq \e^{-\gamma}$.
A lower bound holds similarly.
\end{proof}

\begin{lemma}\lblemma{ModerateAgeContribution}
Suppose \refcond{scalingfn} holds.
Given $K<\infty$, there exist $\epsilon_K>0$ and $n_0\in\N$ such that, for $0\leq t\leq K f_n(1)$ and $n\geq n_0$,
\begin{equation}%\labelPartI{}
\int_0^{\infty}\e^{-\lambda_n y} \mu_n(t+dy)\geq \epsilon_K/s_n.
\end{equation}
\end{lemma}
\begin{proof}
For any $0\leq t\leq Kf_n(1)$,
\begin{equation}%\labelPartI{}
\int_0^{\infty} \e^{-\lambda_n y} \mu_n(t+dy) = \int_0^{\infty} \e^{-\lambda_n(y-t)} \indicator{y\geq t} d\mu_n(y) \geq \e^{-2\lambda_nKf_n(1)} \mu_n(Kf_n(1), 2Kf_n(1)).
\end{equation}
By \reflemma{lambdanAsymp}, $\lambda_n f_n(1)$ converges to a finite constant.
By \refcond{scalingfn}, $f_n(1+x/s_n)/f_n(1)\to \e^x$, and it follows that $\mu_n(Kf_n(1),2Kf_n(1))=f_n^{-1}(2Kf_n(1))-f_n^{-1}(Kf_n(1)) \sim (\log 2)/s_n$.
\end{proof}

\section{Coupling \texorpdfstring{$K_n$}{the complete graph} and the PWIT}\lbsect{Coupling}

In \refthm{CouplingFPP}, we indicated that two random processes, the first passage exploration processes $\cS$ and $\cluster$ on $K_n$ and $\tree$, respectively, could be coupled.
In \refsubsect{PWITFPPIP} we have intuitively described the coupling between FPP on $K_n$ and on the PWIT.
In this section we explain how this coupling arises as a special case of a general family of couplings between $K_n$, understood as a random edge-weighted graph with i.i.d.\ exponential edge weights, and the PWIT.
In \refsubsect{MinimalRule} we define the minimal rule processes and the thinning in this context.
In \refsubsect{ProofCoupling1Source} and \refsubsect{ProofCoupIPPWIT} we prove Theorems~\refPartI{t:Coupling1Source} and \refPartI{t:CoupIP-PWIT}, respectively.

\subsection{Exploration processes and the definition of the coupling}
\lbsubsect{DefExploration}

As in Sections~\refPartI{ss:PWITFPPIP} and \refPartI{ss:DecomposeFPP}, we define $M_{\emptyset_j}=j$, for $j=1,2$, and to each other $v\in \twoPWITs\setminus\set{\emptyset_1,\emptyset_2}$, we associate a mark $M_v$ chosen uniformly and independently from $[n]$.
We next define what an exploration process is:

\begin{defn}[Exploration process on two PWITs]\lbdefn{Explore}
Let $\F_0$ be a $\sigma$-field containing all null sets, and let $(\tree,X)$ be independent of $\F_0$.
We call a sequence $\explore=(\explore_k)_{k\in \N_0}$ of subsets of $\tree$ an \emph{exploration process} if, with probability 1, $\explore_0=\set{\emptyset_1,\emptyset_2}$ and, for every $k\in\N$, either $\explore_k=\explore_{k-1}$ or else $\explore_k$ is formed by adjoining to $\explore_{k-1}$ a previously unexplored child $v_k\in\boundary\explore_{k-1}$, where the choice of $v_k$ depends only on the weights $X_w$ and marks $M_w$ for vertices $w \in \explore_{k-1}\union\boundary\explore_{k-1}$ and on events in $\F_0$.
\end{defn}
Examples for exploration processes are given by FPP and IP on $\tree$.
For FPP, as defined in \refdefn{FPPonPWITdef}, it is necessary to convert to discrete time by observing the branching process at those moments when a new vertex is added, similar to \refthm{CoupIP-PWIT}.
The standard IP on $\tree$ is defined as follows. Set $\IP(0)=\set{\emptyset_1,\emptyset_2}$. For $k \in \N$, form $\IP(k)$ inductively by adjoining to $\IP(k-1)$ the boundary vertex $v\in \boundary \IP(k-1)$ of minimal weight.
However, an exploration process is also obtained when we specify at each step (in any suitably measurable way) whether to perform an invasion step in $\tree^{\sss(1)}$ or in $\tree^{\sss(2)}$.

For $k \in \N$, let $\F_k$ be the $\sigma$-field generated by $\F_0$ together with the weights $X_w$ and marks $M_w$ for vertices $w \in \explore_{k-1}\union\boundary\explore_{k-1}$.
Note that the requirement on the choice of $v_k$ in \refdefn{Explore} can be expressed as the requirement that $\explore$ is $(\F_k)_k$-adapted.

For $v\in \tree$, define the exploration time of $v$ by
	\begin{equation}%\labelPartI{}
	N_v=\inf\set{k \in \N_0\colon v\in \explore_k} .
	\end{equation}

\begin{defn}[Thinning]\lbdefn{Thinning}
The vertex $v\in\tree\setminus\set{\emptyset_1,\emptyset_2}$ is \emph{thinned} if it has an ancestor $v_0=\ancestor{k}{v}$ (possibly $v$ itself) such that $M_{v_0}=M_w$ for some unthinned vertex $w$ with $N_w<N_{v_0}$.
Write $\thinnedExplore_k$ for the subgraph of $\explore_k$ consisting of unthinned vertices.
\end{defn}

Recall the remark below \refdefn{ThinningBP} that explains that the definition above is not circular.

We define the stopping times
	\begin{equation}%\labelPartI{NofiDefn}
	N(i)=\inf\set{k \in \N_0\colon M_v=i\text{ for some }v\in\thinnedExplore_k}.
	\end{equation}
at which $i\in[n]$ first appears as a mark in the unthinned exploration process.
Note that, on the event $\set{N(i)<\infty}$, $\thinnedExplore_k$ contains a \emph{unique} vertex in $\tree$ whose mark is $i$, for any $k\geq N(i)$; call that vertex $V(i)$.
On this event, we define
	\begin{equation}\labelPartI{XijDefinition}
	X(i,i')=\min_{w\in \tree}\set{X_w \colon M_w=i', \parent{w}=V(i)}.
	\end{equation}

\begin{lemma}\lblemma{IndepOrMeas}
Conditional on $\F_{N(i)}$, and on the event $\set{N(i)<\infty}$, the distribution of $X(i,i')$ is exponential with mean $n$, independently for every $i'$.
Moreover, $X(i,i')$ is $\F_{N(i)+1}$ measurable.
\end{lemma}

\begin{proof}
The event $\set{N(i)=k,V(i)=v}$ is measurable with respect to the $\sigma$-field generated by $\F_0$ together with all edge weights $X_v,X_w$ and marks $M_v,M_w$ for which $w$ is not a descendant of $v$.
On the other hand, on $\set{V(i)=v}$, $X(i,i')=\min\set{X_w\colon M_w=i',\parent{w}=v}$ depends only on the marks and edge weights of children of $v$.
Therefore, the distribution of $X(i,i')$ and the independence for different $i'$ follow from the thinning property of Poisson point processes.
Since the marks and edge weights of the children of $V(i)$ are measurable with respect to $\F_{N(i)+1}$, $X(i,i')$ is measurable with respect to this $\sigma$-field.
\end{proof}

We define, for an edge $\set{i,i'} \in E(K_n)$,
\begin{equation}\labelPartI{EdgeWeightCoupling}
X_{\set{i,i'}}^{\sss(K_n)}=
\begin{cases}
\tfrac{1}{n} X(i,i') & \text{if } N(i)<N(i'),\\
\tfrac{1}{n} X(i',i) & \text{if } N(i')<N(i),\\
E_{\set{i,i'}} & \text{if } N(i)=N(i')=\infty \text{ or } N(i)=N(i')=0,
\end{cases}
\end{equation}
where $(E_e)_{e\in E(K_n)}$ are exponential variables with mean 1, independent of each other and of $(X_v)_{v}$.
\begin{theorem}\lbthm{CouplingExpl}
If $\explore$ is an exploration process on the union $\tree$ of two PWITs, then the edge weights $X_e^{\sss(K_n)}$ defined in \eqrefPartI{EdgeWeightCoupling} are exponential with mean $1$, independently for each $e \in E(K_n)$.
\end{theorem}

The idea underlying \refthm{CouplingExpl} is that, by \reflemma{IndepOrMeas}, each variable $\tfrac{1}{n}X(i,i')$ is exponentially distributed conditional on the past up to the moment $N(i)$ when it may be used to set the value of $X_{\set{i,i'}}^{\sss(K_n)}$.
However, formalizing this notion requires careful attention to the relative order of the stopping times $N(i)$ and to which $N(i)$ are infinite.

\begin{proof}
Let $(h_e)_{\in E(K_n)}$ be an arbitrary collection of bounded, measurable functions, and abbreviate $\langle h_e\rangle=\E[h_e(E)]$, where $E$ is exponential with mean $1$.
It suffices to prove that
\begin{equation}\labelPartI{CouplingToProve}
\E\Bigl( \prod_{e\in E(K_n)} h_e(X_e^{\sss(K_n)}) \Bigr) = \prod_{e\in E(K_n)} \langle h_e\rangle.
\end{equation}
We proceed by induction.
To begin, we partition \eqrefPartI{CouplingToProve} according to the number $\ell\in\set{0,1,\dotsc,n-2}$ of indices $i\neq 1,2$ for which $N(i)=\infty$, as well as the relative order of the finite values of $N(i)$.
Define $i_1=1$, $i_2=2$ and, given $\ell \in [n-2]$ and $\boldsymbol{i}=(i_3,\ldots,i_{n-\ell})$, abbreviate $S_{\ell,\boldsymbol{i}}=[n]\setminus \set{i_1,\ldots,i_{n-\ell}}$.

Note that, on the event $\set{N(i)=\infty \, \forall i\in S_{\ell,\boldsymbol{i}}}$, we have $X_{\set{i,j}}^{\sss(K_n)}=E_{\set{i,j}}$ for $\set{i,j} \subset S_{\ell, \boldsymbol{i}}$ by \eqrefPartI{EdgeWeightCoupling}.
The $E_{\set{i,j}}$ are exponential, independently from everything else, so we may perform the integration over these variables separately.
We conclude that
\begin{align}
&\E\Bigl(\prod_{e\in E(K_n)} h_e(X_e^{\sss (K_n)})\Bigr)=\sum_{(i_3,\dotsc,i_n)} \E\Bigl(\indicator{N(i_3)<\dotsb<N(i_{n-1})<N(i_n)\le \infty} \prod_{e \in E(K_n)} h_e(X_e^{\sss (K_n)}) \Bigr)
\labelPartI{CouplingBaseCase}\\
&+ \sum_{\ell=2}^{n-2} \sum_{\boldsymbol{i}=(i_3,\ldots ,i_{n-\ell})} \E\Bigl(\indicator{N(i_3) <\ldots< N(i_{n-\ell})<\infty, N(i)=\infty \, \forall i \in S_{\ell,\boldsymbol{i}}}\prod_{\set{i,j}\not\subset S_{\ell,\boldsymbol{i}}} h_{\set{i,j}}(X_{\set{i,j}}^{\sss (K_n)}) \Bigr) \prod_{\set{i,j} \subset S_{\ell,\boldsymbol{i}}} \langle h_{\set{i,j}} \rangle
,
\notag
\end{align}
where the first sum corresponds to $\ell=0$ and $\ell=1$.
The sums are over vectors of distinct indices $i_3,\ldots,i_n \in \set{3,\ldots,n}$ and $i_3,\ldots,i_{n-\ell} \in \set{3,\ldots,n}$, respectively, and the notation ${\set{i,j}}\not\subset S_{\ell,\boldsymbol{i}}$ means that ${\set{i,j}}$ is an edge with at least one endpoint in $[n]\setminus S_{\ell,\boldsymbol{i}}=\set{i_1,\dotsc,i_{n-\ell}}$.

In general, for $\boldsymbol{i}=(i_3,\dotsc,i_{n-\ell})$ given, define the events
\begin{equation}%\labelPartI{}
A_{\ell,\boldsymbol{i}}=\set{N(i_3)<\dotsb<N(i_{n-\ell})<\infty},
\qquad
B_{\ell,\boldsymbol{i}}=\set{N(i)>N(i_{n-\ell}) \, \forall i\in S_{\ell,\boldsymbol{i}}}.
\end{equation}
We claim that, for all $\ell_0 \in [n-2]$,
\begin{align}
&\E\Bigl(\prod_{e \in E(K_n)} h_{e}(X_e^{\sss (K_n)})\Bigr)
\notag\\&\quad
=
\sum_{\boldsymbol{i}=(i_3,\ldots, i_{n-\ell_0})} \E\Bigl(\indicatorofset{A_{\ell_0,\boldsymbol{i}}}\indicatorofset{B_{\ell_0,\boldsymbol{i}}} \prod_{\set{i,j}\not\subset S_{\ell_0,\boldsymbol{i}}} h_{\set{i,j}}(X_{\set{i,j}}^{\sss (K_n)}) \Bigr) \prod_{\set{i,j} \subset S_{\ell_0,\boldsymbol{i}}} \langle h_{\set{i,j}} \rangle
\notag\\&\qquad
+ \sum_{\ell=\ell_0+1}^{n-2} \sum_{\boldsymbol{i}=(i_3,\ldots ,i_{n-\ell})} \E\Bigl(\indicatorofset{A_{\ell,\boldsymbol{i}}}\indicator{N(i)=\infty \, \forall i \in S_{\ell,\boldsymbol{i}}}\prod_{\set{i,j}\not\subset S_{\ell,\boldsymbol{i}}} h_{\set{i,j}}(X_{\set{i,j}}^{\sss (K_n)}) \Bigr) \prod_{\set{i,j} \subset S_{\ell,\boldsymbol{i}}} \langle h_{\set{i,j}} \rangle.
\labelPartI{CouplingInduction}
\end{align}
When $\ell=n-2$, by convention, the second sum vanishes, while in the first sum $\boldsymbol{i}$ is the empty sequence.

The case $\ell_0=1$ reduces to \eqrefPartI{CouplingBaseCase}: then $S_{\ell_0,\boldsymbol{i}}$ contains only one element, which we called $i_n$ but which is in fact uniquely determined by the values $i_3,\dotsc,i_{n-1}$, and the product $\prod_{{\set{i,j}}\subset S_{\ell_0,\boldsymbol{i}}} \langle h_{\set{i,j}} \rangle$ is empty.
This initializes the induction hypothesis.

We remark that in the right-hand sides of \eqrefPartI{CouplingBaseCase} and \eqrefPartI{CouplingInduction}, the indicators already allow us to determine which of the three cases from \eqrefPartI{EdgeWeightCoupling} occurs.
For notational simplicity, we will introduce this information gradually as we proceed.

Now suppose \eqrefPartI{CouplingInduction} has been proved for a given $\ell_0<n-2$.
In the first summand of the right-hand side of \eqrefPartI{CouplingInduction}, we condition on $\F_{N(i_{n-\ell_0})}$.
By \reflemma{IndepOrMeas} and the presence of the indicators, each factor $h_{\set{i,j}}(X_{\set{i,j}}^{\sss(K_n)})$ is equal to a factor $h_{\set{i,j}}(\frac{1}{n}X(i,j))$ (or $h_{\set{i,j}}(\frac{1}{n}X(j,i))$, if $N(j)<N(i)$) that is $\F_{N(i_{n-\ell_0})}$-measurable, with the exception of the factors $h_{\set{i_{n-\ell_0}, j}}(\tfrac{1}{n}X(i_{n-\ell_0},j))$ for $j\in S_{\ell_0,\boldsymbol{i}}$, which are conditionally independent given $\F_{N(i_{n-\ell_0})}$, again by \reflemma{IndepOrMeas}.
Note furthermore that $A_{\ell_0,\boldsymbol{i}},B_{\ell_0,\boldsymbol{i}}\in\F_{N(i_{n-\ell_0})}$,   $A_{\ell_0,\boldsymbol{i}}=A_{\ell_0+1,\boldsymbol{i}}\intersect\set{N(i_{n-\ell_0-1})<N(i_{n-\ell_0})<\infty}$ and $S_{\ell_0,\boldsymbol{i}}\cup \set{i_{n-\ell_0}}=S_{\ell_0+1,\boldsymbol{i}}$.
Thus,
\begin{align}
&\E\Bigl(\indicatorofset{A_{\ell_0,\boldsymbol{i}}}\indicatorofset{B_{\ell_0,\boldsymbol{i}}} \prod_{\set{i,j}\not\subset S_{\ell_0,\boldsymbol{i}}} h_{\set{i,j}}(X_{\set{i,j}}^{\sss (K_n)}) \Bigr)\prod_{\set{i,j} \subset S_{\ell_0,\boldsymbol{i}}} \langle h_{\set{i,j}} \rangle
\notag\\&\quad
=
\E\Bigl(\indicatorofset{A_{\ell_0,\boldsymbol{i}}}\indicator{N(i)>N(i_{n-\ell_0}) \, \forall i \in S_{\ell_0,\boldsymbol{i}}}
\prod_{\set{i,j}\not\subset S_{\ell_0,\boldsymbol{i}} \union \shortset{i_{n-\ell_0}}} h_{\set{i,j}}(X_{\set{i,j}}^{\sss (K_n)}) \Bigr)\prod_{\set{i,j} \subset S_{\ell_0,\boldsymbol{i}}\union \shortset{i_{n-\ell_0}}} \langle h_{\set{i,j}} \rangle
\notag\\&\quad
=
\E\Bigl(\indicatorofset{A_{\ell_0+1,\boldsymbol{i}}}\indicator{N(i_{n-\ell_0-1})<N(i_{n-\ell_0})<N(i) \, \forall i \in S_{\ell_0,\boldsymbol{i}}}
\prod_{\set{i,j}\not\subset S_{\ell_0+1,\boldsymbol{i}}} h_{\set{i,j}}(X_{\set{i,j}}^{\sss (K_n)}) \Bigr)\prod_{\set{i,j} \subset S_{\ell_0+1,\boldsymbol{i}}} \langle h_{\set{i,j}} \rangle.
\labelPartI{Integratedell0}
\end{align}
Leaving $i_3,\dotsc, i_{n-\ell_0-1}$ fixed, we now sum \eqrefPartI{Integratedell0} over all $i_{n-\ell_0} \in [n]\setminus \set{i_1,\dotsc,i_{n-\ell_0-1}}$.
This rewrites the first sum in \eqrefPartI{CouplingInduction} as
\begin{equation}\labelPartI{Summedell0}
\sum_{\boldsymbol{i}=(i_3,\dotsc,i_{n-\ell_0-1})}\E\Bigl(\indicatorofset{A_{\ell_0+1,\boldsymbol{i}}}\indicatorofset{B_{\ell_0+1,\boldsymbol{i}}} \indicator{N(i)<\infty\text{ for some }i \in S_{\ell_0+1,\boldsymbol{i}}}
\prod_{\set{i,j}\not\subset S_{\ell_0+1,\boldsymbol{i}}} h_{\set{i,j}}(X_{\set{i,j}}^{\sss (K_n)})
 \Bigr)\prod_{\set{i,j} \subset S_{\ell_0+1,\boldsymbol{i}}} \langle h_{\set{i,j}} \rangle.
\end{equation}
However, \eqrefPartI{Summedell0} combines with the summand $\ell=\ell_0+1$ from the second sum in \eqrefPartI{CouplingInduction} (since we can rewrite $\indicatorofset{A_{\ell_0+1,\boldsymbol{i}}}\indicator{N(i)=\infty \, \forall i \in S_{\ell_0+1,\boldsymbol{i}}}$ as $\indicatorofset{A_{\ell_0+1,\boldsymbol{i}}}\indicatorofset{B_{\ell_0+1,\boldsymbol{i}}}\indicator{N(i)=\infty \, \forall i \in S_{\ell_0+1,\boldsymbol{i}}}$) to produce the first sum in \eqrefPartI{CouplingInduction} for $\ell_0$ replaced by $\ell_0+1$.
This advances the induction hypothesis, and thus completes the proof of \eqrefPartI{CouplingInduction} for all $\ell_0\leq n-2$.

We therefore conclude that \eqrefPartI{CouplingInduction} holds when $\ell_0=n-2$.
In this case the second sum vanishes, while in the first sum $\boldsymbol{i}$ is the empty sequence, $S_{n-2,\boldsymbol{i}}=[n]\setminus \set{1,2}$ and the events $A_{n-2,\boldsymbol{i}},B_{n-2,\boldsymbol{i}}$ always occur (note that $N(i_2)=0$).
Therefore
\begin{align}
\E\Bigl(\prod_{e \in E(K_n)} h_{e}(X_e^{\sss (K_n)})\Bigr)
&=
\E\Bigl(\prod_{\set{i,j}\cap \set{1,2} \not=\emptyset} h_{\set{i,j}}(X_{\set{i,j}}^{\sss (K_n)}) \Bigr)\prod_{\set{i,j} \cap\set{1,2}=\emptyset} \langle h_{\set{i,j}} \rangle
\\&
=
\E\Bigl(h_{\set{1,2}}(E_{\set{1,2}}) \prod_{i=3}^n h_{\set{1,i}}(\tfrac{1}{n} X(1,i))h_{\set{2,i}}(\tfrac{1}{n} X(2,i))\Bigr)\prod_{\set{i,j} \cap\set{1,2}=\emptyset} \langle h_{\set{i,j}} \rangle
.
\notag
\end{align}
By \reflemma{IndepOrMeas}, $(\tfrac{1}{n}X(1,i))_{i\geq 3}$ and $(\tfrac{1}{n}X(2,i))_{i\geq 3}$ are each families of independent exponential random variables with mean $1$.
Moreover they are mutually independent, since they are determined from the independent Poisson point processes of edge weights corresponding to $\emptyset_1$ and $\emptyset_2$, respectively.
Since furthermore $E_{\set{1,2}}$ is independent of everything, we conclude that \eqrefPartI{CouplingToProve} holds.
\end{proof}

\subsection{Minimal-rule exploration processes}
\lbsubsect{MinimalRule}

An important class of exploration processes, which includes both FPP and IP, are those exploration processes determined by a minimal rule in the following sense:

\begin{defn}\lbdefn{MinimalRule}
A \emph{minimal rule} for an exploration process $\explore$ on $\tree$ is an $(\F_k)_k$-adapted sequence $(\makebox{$S_k,\prec_k$})_{k=1}^\infty$, where $S_k\subset\boundary\explore_{k-1}$ is a (possibly empty) subset of the boundary vertices of $\explore_{k-1}$ and $\prec_k$ is a strict total ordering of the elements of $S_k$ (if any) such that the implication
\begin{equation}\labelPartI{MinimalImplication}
w\in S_k, \parent{v}=\parent{w}, M_v=M_w, X_v<X_w \quad\implies\quad v\in S_k, v \prec_k w
\end{equation}
holds.
An exploration process is \emph{determined by the minimal rule} $(S_k,\prec_k)_{k=1}^\infty$ if $\explore_k=\explore_{k-1}$ whenever $S_k=\emptyset$ and otherwise $\explore_k$ is formed by adjoining to $\explore_{k-1}$ the unique vertex $v_k\in S_k$ that is minimal with respect to $\prec_k$.
\end{defn}
In words, in every step $k$ there is a set of boundary vertices $S_k$ from which we can select for the next exploration step.
The content of \eqrefPartI{MinimalImplication} is that, whenever a vertex $w\in S_k$ is available for selection, then all siblings of $w$ with the same mark but smaller weight are also available for selection and are preferred over $w$.

For FPP without freezing on $\tree$ with edge weights $f_n(X_v)$, we take $v \prec_k w$ if and only if $T_v < T_w$ (recall \eqrefPartI{TvDefinition}) and take $S_k=\boundary\explore_{k-1}$.
For IP on $\tree$, we have $v \prec_k w$ if and only if $X_v<X_w$; the choice of subset $S_k$ can be used to enforce, for instance, whether the $k^\th$ step is taken in $\tree^{\sss(1)}$ or $\tree^{\sss(2)}$.

Recall the subtree $\thinnedExplore_k$ of unthinned vertices from \refdefn{Thinning} and the subgraph $\pi_M(\thinnedExplore_k)$ from \refdefn{InducedGraph}.
That is, $\pi_M(\thinnedExplore_k)$ is the union of two trees with roots $1$ and $2$, respectively, and for $v\in \thinnedExplore_k \setminus\set{\emptyset_1,\emptyset_2}$, $\pi_M(\thinnedExplore_k)$ contains vertices $M_v$ and $M_{\parent{v}}$ and the edge $\set{M_v, M_{\parent{v}}}$.

For any $i\in [n]$ for which $N(i)<\infty$, recall that $V(i)$ is the unique vertex of $\thinnedExplore_k$ ($k\geq N(i)$) for which $M_{V(i)}=i$.
Define $V(i,i')$ to be the first child of $V(i)$ with mark $i'$.

Recalling \eqrefPartI{XijDefinition}, an equivalent characterization of $V(i,i')$ is
\begin{equation}\labelPartI{XijVij}
X(i,i')=X_{V(i,i')}.
\end{equation}
The following lemma shows that, for an exploration process determined by a minimal rule, unthinned vertices must have the form $V(i,i')$:

\begin{lemma}\lblemma{NextNotThinned}
Suppose $\explore$ is an exploration process determined by a minimal rule $(S_k,\prec_k)_{k=1}^\infty$ and $k\in \N$ is such that $\thinnedExplore_k\neq \thinnedExplore_{k-1}$. Let $i_k=M_{\parent{v_k}}$ and $i_k'=M_{v_k}$. Then $v_k=V(i_k,i_k')$.
\end{lemma}

\begin{proof}
By construction, $\parent{v_k}\in\thinnedExplore_{k-1}$ and $M_{\parent{v_k}}=i_k$, so $V(i_k)=\parent{v_k}$ by definition.
Moreover, $V(i_k)=\parent{V(i_k,i_k')}$ and $M_{V(i_k,i_k')} =i_k' = M_{v_k}$.
Suppose to the contrary that $V(i_k,i_k')\neq v_k$.
By the definition of $V(i_k,i_k')$, it follows that $X_{V(i_k,i_k')}<X_{v_k}$ and \eqrefPartI{MinimalImplication} yields $V(i_k,i_k')\in S_k$ and $V(i_k,i_k')\prec_k v_k$, a contradiction since $v_k$ must be minimal for $\prec_k$.
\end{proof}

If $\explore$ is an exploration process determined by a minimal rule, then we define
\begin{equation}\labelPartI{SelectableInKn}
S_k^{\sss (K_n)}=\set{\set{i,i'} \in E(K_n) \colon i \in \pi_M(\thinnedExplore_{k-1}), i' \notin \pi_M(\thinnedExplore_{k-1}), V(i,i') \in S_k}
\end{equation}
and
\begin{equation}\labelPartI{OrderInKn}
e_1 \;\widetilde{\prec}_k\; e_2 \quad\iff\quad V(i_1,i_1') \prec_k V(i_2,i_2'), \qquad e_1,e_2\in S_k^{\sss(K_n)},
\end{equation}
where $e_j=\set{i_j,i'_j}$ and $i_j \in \pi_M(\thinnedExplore_{k-1}), i'_j \notin \pi_M(\thinnedExplore_{k-1})$ as in \eqrefPartI{SelectableInKn}.
\begin{prop}[Thinned minimal rule]\lbprop{MinimalRuleThinning}
Suppose $\explore$ is an exploration process determined by a minimal rule $(S_k,\prec_k)_{k=1}^\infty$.
Then, under the edge-weight coupling \eqrefPartI{EdgeWeightCoupling}, the edge weights of $\pi_M(\thinnedExplore_k)$ are determined by
\begin{equation}\labelPartI{InternalEdgeWeights}
X_{\shortset{M_v,M_{\parent{v}}}}^{\sss(K_n)} = \tfrac{1}{n} X_v
\quad
\text{for any }v\in \union_{k=1}^\infty \thinnedExplore_k\setminus\set{\emptyset_1,\emptyset_2}
\end{equation}
and generally
\begin{equation}\labelPartI{BoundaryEdgeWeights}
X_{\set{i,i'}}^{\sss(K_n)} = \tfrac{1}{n} X_{V(i,i')}
\quad
\text{whenever}
\quad
i\in \pi_M(\thinnedExplore_{k-1}), i'\notin \pi_M(\thinnedExplore_{k-1})\text{ for some }k\in\N.
\end{equation}
Moreover, for any $k\in\N$ for which $\thinnedExplore_k\neq \thinnedExplore_{k-1}$, $\pi_M(\thinnedExplore_k)$ is formed by adjoining to $\pi_M(\thinnedExplore_{k-1})$ the unique edge $e_k\in S_k^{\sss (K_n)}$ that is minimal with respect to $\widetilde{\prec}_k$.
\end{prop}

\refprop{MinimalRuleThinning} asserts that the subgraph $\pi_M(\thinnedExplore_k)$ of $K_n$, equipped with the edge weights $(X_e^{\sss(K_n)})_{e\in E(\pi_M(\thinnedExplore_k))}$, is isomorphic as an edge-weighted graph to the subgraph $\thinnedExplore_k$ of $\tree$, equipped with the rescaled edge weights $(\tfrac{1}{n} X_v)_{v\in\thinnedExplore_k\setminus\set{\emptyset_1,\emptyset_2}}$.
Furthermore, the subgraphs $\pi_M(\thinnedExplore_k)$ can be grown by an inductive rule.
Thus the induced subgraphs $(\pi_M(\thinnedExplore_k))_{k=0}^\infty$ themselves form a minimal-rule exploration process on $K_n$, with a minimal rule derived from that of $\explore$, with the caveat that $\widetilde{\prec}_k$ may depend on edge weights from $\explore_{k-1}\setminus\thinnedExplore_{k-1}$ as well as from $\pi_M(\thinnedExplore_{k-1})$.
\begin{proof}[Proof of \refprop{MinimalRuleThinning}]
We first prove \eqrefPartI{BoundaryEdgeWeights}.
By assumption, $N(i)\leq k-1<N(i')$, so \eqrefPartI{BoundaryEdgeWeights} is simply the first case in \eqrefPartI{EdgeWeightCoupling} (see also \eqrefPartI{XijVij}).

Take $v\in\union_{k=1}^\infty\thinnedExplore_k\setminus\set{\emptyset_1,\emptyset_2}$, and assume that $v=v_k$, i.e., set $k=N_v\geq 1$.
Set $i_k=M_{\parent{v_k}}$ and $i'_k=M_{v_k}$.
By construction, $i_k\in\pi_M(\thinnedExplore_{k-1})$ but $i'_k\notin\pi_M(\thinnedExplore_{k-1})$, and according to \reflemma{NextNotThinned}, $v_k=V(i_k,i'_k)$.
So \eqrefPartI{InternalEdgeWeights} is a special case of \eqrefPartI{BoundaryEdgeWeights}.

By construction, $\pi_M(\thinnedExplore_k)$ is formed by adjoining to $\pi_M(\thinnedExplore_{k-1})$ the vertex $i_k=M_{v_k}\in [n]$ via the edge $e_k=\set{i_k,i_k'}$.
By \reflemma{NextNotThinned}, $v_k=V(i_k,i_k')$, and by the definition of minimal rule, the vertex $v_k$ belongs to $S_k$ and is minimal for $\prec_k$.
It follows from the definitions \eqrefPartI{SelectableInKn}--\eqrefPartI{OrderInKn} that $e_k\in S_k^{\sss (K_n)}$ is minimal for $\widetilde{\prec}_k$.
\end{proof}

\subsection{Coupling \texorpdfstring{$\SWT^{\sss(1)}$ and $\BP^{\sss(1)}$}{SWT and BP}: Proof of \refthm{Coupling1Source}}
\lbsubsect{ProofCoupling1Source}

In this section, we prove \refthm{Coupling1Source}: that is, we couple the smallest-weight tree $\SWT^{\sss(1)}$ on $K_n$ to a single branching process $\BP^{\sss(1)}$ on $\tree^{\sss(1)}$.
Since this statement is concerned with processes starting from only one source, we use exploration processes on $\tree^{\sss (1)}$ instead of $\tree$.
All results from Sections~\refPartI{ss:DefExploration} and \refPartI{ss:MinimalRule} carry over up to obvious changes (indeed, the results hold for any finite number of copies of the PWIT).

\begin{proof}[Proof of \refthm{Coupling1Source}]
Let the exploration process $\explore^{\sss(\FPP,1)}$ on $\tree^{\sss (1)}$ be determined by the minimal rule where $S_k=\boundary\explore^{\sss(\FPP,1)}_{k-1}$ and $v \prec_k^\FPP w$ if and only if $T_v < T_w$.
In words, $\explore^{\sss(\FPP,1)}$ performs FPP steps on $\tree^{\sss(1)}$ in discrete time, and it is easy to verify that $\explore^{\sss(\FPP,1)}_k=\BP^{\sss(1)}_{T_{v_k}}$, where $v_k$ was defined in \refdefn{MinimalRule}.
Throughout this proof, we choose $\explore=\explore^{\sss(\FPP,1)}$.

The smallest-weight tree $\SWT^{\sss(1)}$ evolves in discrete time as follows. At time $0$, $\SWT_0^{\sss(1)}$ contains only vertex $1$ and no edges. Let $\tau_{k'-1}$ be the time that the $(k'-1)^\st$ vertex, not including vertex $1$, is added. Write $\tau_0=0$. Starting from time $\tau_{k'-1}$, the next edge added is the minimizer $e'_{k'}$ of $d_{K_n,Y^{(K_n)}}(1,i)+Y_e^{\sss(K_n)}$ over the set of boundary edges $e=\set{i,j}$ with $i\in\SWT_{\tau_{k'-1}}^{\sss(1)}$, $j\notin\SWT_{\tau_{k'-1}}^{\sss(1)}$, and $e_{k'}=\set{i'_k,j'_k}$ is added at time $\tau_{k'}=d_{K_n,Y^{(K_n)}}(1,i'_{k'})+Y_{e_{k'}}^{\sss(K_n)}$.
It is easy to verify by induction that for any $i\in\SWT^{\sss(1)}_{\tau_{k'-1}}$, $d_{K_n,Y^{(K_n)}}(1,i)$ equals the sum of edge weights $Y_{e}^{\sss(K_n)}$ for $e$ belonging to the unique path in $\SWT^{\sss(1)}_{\tau_{k'-1}}$ from $1$ to $i$.

Define the edge weights on $K_n$ according to \eqrefPartI{EdgeWeightCoupling} using the exploration process $\explore=\explore^{\sss(\FPP,1)}$.
By \refthm{CouplingExpl}, the edge weights $X_e^{\sss(K_n)}$ are i.i.d.\ exponential with mean 1.
Hence, as discussed in \refsubsect{univ-class}, the edge weights $Y_e^{\sss(K_n)}=g(X_e^{\sss(K_n)})$ have the distribution function $F_Y$, so that $\SWT^{\sss(1)}$ has the correct law under this coupling.

Both $\thinnedBP^{\sss(1)}$ and $\SWT^{\sss(1)}$ are increasing jump processes and $\pi_M(\thinnedBP^{\sss(1)}_0)=\SWT_{0}^{\sss(1)}$.
By an inductive argument, it suffices to show that if $k$, $k'$ are such that $\thinnedExplore_k\neq\thinnedExplore_{k-1}$ and $\pi_M(\thinnedExplore_{k-1})=\SWT^{\sss(1)}_{\tau_{k'-1}}$ then (a)~the edge $e'_{k'}$ next added to $\SWT_{\tau_{k'-1}}^{\sss(1)}$ is the same as the edge $e_k=\set{i_k,i_k'}$ that is minimal with respect to $\widetilde{\prec}_k$ and therefore next added to $\pi_M(\thinnedExplore_{k-1})$; and (b)~$\tau_{k'}=T_{v_k}$.
(We remark that $k>k'$ is possible.)

Let $i\in V(\pi_M(\thinnedExplore_{k-1}))$.
The unique path in $\SWT^{\sss(1)}_{\tau_{k'-1}} = \pi_M(\thinnedExplore_{k-1})$ from $i$ to $1$ is the image of the unique path in $\thinnedExplore_{k-1}$ from $V(i)$ to $\emptyset_1$ under the mapping $v\mapsto M_v$ (recall \refdefn{InducedGraph}).
According to \eqrefPartI{InternalEdgeWeights}, \eqrefPartI{EdgesByIncrFunct} and \eqrefPartI{fnFromParamEdges}, the edge weights along this path are $Y^{\sss(K_n)}_{\shortset{M_{\ancestor{m-1}{V(i)}},M_{\ancestor{m}{V(i)}}}}=g(X^{\sss(K_n)}_{\shortset{M_{\ancestor{m-1}{V(i)}},M_{\ancestor{m}{V(i)}}}})=g(\tfrac{1}{n}X_{\ancestor{m-1}{V(i)}})$ and $f_n(X_{\ancestor{m-1}{V(i)}})$, for $m=1,\dotsc,\abs{V(i)}$.
Summing gives $d_{K_n,Y^{(K_n)}}(1,i)=T_{V(i)}$.

In addition, let $i'\notin V(\pi_M(\thinnedExplore_{k-1}))$ and write $e=\set{i,i'}$.
By \eqrefPartI{BoundaryEdgeWeights}, $X_e^{\sss(K_n)}=\tfrac{1}{n}X_{V(i,i')}$, so that $Y_e^{\sss(K_n)}=f_n(X_{V(i,i')})$.
Thus $e'_{k'}$ is the minimizer of $d_{K_n,Y^{(K_n)}}(1,i)+Y_e^{\sss(K_n)}=T_{V(i)}+f_n(X_{V(i,i')})=T_{V(i,i')}$.
By \refprop{MinimalRuleThinning}, so is $e_k$; thus (a) follows since the minimizer is unique.
Moreover $\tau_{k'}$ is the corresponding minimum value, namely $\tau_{k'}=T_{V(i_k,i_k')}$.
According to \reflemma{NextNotThinned}, $v_k=V(i_k,i_k')$ and (b) follows.
\end{proof}

\subsection{Comparing FPP and IP: Proof of \refthm{CoupIP-PWIT}}\lbsubsect{ProofCoupIPPWIT}

In this section, we prove \refthm{CoupIP-PWIT} by comparing the FPP and IP dynamics on the PWIT.

\begin{proof}[Proof of \refthm{CoupIP-PWIT}]
It is easy to see that $\IP^{\sss(1)}$ is an exploration process determined by a minimal rule.
For instance, we may take $S_k=\boundary\IP^{\sss(1)}_{k-1}$ and $v \prec_k w$ if and only if $X_v<X_w$.

In fact, it will be more convenient to use a different characterization of $\IP^{\sss(1)}$.
For $v\in\tree^{\sss(1)}$, write $\mathcal{O}(v)=(X_{(v,1)},\ldots,X_{(v,|v|)})$ for the vector of edge weights $X_{v'}$ along the path from $\emptyset_1$ to $v$, ordered from largest to smallest.
Set $v \prec_k^\IP w$ if and only if $\mathcal{O}(v)$ is lexicographically smaller than $\mathcal{O}(w)$.
It is an elementary exercise that this minimal rule $(S_k,\prec_k^\IP)_{k=1}^\infty$ also determines $\IP^{\sss(1)}$.

Couple the edge weights on $K_n$ according to \eqrefPartI{EdgeWeightCoupling}, where the exploration process is $\explore^{\sss(\FPP,1)}$ from the proof of \refthm{Coupling1Source}.
Fix $m\in\N$.
With high probability, none of the first $m$ vertices explored by $\explore^{\sss(\FPP,1)}$ is thinned.
By \refthm{Coupling1Source}, it therefore suffices to show that $(\explore_k^{\sss(\FPP,1)})_{k=1}^m = (\IP^{\sss(1)}(k))_{k=1}^m$ \whpdot

Let $\epsilon>0$ be given.
Write $B_m$ for the collection of all vertices of the form $\emptyset_1 j_1\dotsb j_r$ with $1\leq r\leq m$ and $j_1,\dotsc, j_r\leq m$.
(That is, $B_m$ consists of all vertices in $\tree^{\sss(1)}$ within $m$ generations for which each ancestor is at most the $m^\th$ child of its parent.)
Note that the first $m$ explored vertices $v_1,\ldots,v_m$ necessarily belong to $B_m$, for both $\explore^{\sss(\FPP,1)}$ and $\IP^{\sss(1)}$.
Let $\delta>0$ and write $A_\delta$ for the event that $\inf\set{X_v\colon v\in B_m}\geq \delta$, $\sup\set{X_v\colon v\in B_m}\leq 1/\delta$, and $\inf\set{\abs{X_v-X_w}\colon v,w\in B_m, v\neq w}\geq\delta$.
We may choose $\delta>0$ sufficiently small that $\P(A_\delta)\geq 1-\epsilon$.

Choose $x_0<x_1<\dotsb<x_N$ such that $x_0=\delta$, $x_N=1/\delta$ and $x_j-x_{j-1} \le \delta/2$ for all $j\in [N]$.
By assumption, there is an $n_0 \in \N$ such that $f_n(x_j)/f_n(x_{j-1}) > m$ for all $j \in [N]$ and $n\ge n_0$.
Hence, for any $x,x'\in[\delta, 1/\delta]$ with $x' \ge x+ \delta$, the monotonicity of $f_n$ implies
\begin{equation}%\labelPartI{}
\frac{f_n(x')}{f_n(x)} \ge \frac{f_n(x_j)}{f_n(x_{j-1})}> m,
\end{equation}
since we may choose $j$ such that $[x_{j-1},x_j]\subset [x,x']$.
From now on assume $n\geq n_0$.

Consider any $v,w\in B_m$ such that $v\neq w$ and neither $v$ nor $w$ is an ancestor of the other.
Then there is a smallest index $j$ with $X_{(v,j)} \neq X_{(w,j)}$.
If $X_{(v,j)}<X_{(w,j)}$ then, on $A_\delta$,
\begin{equation}%\labelPartI{}
\sum_{i=j}^{|v|} f_n(X_{(v,i)}) \le m f_n(X_{(v,j)}) < f_n(X_{(w,j)}) \le \sum_{i=j}^{|w|} f_n(X_{(w,i)}),
\end{equation}
and similarly if $X_{(v,j)} >X_{(w,j)}$.
Hence $v\prec_k^\FPP w$ if and only if $\mathcal{O}(v)$ is lexicographically smaller than $\mathcal{O}(w)$, i.e.\ $v\prec_k^\IP w$, for any of the vertices $v,w$ that may be relevant to $(\explore^{\sss(\FPP,1)}_k)_{k=1}^m$ or $(\IP^{\sss(1)}(k))_{k=1}^m$.
Since $\explore_0^{\sss(\FPP,1)}=\IP^{\sss(1)}(0)$, it follows that, on $A_\delta$ for $n$ sufficiently large, we have $(\explore_k^{\sss(\FPP,1)})_{k=1}^m = (\IP^{\sss(1)}(k))_{k=1}^m$, and since $\P(A_\delta)\geq 1-\epsilon$ with $\epsilon>0$ arbitrary, this completes the proof.
\end{proof}

\section{Poisson Galton-Watson trees, lucky vertices and the emergence of the frozen cluster}\lbsect{FrozenGeometry}

In this section, we prove \refprop{LuckyProb}, \reflemma{RluckyBornSoon} and \refthm{TfrScaling}.
We begin with preliminary results on Poisson Galton--Watson trees.

Throughout this section, we assume Conditions~\refPartI{cond:scalingfn}, \refPartI{cond:LowerBoundfn} and \refPartI{cond:boundfnExtended}.
We will reserve the notation $\epsilonCondition, \deltaCondition$ for some fixed choice of the constants in Conditions~\refPartI{cond:LowerBoundfn} and \refPartI{cond:boundfnExtended}, with $\epsilonCondition$ chosen small enough to satisfy both conditions.

\subsection{Properties of Poisson Galton-Watson trees}

\begin{prop}\lbprop{PGWtrees}
Let $(\tau;\P_m)$ be a Poisson Galton--Watson tree with offspring mean $m>0$ and write $\theta(m)=\P_m(\abs{\tau}=\infty)$ for its survival probability.
For $d\in (0,\infty)$, we denote by $(\tau^{\sss(\le d)};\P_m)$ and $(\tau^{\sss(\ge d)};\P_m)$ the subgraph of $\tau$ consisting of the vertices within distance $d$ and with distance at least $d$, respectively, from the root. In addition, let $(s_n)_n$ be a positive sequence with $s_n \to \infty$.
\begin{enumerate}
\item \lbitem{PGWSurvivalProb}
$\theta \colon (0,\infty) \to [0,1]$ is non-decreasing with $\theta(m)=0$ for $m\leq 1$, $\theta(m)>0$ for $m>1$, and
\begin{equation}\labelPartI{thetaRecursion}
1-\theta(m)=\e^{-m\theta(m)}\qquad \text{for all }m \in (0,\infty).
\end{equation}

\item \lbitem{PGWSurvivalAsymp} As $m\downarrow 1$,
\begin{equation}
\theta(m)\sim 2(m-1) \qquad \text{and}\qquad 1-m(1-\theta(m)) \sim m-1,\labelPartI{thetaAsymptotics}
\end{equation}
and, uniformly over $m\geq 1$, $\theta(m)=O(m-1)$ and $1-m(1-\theta(m))=O(m-1)$.

\item \lbitem{PGWdual} For $m>1$, $\P_m\condparentheses{\big.\abs{\tau}\in \cdot}{ \abs{\tau}<\infty}=\P_{\hat{m}}(\tau \in \cdot)$, where $\hat{m}=m (1-\theta(m))<1$ and
\begin{equation}\labelPartI{hatmandm}
1-\hat{m} \sim m-1 \qquad \text{as } m\downarrow 1.
\end{equation}
\item \lbitem{GWTotalProg-gen} Uniformly over $m \in (0,\infty)$,
\begin{equation}\labelPartI{PGWTotalProgExact}
\P_m(\abs{\tau}=n)=\frac{1}{m \sqrt{2\pi n^3} }\e^{-(m-1-\log m) n} (1+o(1)).
\end{equation}
\item \lbitem{GWTotalProg-LargeConst}
Let $K, R\in (0,\infty)$.
There exists $c\in (0,\infty)$ such that for every $r\in \ocinterval{0,R}$ there is $n_0 \in \N$ such that for all $n\ge n_0$,
\begin{equation}
\P_{1-K/s_n}(rs_n^2 \le |\tau|\le 2rs_n^2) \ge \frac{c}{s_n \sqrt{r}}.
\end{equation}
\item \lbitem{GWCut}
Let $K,R \in (0,\infty)$.
There exist $n_0 \in \N$ and $\delta_1>0$ such that for all $r\in \ocinterval{0,R}$ and $n\ge n_0$,
\begin{equation}
\P_{1-K/s_n}(|\tau^{\sss(\le s_n)}|\ge rs_n^2) \ge \delta_1 \P_{1-K/s_n}(rs_n^2 \le |\tau|\le 2rs_n^2) .
\end{equation}

\item \lbitem{PGWSurvivalTime} Uniformly in $m\in (1,\infty)$,
\begin{equation}
\P_m\left(\abs{\tau^{\sss(\ge s_n)}} \ge 1\right) \le  \theta(m)+ O(1/s_n).
\end{equation}
\end{enumerate}
\end{prop}

\begin{proof}
\refitem{PGWSurvivalProb} Identity \eqrefPartI{thetaRecursion} is obtained by considering the individuals in the first generation.
\\
\refitem{PGWSurvivalAsymp} From \eqrefPartI{thetaRecursion}, we obtain \eqrefPartI{thetaAsymptotics}. Since $\theta(m) \le 1$, the uniform bounds follow.\\
\refitem{PGWdual} See for example Theorem~3.15 in \cite{Hof14pre}.
Asymptotic \eqrefPartI{hatmandm} follows from \eqrefPartI{thetaAsymptotics}.\\
\refitem{GWTotalProg-gen} It is well-known (see for example the first display on page 951 in \cite{AddGriKan12}) that
\begin{equation}
\P_m(\abs{\tau}=n)=\frac{\e^{-m n} (m n)^{n-1}}{n!}.
\end{equation}
Stirling's formula yields the claim.\\
\refitem{GWTotalProg-LargeConst} Choose $n_0$ so large that $1+o(1)$ in \eqrefPartI{PGWTotalProgExact} is bounded from below by $1/2$ and that $-K/s_n-\log(1-K/s_n) \le K^2/s_n^2$ and $\lceil rs_n^2\rceil \le \lfloor 2rs_n^2\rfloor$ for $n \ge n_0$.
According to \refitem{GWTotalProg-gen},
\begin{align}
\P_{1-K/s_n}(rs_n^2 \le |\tau|\le 2rs_n^2)
&\ge \sum_{i=\lceil rs_n^2\rceil}^{\lfloor 2rs_n^2\rfloor}
\frac{1}{(1-K/s_n) \sqrt{2\pi i^3}} \e^{-K^2 i/s_n^2} \frac{1}{2}
\notag \\
&\ge \frac{1}{2\sqrt{2^4 \pi r^3s_n^6}} \frac{s_n^2}{K^2} \left(\e^{-K^2\lceil rs_n^2\rceil/s_n^2} -\e^{-K^2\lfloor 2 rs_n^2\rfloor/s_n^2}\right).
\end{align}
\refitem{GWCut} Denote by $(\mathscr{G}_i,\P)$ the uniform labelled rooted tree on $i$ nodes after the labels of the children have been discarded and by $h(\mathscr{G}_i)$ the height of $\mathscr{G}_i$.
Given $d\in (0,\infty)$ and $\eps>0$, we write $\mathscr{G}_i^{\sss(\le d)}$ for the subtree of $\mathscr{G}_i$ consisting of vertices within distance $d$ from the root and $\mathscr{G}_i\eps$ for $\mathscr{G}_i$ with distances rescaled by $\eps$.
For all $m>0$ and $i \in \N$, the Poisson Galton--Watson tree with offspring mean $m$ conditioned on having $i$ nodes has the same distribution as $\mathscr{G}_i$.
Moreover, $h(\mathscr{G}_i)/\sqrt{i}$ converges, as $i\to \infty$, to the maximum of $2 B$ where $B=(B_t)_{t\in [0,1]}$ is a standard Brownian excursion \cite{Ald91,Ald93}.
We deduce that
\begin{align}
\P_{1-K/s_n}(\abs{\tau^{\sss(\le s_n)}}\ge rs_n^2)&= \sum_{i=\lceil rs_n^2 \rceil}^{\infty} \P(\abs{\mathscr{G}_i^{\sss(\le s_n)}}\ge rs_n^2)\P_{1-K/s_n}(\abs{\tau}=i)
\notag \\
& \ge \min_{rs_n^2 \le i \le 2rs_n^2} \P\left(h(\mathscr{G}_i) \le s_n\right) \P_{1-K/s_n}(rs_n^2 \le \abs{\tau} \le 2rs_n^2).\labelPartI{eq:heightCutGW}
\end{align}
The first factor on the right-hand side of \eqrefPartI{eq:heightCutGW}  can be estimated, for sufficiently large $n$, by
\begin{align}
\min_{rs_n^2 \le i \le 2rs_n^2} \P\left(h(\mathscr{G}_i)/\sqrt{i} \le s_n/\sqrt{i}\right)& \ge \min_{rs_n^2 \le i \le 2rs_n^2} \P\left(h(\mathscr{G}_i)/\sqrt{i} \le 1/\sqrt{2r}\right)
\notag \\
&\ge \frac{1}{2} \P\left(\max_{t\in [0,1]}2B_t \le 1/\sqrt{2R}\right)=\delta_1.
\end{align}
\refitem{PGWSurvivalTime} Using \refitem{PGWdual}, we obtain
\begin{align}
\P_m\left(\abs{\tau^{\sss(\ge s_n)}} \ge 1\right)
&=
\P_m(\abs{\tau}=\infty) +\P_m(\abs{\tau}<\infty) \P_m\condparentheses{\abs{\tau^{\sss(\ge s_n)}} \ge 1}{\abs{\tau}<\infty}
\notag\\&
=
\theta(m)+(1-\theta(m)) \P_{\hat{m}}(\abs{\tau^{\sss(\ge s_n)}} \ge 1)
\notag\\&
\le
\theta(m)+\P_1(\abs{\tau^{\sss(\ge s_n)}} \ge 1).
\end{align}
The claim now follows from a standard result on critical Galton--Watson processes (see for example \cite[Lemma~I.10.1]{Har63}).
\end{proof}

\subsection{The probability of luckiness: proof of \refprop{LuckyProb}}\lbsubsect{LuckyProbProof}

In the proof of \refprop{LuckyProb} we use the following estimates:
\begin{lemma}\lblemma{ReverseMarkov}
Suppose $0\leq X\leq Y$ are non-negative random variables such that $\condE{X}{Y}\geq pY$ a.s.
Then $\P(X\geq m)\geq \tfrac{1}{2}p\P(Y\geq 2m/p)$ for any $m\in \cointerval{0,\infty}$.
\end{lemma}
\begin{proof}
Markov's inequality applied to $Y-X$ gives $\condP{Y-X > (1-\tfrac{1}{2}p)Y}{Y} \leq (1-p)/(1-\tfrac{1}{2}p)$, so $\condP{X\geq\tfrac{1}{2}pY}{Y} \geq \frac{1}{2}p/(1-\tfrac{1}{2}p)\geq\tfrac{1}{2}p$ and the result follows.
\end{proof}
\begin{lemma}\lblemma{Integralfn}
There exists $n_0 \in \N$ such that for any $m\in (0,\infty)$, $a \in (0,1]$ and $n\ge n_0$,
\begin{equation}%\labelPartI{}
\int_0^1 f_n(amx)dx \leq \frac{f_n(1) (1-\deltaCondition)^{\epsilonCondition s_n}}{am}+\frac{f_n(m) a^{\epsilonCondition s_n}}{\epsilonCondition s_n}
\end{equation}
\end{lemma}
\begin{proof}
By \reflemma{ExtendedImpliesWeak}, for sufficiently large $n$,
\begin{align}
\int_0^1 f_n(amx)dx
&\leq
\frac{f_n(1-\deltaCondition)}{am}+\int_{1-\deltaCondition}^{am} f_n(y)\frac{dy}{am}
\notag\\&
\leq
\frac{f_n(1) (1-\deltaCondition)^{\epsilonCondition s_n}}{am}+\int_{0}^{am} \left(\frac{y}{m}\right)^{\epsilonCondition s_n} f_n(m)\frac{dy}{am}
.
\qedhere
\end{align}
\end{proof}
\begin{proof}[Proof of \refprop{LuckyProb}]
Applying \reflemma{Integralfn} with $m=1$ and $a=1-K/s_n$, we find that $K$ can be chosen such that for sufficiently large $n$,
\begin{equation}\labelPartI{SumOfsnEdgeWeights}
\int_0^1 f_n((1-K/s_n)x)\,dx \leq \frac{f_n(1)}{2s_n}.
\end{equation}

Let $\tau$ be the subtree of $\tree$ consisting of $v$ and those descendants joined to $v$ by a path whose edge weights $X_w$ all satisfy $X_w\leq 1-K/s_n$.
We consider $\tau$ as a rooted labelled tree equipped with edge weights, but with the vertex labels from the PWIT forgotten\footnote{Formally, we should consider instead of $\tau$ the subset $\tilde{\tau}$ where we replace each vertex $w\in\tau\setminus\set{v}$ by an arbitrary label $\ell(w)$ drawn independently from some continuous distribution.
By a slight abuse of notation, we will refer to $\tau$ and $X_w$, $w\in\tau\setminus\set{v}$ instead of $\tilde{\tau}$ and $X_{\ell^{-1}(w)}$, $w\in\tilde{\tau}$.
This procedure avoids the complication, implicit in our Ulam--Harris notation, that the vertex $w=\emptyset_j k_1 k_2\dotsc k_r\in\tree^{\sss(j)}$ automatically gives information about the number of its siblings with smaller edge weights.}.
Then, ignoring weights, $\tau$ is equal in distribution to a Poisson Galton--Watson tree with offspring mean $1-K/s_n$ and by \refprop{PGWtrees}~\refitem{GWTotalProg-LargeConst} there is some $c>0$ independent of $r \in (0,R]$ such that for every $r\in (0,R]$ there is $n_0 \in \N$ with $\P(4r s_n^2\leq \abs{\tau} \leq 8r s_n^2) \geq \frac{c}{s_n\sqrt{r} }$ for all $n\ge n_0$.
Write $\tau'$ for the subtree of $\tau$ consisting of vertices within distance $s_n$ from the root.
By \refprop{PGWtrees}~\refitem{GWCut},
there is some $\delta_1>0$ independent of $r \in (0,R]$ such that with $\delta=\delta_1 c/4$, $\P(\abs{\tau'}\geq 4rs_n^2) \geq \frac{4\delta}{s_n\sqrt{r}}$ for $n$ sufficiently large.
Conditional on $\tau'$, the PWIT edge weights $X_{wk}$, $wk\in\tau'$, are uniformly distributed on $[0,1-K/s_n]$, and therefore the first passage edge weights $Y_{wk}=f_n(X_{wk})$ satisfy
\begin{equation}\labelPartI{SingleEdgeWeightConditionedPWIT}
\condE{Y_{wk}}{wk\in\tau'} = \int_0^1 f_n((1-K/s_n)x)dx.
\end{equation}
By the definition of $\tau'$, the graph distance between $u\in \tau'$ and $v$ is at most $s_n$.
Hence, \eqrefPartI{SingleEdgeWeightConditionedPWIT} and \eqrefPartI{SumOfsnEdgeWeights} give $\condE{T_u-T_v}{u\in\tau'}\leq s_n\frac{f_n(1)}{2s_n}$.
By Markov's inequality $\condP{T_u-T_v>f_n(1)}{u\in\tau'}\leq \frac{1}{2}$.

Let $\tau''$ denote the subtree $\set{u\in\tau'\colon T_u-T_v\leq f_n(1)}$, so that $\condP{u\in\tau''}{u \in \tau'}\geq\frac{1}{2}$ and $\condE{\abs{\tau''}}{\abs{\tau'}}\geq\frac{1}{2}\abs{\tau'}$.
By \reflemma{ReverseMarkov} with $X=\abs{\tau''}$, $Y=\abs{\tau'}$, we obtain
\begin{equation*}
\P\left(\bigabs{\BP_{f_n(1)}^{\sss(v)}}\geq r s_n^2\right)\ge \P\left(\abs{\tau''}\geq rs_n^2\right)\geq\frac{1}{4}\P\left(\abs{\tau'}\geq 4rs_n^2\right)\ge \frac{\delta}{s_n\sqrt{r}}.
\qedhere
\end{equation*}
\end{proof}

In the proof of \refthm{FrozenCluster}~\refitem{FrozenVolume}, we will bound the number of vertices of large age in the frozen cluster.
To do so, we will argue that for each vertex of large age, there is an independent chance to have an $R$-lucky child $v$, where $R$ is chosen large enough that $v$ and its descendants are on their own sufficient to bring about the freezing time.
This motivates the following definition:

\begin{defn}\lbdefn{Lucky}
Let $\epsilon_1$ denote the constant from \reflemma{ModerateAgeContribution} for $K=1$.
Call a vertex $v \in\tree^{\sss (j)}$ \emph{lucky} if it is $(1/\eps_1)$-lucky, and set
\begin{equation}\labelPartI{LuckyTimeDefn}
T^{\sss(j)}_\lucky=\inf\set{T_v\colon v\in\tree^{\sss(j)}\setminus\set{\emptyset_j}\text{ is lucky and }T_v> T_{\parent{v}}+f_n(1)},
\end{equation}
the first time that a lucky vertex is born to a parent of age greater than $f_n(1)$.
\end{defn}
In view of \refdefn{Freezing} and \reflemma{ModerateAgeContribution}, we have
\begin{equation}\labelPartI{FreezingFromLucky}
v\in\tree^{\sss(j)}\text{ is lucky} \quad \implies \quad T_\fr^{\sss(j)}\leq T_v+f_n(1).
\end{equation}
In other words, a lucky vertex has enough descendants in time $f_n(1)$ that the integral in the definition \eqrefPartI{TfrDefn} of the freezing time must be at least $s_n$.

\begin{lemma}\lblemma{SumfninverseExponential}
The distribution of
\begin{equation}\labelPartI{Sumfninverse}
\sum_{v\in\BP^{\sss(j)}_{T^{\sss(j)}_\lucky}} \left( f_n^{-1}\left( T^{\sss(j)}_\lucky-T_v \right)-1 \right)^+
\end{equation}
is exponential with rate $\P(v\text{ is lucky})$.
\end{lemma}
Since $T_v-T_{\parent{v}}=f_n(X_v)$, the condition $T_v>T_{\parent{v}}+f_n(1)$ in the definition of $T^{\sss(j)}_\lucky$ is equivalent to $X_v>1$.
On the other hand, the event $\set{v\text{ is lucky}}$ depends only on the evolution of $\BP^{\sss(v)}$ until time $f_n(1)$ and is therefore determined by those descendants $v'$ of $v$ for which $X_{v'}\leq 1$.
Because these two conditions on edge weights are mutually exclusive, it will follow that $T^{\sss(j)}_\lucky$ is the first arrival time of a certain Cox process.
We now formalize this intuition, which requires some care.

\begin{proof}[Proof of \reflemma{SumfninverseExponential}]
To avoid complications arising from our Ulam--Harris notation, we consider instead of a vertex $v=\emptyset_j k_1 k_2\dotsc k_r$ the modified vertex $w(v)=\emptyset_j X_{k_1} X_{k_1 k_2} \dotsc X_{k_1 k_2 \dotsc k_r}$ formed out of the edge weights along the path from $\emptyset_j$ to $v$.
We can extend our usual notation for parents, length, concatenation, edge weight, and birth times to vertices of the form $w=\emptyset_j x_1 \dotsb x_r$, $x_i\in(0,\infty)$: for instance, $\abs{w}=r$, $X_w=x_r$ and $T_w=f_n(x_1)+\dotsb+f_n(x_r)$.

Form the point measure $\cM=\sum_{v\in\tree} \delta_{w(v)}$ on $\union_{r=0}^\infty \set{\emptyset_1,\emptyset_2}\times (0,\infty)^r$.
Given $\cM$, we can recover the PWIT $(\tree,X)$: for instance, $X_{\emptyset_j k_1}=\inf\set{x>0\colon \cM(\set{\emptyset_j}\times \ocinterval{0,x})\geq k_1}$ and $X_{\emptyset_j k_1 k_2}=\inf\set{x>0\colon \cM(\set{\emptyset_j}\times \shortset{X_{\emptyset_j k_1}}\times\ocinterval{0,x})\geq k_2}$.
The point measure $\cM$ has the advantage that a value such as $\cM(\set{\emptyset_j}\times(a,b))$ (the number of children of $\emptyset_j$ with edge weights in the interval $(a,b)$) does not reveal information about the number of sibling edges of smaller edge weight.

The Poisson property of the PWIT can be expressed in terms of $\cM$ by saying that, conditional on the restriction $\cM\big\vert_{\set{\emptyset_1,\emptyset_2}\times(0,\infty)^r}$ to the first $r$ generations, the $(r+1)^\st$ generation $\cM\big\vert_{\set{\emptyset_1,\emptyset_2}\times(0,\infty)^{r+1}}$ is formed as a Cox process with intensity $\cM\big\vert_{\set{\emptyset_1,\emptyset_2}\times(0,\infty)^r}\otimes \indicator{x>0} dx$, where $\indicator{x>0}dx$ denotes Lebesgue measure on $(0,\infty)$.

We next rearrange the information contained in $\cM$.
Given $w=\emptyset_j x_1 \dots x_r$, let $\cM^{\sss(w)}=\sum_{w'\colon \cM(\set{ww'})=1}\delta_{w'}$ denote the point measure corresponding to all descendants of $w$ (thus $\cM^{\sss(w)}=0$ if $\cM(\set{w})=0$, and the point measures $\cM^{\sss(w(v))}$, $v\in\tree$, are identically distributed and non-zero)\footnote{In this instance, we assume that $w'$ does not have a $\varnothing_j$ symbol for convenience of notation.}.
Further, write $\cM^{\sss(w)}_{\sss\leq}$ for the restriction of $\cM^{\sss(w)}$ to $\union_{k=0}^\infty \ocinterval{0,1}^k$.
Since luckiness only depends on descendants explored within time $f_n(1)$, it follows that whether or not $v$ is lucky can be determined solely in terms of $\cM^{\sss(w(v))}_{\sss\leq}$.
Indeed, call a point measure $\mathfrak{m}$ on $\union_{k=0}^\infty \ocinterval{0,1}^k$ \emph{lucky} if $\mathfrak{m}(\set{w'\colon T_{w'}\leq f_n(1)}) \geq s_n^2/\epsilon_1$; then $v$ is lucky if and only if $\cM^{\sss(w(v))}_{\sss\leq}$ is lucky.

Define $A^{\sss(j)}$ to be the collection of vertices $w(v)$, $v\in\tree^{\sss(j)}$, that are born before time $T_\lucky^{\sss(j)}$.
That is, for any non-root ancestor $v'$ of $v$, including $v$ itself, it is \emph{not} the case that $T_{v'}>T_{\parent{v'}}+f_n(1)$ and $v'$ is lucky.
To study $A^{\sss(j)}$, we decompose the PWIT according to vertices $v'$ that are born late (i.e., $T_{v'}>T_{\parent{v'}}+f_n(1)$) and keep track of their early explored descendants (i.e., $\cM_{\le}^{\sss(w(v'))}$).

For $w=\emptyset_j x_1\dots x_r$, let $i_1<\dots<i_k$ denote those indices (if any) for which $x_i>1$, and write $w_\ell=\emptyset_j x_1\dots x_{i_\ell}$, $\ell=1,\dotsc,k$.
Set $q(w,\cM)$ to be the sequence $\cM^{\sss(\emptyset_j)}_{\sss\leq} \cM^{\sss(w_1)}_{\sss\leq} \dots \cM^{\sss(w_k)}_{\sss\leq}$.
Define the rearranged point measure
\begin{equation}%\labelPartI{}
\cR =
\int_{\set{\emptyset_1,\emptyset_2}\union\bigunion_{k=1}^\infty \set{\emptyset_1,\emptyset_2}\times(0,\infty)^{k-1}\times(1,\infty)} \delta_{(w,q(w,\cM))} d\cM(w).
\end{equation}
(That is, $\cR$ is a point measure on pairs $(w,q)$ such that $w$ satisfies $w\in\set{\emptyset_1,\emptyset_2}$ or $X_w>1$, and $q$ is a sequence of measures on $\union_{r=0}^\infty \ocinterval{0,1}^r$.
Considering $\cR$ instead of $\cM$ corresponds to partitioning vertices according to their most recent ancestor (if any) having edge weight greater than 1.)

The Poisson property of the PWIT implies that, conditional on the restriction $\cR\big\vert_{\set{(w,q)\colon \abs{w}\leq r}}$, the restriction $\cR\big\vert_{\set{(w,q)\colon \abs{w}=r+1}}$ forms a Cox process with intensity measure
\begin{equation}\labelPartI{cRintensity}
\int_{\set{(w,q)\colon \abs{w}\leq r}}  d\cR(w,q) \int_{\set{w'\colon \abs{ww'}=r}} d\cM_{\le}^{\sss(w)}(w') \left[ \delta_{ww'}\otimes\indicator{x>1}dx \right] \otimes \left[ \delta_q \otimes d\P(\cM_\le^{\sss(\emptyset_j)}\in\cdot) \right],
\end{equation}
where $\delta_{ww'}\otimes\indicator{x>1}dx$ means the image of Lebesgue measure on $(1,\infty)$ under the concatenation mapping $x\mapsto ww'x$.
The formula \eqrefPartI{cRintensity} expresses the fact that every vertex in the $(r+1)^\st$ generation has a parent uniquely written as $ww'$, with $(w,q(w,\cM))$ corresponding to a point mass in $\cR$, $w'$ corresponding to a point mass in the last entry $\cM_{\le}^{\sss(w)}$ of the sequence $q(w,\cM)$, and $r=\abs{w}+\abs{w'}$.

Now, it is easy to verify that $A^{\sss(j)}$ is measurable with respect to the restriction of $\cR$ to pairs $(w,q)$ such that $q=\mathfrak{m}_0\dots \mathfrak{m}_k$ with $\mathfrak{m}_\ell$ not lucky for each $\ell\neq 0$.

Because of the Poisson property of the PWIT, as expressed via $\cR$ in \eqrefPartI{cRintensity}, it follows that, conditional on $A^{\sss(j)}$, the point measure $L=\sum_{w''\notin A^{\sss(j)},\parent{w''}\in A^{\sss(j)}} \delta_{T_{w''}}$ forms a Cox process on $(0,\infty)$.
Furthermore, the first point of $L$ is precisely $T^{\sss(j)}_\lucky$.
To determine the intensity measure of $L$, we note that for a vertex $ww' \in A^{\sss(j)}$, $w''=ww'x$ satisfies $ w'' \not\in A^{\sss(j)}$ if and only if $X_{w''}>1$ and $\cM_\le^{\sss(w'')}$ is lucky.
Furthermore, the condition $T_{w''}=T_{ww'}+f_n(X_{w''})\leq t$ is equivalent to $X_{w''}\leq f_n^{-1}(t-T_{ww'})$.
Using \eqrefPartI{cRintensity}, it follows that the cumulative intensity measure of $L$ on $\ocinterval{0,t}$ is given by
\begin{equation}\labelPartI{SumfninversecR}
\int_{\set{\substack{(w,q)\colon q=\mathfrak{m}_0\dots\mathfrak{m}_k,\\ \mathfrak{m}_\ell\text{ not lucky for any }\ell\neq 0}}} d\cR(w,q) \int_{\set{w'\colon T_{w w'}\leq t}} d\cM_\le^{\sss(w)}(w') \left( f_n^{-1}(t-T_{w w'})-1 \right)^+ \cdot \P(\cM_{\le}^{\sss(\emptyset_j)}\text{ is lucky}).
\end{equation}
The vertices $w w'$ from the integral in \eqrefPartI{SumfninversecR} are in one-to-one correspondence with the vertices $u\in A^{\sss(j)}$ satisfying $T_u\leq t$.
Consequently we may re-write the cumulative intensity as
\begin{equation}\labelPartI{SumfninverseA}
\sum_{u\in A^{\sss(j)}\colon T_u\leq t} \left( f_n^{-1}(t-T_u)-1 \right)^+ \cdot \P(u\text{ is lucky}).
\end{equation}
Finally we note that $\P(v\text{ is lucky})$ times the sum in \eqrefPartI{Sumfninverse} is exactly the sum in \eqrefPartI{SumfninverseA} evaluated at $t=T^{\sss(j)}_\lucky$.
(The vertex $v$ for which $T_v=T^{\sss(j)}_\lucky$ does not contribute to \eqrefPartI{Sumfninverse}.)
The cumulative intensity in \eqrefPartI{SumfninverseA} is a.s.\ continuous as a function of $t$ (since $f_n^{-1}$ is continuous and the jumps at the times $T_u$ are zero).
But for any Cox process with continuous cumulative intensity function and infinite total intensity, it is elementary to verify that when the cumulative intensity is evaluated at the first point of the Cox process, the result is exponential with mean $1$.
This completes the proof.
\end{proof}

\subsection{Volume of the frozen cluster: proof of \refthm{FrozenCluster}~\refitem{FrozenVolume}}
\lbsubsect{VolClusterFr}

To study the frozen cluster, we introduce the frozen intensity measures

\begin{equation}\labelPartI{FrozenIntensity}
d\mu_{n,\fr}^{\sss(j)}(y) = \sum_{v\in\cluster_\fr^{\sss(j)}} \indicator{y\geq 0} \mu_n\bigl(T_\fr^{\sss(j)}-T_v + dy\bigr).
\end{equation}
Recall from \refsect{densityBoundsfn} that the notation $\mu(t_0+dy)$ denotes the translation of the measure $\mu$ by $t_0$; thus \eqrefPartI{FrozenIntensity} means that, for a test function $h\geq 0$,
\begin{equation}\labelPartI{FrozenIntensityTestFunction}
\int h(y) d\mu_{n,\fr}^{\sss(j)}(y) = \sum_{v\in\cluster_\fr^{\sss(j)}} \int_{T_\fr^{\sss(j)}-T_v}^\infty h\left( y-(T_\fr^{\sss(j)}-T_v) \right) d\mu_n(y).
\end{equation}
\begin{lemma}\lblemma{ExpectedUnfrozenChildren}
Almost surely, for $j=1,2$,
\begin{equation}%\labelPartI{ExpectedUnfrozenChildren}
s_n\leq
\int \e^{-\lambda_n y} d\mu_{n,\fr}^{\sss(j)}(y)
\leq s_n+1.
\end{equation}
\end{lemma}
\begin{proof}
By \eqrefPartI{FrozenIntensityTestFunction},
\begin{equation}\labelPartI{FrozenIntensityExponential}
\int \e^{-\lambda_n y} d\mu_{n,\fr}^{\sss(j)}(y)=\sum_{v\in\cluster_\fr^{\sss(j)}} \int_{T_\fr^{\sss(j)}-T_v}^\infty \e^{-\lambda_n\left(y-(T_\fr^{\sss(j)}-T_v)\right)} d\mu_n(y).
\end{equation}
The expression in \eqrefPartI{FrozenIntensityExponential} is the value of the process $\sum_{v\in\BP_t^{\sss(j)}}\int_{t-T_v}^\infty \e^{-\lambda_n(y-(t-T_v))} d\mu_n(y)$ from \refdefn{Freezing}, stopped at $t=T_\fr^{\sss(j)}$ (recall that $\cluster_\fr^{\sss(j)}=\BP_{T_\fr^{\sss(j)}}^{\sss(j)}$).
Since $\mu_n$ has no atoms, this process is continuous in $t$ except for jumps at the birth times, and since the birth times are distinct a.s., the corresponding jump has size $\int_0^\infty \e^{-\lambda_ny}d\mu_n(y)=1$.
By definition, $T_\fr^{\sss(j)}$ is the first time the process in \eqrefPartI{FrozenIntensityExponential} exceeds $s_n$, so it can have value at most $s_n+1$ at that time.
\end{proof}

To prove \refthm{FrozenCluster}~\refitem{FrozenVolume}, we will separate vertices $v\in\cluster_\fr^{\sss(j)}$ according to whether their age at freezing, $T_\fr^{\sss(j)}-T_v$, is large or small.
We then use \reflemma{SumfninverseExponential} and \reflemma{ExpectedUnfrozenChildren}, respectively, to bound the number of such vertices.

\begin{proof}[Proof of \refthm{FrozenCluster}~\refitem{FrozenVolume}]
\reflemma{SumfninverseExponential} and \refprop{LuckyProb} imply that the sum in \eqrefPartI{Sumfninverse} is $O_\P(s_n)$.
By \eqrefPartI{FreezingFromLucky}, any vertex $v$ with $T_{\fr}^{\sss (j)}-T_v \ge f_n(1)+f_n(1+1/s_n)$ satisfies $T^{\sss(j)}_{\mathrm{lucky}}-T_v\ge f_n(1+1/s_n)$ and therefore must contribute at least $1/s_n$ to the sum in \eqrefPartI{Sumfninverse}.
Consequently there can be at most $O_\P(s_n^2)$ vertices of age at least $f_n(1)+f_n(1+1/s_n)$ in $\cluster^{\sss(j)}_\fr$.

For the vertices of small age, recall the definition of $\mu_{n,\fr}^{\sss(j)}$ from \eqrefPartI{FrozenIntensity} and that $\int \e^{-\lambda_n y} d\mu_{n,\fr}^{\sss(j)}(y)\leq s_n+1$ by \reflemma{ExpectedUnfrozenChildren}.
Apply \reflemma{ModerateAgeContribution} with $K$ chosen large enough that $Kf_n(1)\geq f_n(1+1/s_n)+f_n(1)$ for each $n$ (such a $K$ exists since $\lim_{n \to \infty} f_n(1+1/s_n)/f_n(1)=\e$), to see that summands corresponding to vertices of age at most $K f_n(1)$ in $\cluster_\fr^{\sss(j)}$ contribute at least $\eps_K/s_n$ to the integral in \reflemma{ExpectedUnfrozenChildren}.
Hence, there can be at most $s_n(s_n+1)/\epsilon_K$ such vertices.
\end{proof}

For future reference, we now state a lemma showing that most of the mass of the frozen intensity measures $\mu_{n,\fr}^{\sss(j)}$ comes from small times:

\begin{lemma}\lblemma{FrozenVerticesAreYoung}
Let $\delta,\delta'>0$ be given.
Then there exists $K<\infty$ and $n_0 \in \N$ such that, for all $n\ge n_0$,
\begin{equation}%\labelPartI{}
\P\left(
\int \e^{-\lambda_n y} \indicator{\lambda_ny\geq K} d\mu_{n,\fr}^{\sss(j)}(y)
>\delta s_n
\right)\leq \delta'.
\end{equation}
\end{lemma}

\begin{proof}
Let $\epsilon= \e^{-\gamma}/2$, where $\gamma$ denotes Euler's constant.
Using the definition of $\mu_{n,\fr}^{\sss(j)}$ from \eqrefPartI{FrozenIntensity}, and \reflemma{lambdanAsymp}, we obtain $n_0 \in \N$ such that for all $K<\infty$, and $n\ge n_0$,
\begin{equation}\labelPartI{MassUpperTailFrozenIntensity}
\int \e^{-\lambda_n y} \indicator{\lambda_n y\geq K} d\mu_{n,\fr}^{\sss(j)}(y) \le
\sum_{v \in \cluster_\fr^{\sss(j)}} \int \e^{-\epsilon y/f_n(1)} \indicator{y\ge Kf_n(1)} \mu_n(T_\fr^{\sss(j)}-T_v+dy).
\end{equation}
According to \reflemma{BoundOnContribution}, for any $\epsilon'>0$ we can choose some $K<\infty$ such that, after possibly increasing $n_0$, the right-hand side of \eqrefPartI{MassUpperTailFrozenIntensity} is bounded from above by $\abs{\cluster_\fr^{\sss(j)}} \epsilon'/s_n$. Since $\abs{\cluster_\fr^{\sss(j)}}=O_\P(s_n^2)$ by \refthm{FrozenCluster}~\refitem{FrozenVolume}, the proof is complete.
\end{proof}

\subsection{Emergence of \texorpdfstring{$R$}{R}-lucky vertices: Proof of \reflemma{RluckyBornSoon}}
\lbsubsect{RLuckyBirthProof}

In this section we show how to express $\cluster_t\setminus\cluster_\fr$, $t\geq T_\unfr$, as a suitable union of branching processes.
This representation will be useful in the proof of \reflemma{RluckyBornSoon}.

Consider the immediate children $v\in\boundary\cluster_\fr$ of individuals in the frozen cluster $\cluster_\fr$.
Then, for $t'\geq 0$,
\begin{equation}\labelPartI{UnfrozenUnion}
\cluster_{T_\unfr+t'}\setminus\cluster_\fr = \bigunion_{v\in\boundary\cluster_\fr\colon T_v^\cluster\leq T_\unfr+t'} \set{vw\colon w\in\BP^{\sss(v)}_{t'+T_\unfr-T_v^\cluster}},
\end{equation}
where $\BP^{\sss(v)}$ denotes the branching process of descendants of $v$, re-rooted and time-shifted as in \eqrefPartI{BPvDefinition}.
Furthermore, conditional on $\cluster_\fr$, the children $v\in\boundary\cluster_\fr$ appear according to a Cox process.
Formally, the point measures
\begin{equation}\labelPartI{UnfrozenChildrenCox}
\cP_{n,\unfr}^{\sss (j)}=\sum_{v\in\boundary\cluster_\fr^{(j)}} \delta_{(T_v^\cluster-T_\unfr,\BP^{(v)})}
\end{equation}
form Cox processes with intensities $d\mu_{n,\fr}^{\sss(j)} \otimes d\P(\BP^{\sss(1)}\in\cdot)$, $j=1,2$, where the frozen intensity measures $\mu_{n,\fr}^{\sss (j)}$ were introduced in \eqrefPartI{FrozenIntensity}.

\begin{proof}[Proof of \reflemma{RluckyBornSoon}]
Let $\epsilon>0$ be given.
We first claim that there exists $r>0$ and $C<\infty$ such that $T_{v_{j,r}}^\cluster\leq T_\unfr+Cf_n(1)$ with probability at least $1-\epsilon$ for $n$ sufficiently large.
To prove this, apply \reflemma{ExpectedUnfrozenChildren} and \reflemma{FrozenVerticesAreYoung} (with $\delta=\tfrac{1}{2}$, $\delta'=\eps/2$) to find $K<\infty$ such that
\begin{equation}\labelPartI{munFrozenNotSmall}
\mu_{n,\fr}^{\sss(j)}(0, K/\lambda_n)
\geq
\int\e^{-\lambda_ny}\indicator{\lambda_ny\leq K}d\mu_{n,\fr}^{\sss(j)}(y)
\geq
s_n - \int\e^{-\lambda_ny}\indicator{\lambda_ny> K}d\mu_{n,\fr}^{\sss(j)}(y)
\geq
\tfrac{1}{2}s_n
\end{equation}
with probability at least $1-\tfrac{1}{2}\epsilon$ for $n$ sufficiently large.
Use \reflemma{lambdanAsymp} to choose $C<\infty$ such that $K/\lambda_n\leq Cf_n(1)$ for $n$ large enough.
Conditional on $\cluster_\fr$, each vertex $v\in\boundary\cluster_\fr^{\sss(j)}$ has an independent chance of being $r$-lucky, so the number of $r$-lucky vertices $v\in\boundary\cluster_\fr^{\sss(j)}$ with $T_v\in(T_\unfr, T_\unfr+Cf_n(1))$ is Poisson with mean $\mu_{n,\fr}^{\sss(j)}(0,Cf_n(1))\P(v\text{ is $r$-lucky})$.
By taking $r$ sufficiently small, \refprop{LuckyProb} allows us to make $(\tfrac{1}{2}s_n)\P(v\text{ is $r$-lucky})$ large enough (uniformly in $n$ sufficiently large) so that, by \eqrefPartI{munFrozenNotSmall}, the number of such vertices $v$ will be positive with probability at least $1-\epsilon$.
This proves the claim.

It now suffices to show that $T_{v_{j,R}}^\cluster-T_{v_{j,r}}^\cluster=O_\P(f_n(1))$.
The argument is very similar to the previous one, except that now we have access to at least $r s_n^2$ vertices whose ages are known to be small.

Let $\cL_r^{\sss(j)}$ denote the collection of descendants $v$ of $v_{j,r}$ such that $T_v^\cluster-T_{v_{j,r}}^\cluster\leq f_n(1)$, so that $\abs{\cL_r^{\sss(j)}}\geq rs_n^2$ by definition.
Conditional on $\cL_r^{\sss(j)}$, for every $K' \in (0,\infty)$, the number of children $w$ of vertices $v\in\cL_r^{\sss(j)}$ such that $T_w^\cluster-T_v^\cluster\in(f_n(1),f_n(1+K'/s_n))$ is Poisson with mean at least $(rs_n^2)(K'/s_n)$.
Given $\epsilon>0$, \refprop{LuckyProb} allows us to choose $K'$ large enough, so that at least one such vertex will be $R$-lucky with probability at least $1-\epsilon-o(1)$.
By \refcond{scalingfn}, $f_n(1+K'/s_n)=O(f_n(1))$, and this completes the proof.
\end{proof}

\section{IP and the geometry of the frozen cluster}\lbsubsect{FrozenVsIP}

In this section, we compare the frozen cluster to the IP cluster $\IP^{\sss(j)}(\infty)$ -- the set of all vertices ever invaded in the IP process.
The structure of the IP cluster is encoded in a single infinite backbone and an associated process of maximum weights, with off-backbone branches expressed in terms of Poisson Galton-Watson trees.
See \refprop{IPstructure} below.

The proofs in this section rely on detailed comparisons between the frozen cluster and the part of the IP cluster within distance of order $s_n$ from the root.
Specifically, we show that (a) the freezing time can be effectively bounded by the time $T_{V^{\sss\BB,j}_{\floor{K s_n}}}$ when the first passage exploration process first explores to distance $Ks_n$ along the IP backbone (\reflemma{FreezingByHeight}); (b) the time to complete this exploration is comparable to the largest first passage weight along the path (\reflemma{ExtraTimeToFreeze}); and (c) the likelihood of exploring very long paths that do not belong to the IP cluster is moderate.
Assertions (a) and (b) will allow us to prove \refthm{TfrScaling}, and assertion (c) will be made precise in \reflemma{ExpExplUninvChildrenBound} and \reflemma{NoFastLongPaths} and the proof of \refthm{FrozenCluster}~\refitem{FrozenDiameter}.

For the proof of Theorems~\refPartI{t:TfrScaling} and \refPartI{t:FrozenCluster}~\refitem{FrozenDiameter}, we will obtain bounds on $\BP_T^{\sss(j)}$ for certain random times $T$ satisfying $T\geq T_\fr^{\sss(j)}$ with large probability.
Since $\cluster^{\sss(j)}$ and $\BP^{\sss(j)}$ evolve in the same way until the time $T_\fr^{\sss(j)}$, such bounds will apply \emph{a fortiori} to $\cluster_\fr^{\sss(j)}$, and we write $\BP_\fr^{\sss(j)}=\cluster_\fr^{\sss(j)}=\cluster_{T_\fr^{\sss(j)}}^{\sss(j)}=\BP^{\sss(j)}_{T_\fr^{\sss(j)}}$.

Throughout this section, we assume Conditions~\refPartI{cond:scalingfn}, \refPartI{cond:LowerBoundfn} and \refPartI{cond:boundfnExtended}.
We will reserve the notation $\epsilonCondition, \deltaCondition$ for some fixed choice of the constants in Conditions~\refPartI{cond:LowerBoundfn} and \refPartI{cond:boundfnExtended}, with $\epsilonCondition$ chosen small enough to satisfy both conditions.

\subsection{Structure and scaling of the IP cluster}\lbsubsect{IPstructure}

Our description of $\IP^{\sss(j)}(\infty)$ is based on \cite{AddGriKan12}, which examines the structure of the IP cluster on the PWIT, and the scaling limit results in \cite{AGdHS2008}, which proves similar results for regular trees.
As remarked in \cite{AddGriKan12}, the scaling limit results of \cite{AGdHS2008} can be transferred to the PWIT without difficulty.

To describe the structure of the IP cluster, we first define the {\em backbone:}

\begin{defn}\lbdefn{Backbone}
The \emph{backbone} of the IP cluster $\IP^{\sss(j)}(\infty)$ is the unique infinite oriented path in $\IP^{\sss(j)}(\infty)$ starting at the root.
That is, the backbone is the unique (random) sequence of vertices $V^{\sss\BB,j}_0,V^{\sss\BB,j}_1,\dotsc\in\IP^{\sss(j)}(\infty)$ with $V^{\sss\BB,j}_0=\emptyset_j$ and $\parent{V^{\sss\BB,j}_k}=V^{\sss\BB,j}_{k-1}$ for all $k\in\N$.

The PWIT edge weight between $V^{\sss\BB,j}_{k-1}$ and $V^{\sss\BB,j}_k$ is denoted $X^{\sss\BB,j}_k$.
We define the forward maximum $M_k^{\sss(j)}$ by
\begin{equation}\labelPartI{MkDefinition}
M^{\sss(j)}_k = \sup_{i>k} X_i^{\sss\BB,j}.
\end{equation}

The off-backbone branch at height $k$ means the subtree of $\IP^{\sss(j)}(\infty)$ consisting of vertices that are descendants of $V^{\sss\BB,j}_k$ but not descendants of $V^{\sss\BB,j}_{k+1}$, and is denoted by $\tau_k$.
We consider $\tau_k$ as a rooted labelled tree, but with the edge weights and vertex labels from the PWIT forgotten\footnote{As in the proof of \refprop{LuckyProb}, we should consider instead of $\tau_k$ the set $\tilde{\tau}_k$ where we replace each vertex $v\in\tau_k\setminus\set{V^{\sss\BB,j}_k}$ by an arbitrary label $\ell(v)$ drawn independently from some continuous distribution.
By a slight abuse of notation, we will refer to $\tau$ and $X_v$, $v\in\tau_k\setminus\set{V^{\sss\BB,j}_k}$ instead of $\tilde{\tau}_k$ and $X_{\ell^{-1}(v)}$, $v\in\tilde{\tau}_k$.
This procedure avoids the complication, implicit in our Ulam--Harris notation, that the vertex $v=\emptyset_j k_1 k_2\dotsc k_r\in\tree^{\sss(j)}$ automatically gives information about the number of its siblings with smaller edge weights.}.
\end{defn}

In the notation of \refdefn{Backbone}, the maximum invaded edge weight $M^{\sss(j)}$ from \eqrefPartI{MDefinition} is now $M^{\sss(j)}_0$.
(This amounts to the observation, elementary to verify, that the largest edge weight $M^{\sss(j)}$ must occur as one of the backbone edge weights $X^{\sss\BB,j}_k$.)

\begin{prop}[\cite{AddGriKan12,AGdHS2008}]\lbprop{IPstructure}
The backbone is well-defined, and $M^{\sss(j)}_k>1$ for each $k$, a.s.
Furthermore:
\begin{enumerate}
\item\lbitem{BBmaxAttained}
The maximum in \eqrefPartI{MkDefinition} is attained uniquely, for each $k$, a.s.
Writing $I_k$ for the random height at which the maximum in \eqrefPartI{MkDefinition} is attained, it holds that $I_k=O_\P(k)$ for each $k\geq 1$.
\item\lbitem{BBMarkov}
The sequence $(M^{\sss(j)}_k)_{k=0}^\infty$ is non-increasing and forms a Markov chain with initial distribution $\P(M^{\sss(j)}_0\leq m)=\theta(m)$ and transition mechanism
\begin{equation}\labelPartI{MarkovForwardMax}
\begin{aligned}
\condP{M^{\sss(j)}_{k+1}=m}{M^{\sss(j)}_k=m} &= m(1-\theta(m)), && m>1, \\
\condP{M^{\sss(j)}_{k+1} < m'}{M^{\sss(j)}_k=m} &= \frac{\theta(m')}{\theta(m)}(1-m(1-\theta(m))), && 1\leq m'\leq m.
\end{aligned}
\end{equation}
\item\lbitem{BBscaling}
$M^{\sss(j)}_k=1+\Theta_\P(1/k)$ and indeed $k(M^{\sss(j)}_k-1)$ converges weakly to an exponential distribution with mean $1$ as $k\to\infty$.
\item\lbitem{BBdual}
Conditional on $(M^{\sss(j)}_k)_{k=0}^\infty$, the off-backbone branches $(\tau_k)_{k=0}^\infty$ are distributed as subcritical Poisson Galton--Watson trees with means $M^{\sss(j)}_k(1-\theta(M^{\sss(j)}_k))$, conditionally independent (but not identically distributed) for each $k$.
\item\lbitem{BBweights}
Conditional on $(M^{\sss(j)}_k)_{k=0}^\infty$, the PWIT edge weight $X^{\sss\BB,j}_k$ either (i) equals $M^{\sss(j)}_{k-1}$, if $M^{\sss(j)}_k<M^{\sss(j)}_{k-1}$; or (ii) has the Uniform$[0,M^{\sss(j)}_k]$ distribution, if $M^{\sss(j)}_k=M^{\sss(j)}_{k-1}$.
Furthermore the weights are conditionally independent for each $k$.
\item\lbitem{OffBBweights}
Conditional on $(M^{\sss(j)}_k)_{k=0}^\infty$ and $(\tau_k)_{k=0}^\infty$, the PWIT edge weight of an edge between two vertices of $\tau_k$ has the Uniform$[0,M^{\sss(j)}_k]$ distribution, conditionally independent over the choice of edge and of $k$.
\item\lbitem{OffBBboundaryWeights}
Conditional on $(M^{\sss(j)}_k)_{k=0}^\infty$ and $(\tau_k)_{k=0}^\infty$, the collection of PWIT edge weights between a vertex $v\in\tau_k$ and all child vertices $vi$ for which $vi\notin\tau_k$, forms a Poisson point process of intensity $1$ on the interval $(M^{\sss(j)}_k,\infty)$.
Moreover these Poisson point processes are conditionally independent for every $k$ and every $v\in\tau_k$.
\item\lbitem{IPdiameter}
The part of $\IP^{\sss(j)}(\infty)$ not descended from $V^{\sss\BB,j}_k$ has diameter $O_\P(k)$.
\end{enumerate}
\end{prop}

\begin{proof}
The backbone is well-defined by Corollary~22 in \cite{AddGriKan12}.
The same paper proves \refitem{BBmaxAttained} in Theorems~21 and 30, \refitem{BBMarkov} in Section~3.3, \refitem{BBdual} in Theorem~31, and \refitem{BBweights} in Theorem~3.
It has been observed on the top of page 954 in \cite{AddGriKan12} that the methodology of \cite{AGdHS2008} can be applied to show that \cite[Proposition~3.3]{AGdHS2008} holds for the PWIT, proving \refitem{BBscaling}.

For parts~\refitem{OffBBweights} and \refitem{OffBBboundaryWeights}, notice that the event that $\tau_k$ equals a particular finite tree $\tau$ requires that the children of $V^{\sss\BB,j}_k$ should consist of (i) the child $V^{\sss\BB,j}_{k+1}$ with edge weight consistent with the process $M^{\sss(j)}$; and (ii) other children, and their descendants, joined to $V^{\sss\BB,j}_k$ by edges of weight less than $M^{\sss(j)}_k$, in numbers corresponding to the structure specified by $\tau$.
However, conditioning on $(M^{\sss(j)}_k)_{k=0}^\infty$ and $(\tau_k)_{k=0}^\infty$ does not impose any constraint on the precise value of the edge weights less than $M^{\sss(j)}_k$, nor on the uninvaded edge weights that exceed $M^{\sss(j)}_k$.
Parts~\refitem{OffBBweights} and \refitem{OffBBboundaryWeights} therefore follow from properties of Poisson point processes.

For \refitem{IPdiameter} it suffices to notice that \refitem{BBdual} implies that the
diameter of $\bigcup_{j=0}^{k-1}\tau_j$ is stochastically dominated by $k$ plus the maximum of $k$ extinction times from $k$ independent critical Poisson Galton--Watson branching processes. Since the probability that a critical Poisson Galton--Watson branching process lives to generation $\ell$ is $O(1/\ell)$, the claim follows.
\end{proof}

We next give two lemmas that we will use in the following section to bound expectations of functions of the backbone edge weights $X^{\sss\BB,j}_k$:

\begin{lemma}\lblemma{BackboneWeightsBound}
There is a constant $K<\infty$ such that, for any non-negative measurable function $h$ and any $k\in\N$,
\begin{equation}
\condE{h(X^{\sss\BB,j}_k)}{M^{\sss(j)}_0,\dotsc,M^{\sss(j)}_{k-1}}
\leq
K(M^{\sss(j)}_{k-1} - 1) h(M^{\sss(j)}_{k-1}) + \int_0^1 h(M^{\sss(j)}_{k-1} x)dx.
\end{equation}
\end{lemma}
\begin{proof}
This is immediate from \refprop{IPstructure}~\refitem{BBweights} and the bound $\condP{M^{\sss(j)}_k<M^{\sss(j)}_{k-1}}{M^{\sss(j)}_{k-1}}=1-M^{\sss(j)}_{k-1}(1-\theta(M^{\sss(j)}_{k-1})) = O(M^{\sss(j)}_{k-1}-1)$ from \eqrefPartI{MarkovForwardMax} and \refprop{PGWtrees}~\refitem{PGWSurvivalAsymp}.
\end{proof}
\begin{lemma}\lblemma{MStochasticBound}
Given $m_0 \in (1,\infty)$, there is a constant $K<\infty$ such that, for any $k,k_0 \in \N_0$ with $k>k_0$ and for all $m\in \ocinterval{1,m_0}$, the law of $M^{\sss(j)}_k$ conditional on $M^{\sss(j)}_{k_0}=m$ is stochastically dominated by $m \land (1+\frac{K}{k-k_0}E)$, where $E$ is an exponential random variable with mean $1$.
\end{lemma}
\begin{proof}
Use \eqrefPartI{MarkovForwardMax} and \eqrefPartI{thetaAsymptotics} to obtain $K \in (0,\infty)$ such that, for all $k\in \N_0$, $1<m'\le m\le m_0$,
\begin{equation}%\labelPartI{}
\condP{M^{\sss(j)}_{k+1} < m'}{M^{\sss(j)}_{k}=m} = \theta(m') \frac{1-m(1-\theta(m))}{\theta(m)} \geq \frac{m'-1}{K}.
\end{equation}
Moreover, for $1\le m'\le m\le m_0$,
\begin{equation}%\labelPartI{}
\P(m\land (1+KE) < m') = 1-\e^{-(m'-1)/K'} \leq \frac{m'-1}{K}.
\end{equation}
Hence, for all $k \in \N_0$, $1\le m'\le m\le m_0$,
\begin{equation}\labelPartI{eq:MStochBdBase}
\condP{M^{\sss(j)}_{k+1} < m'}{M^{\sss(j)}_{k}=m}\ge \P\left(m\land (1+KE) < m'\right)
\end{equation}
and the statement of the lemma is proved for $k=k_0+1$.
To establish the result for general $k>k_0$, we use an induction over $k$.
According to \eqrefPartI{eq:MStochBdBase}, we can couple $M_{k+1}^{\sss (j)}$ given $M_k^{\sss (j)}$ to an exponential random variable $E'$ of mean $1$ that is independent of $M_k^{\sss (j)}$ and satisfies $M_{k+1}^{\sss (j)} \le M_k^{\sss (j)} \land (1+K E')$ on $\set{M_k^{\sss (j)} \le m_0}$.
Using the Markov property of $(M_k^{\sss (j)})_k$ from \refprop{IPstructure}~\refitem{BBMarkov},
\begin{align}
\condP{M^{\sss(j)}_{k+1} < m'}{M^{\sss(j)}_{k_0}=m} \ge \E\Big( \condP{M^{\sss(j)}_{k} \land (1+K E') < m'}{M^{\sss(j)}_{k}}\Big|M_{k_0}^{\sss (j)} =m\Big).
\end{align}
By the induction hypothesis, the distribution of $M_k^{\sss (j)}$ conditional on $M_{k_0}^{\sss(j)}=m$ is stochastically dominated by $m \land (1+\frac{K}{k-k_0} E'')$ for an exponential random variable $E''$ of mean $1$ which can be chosen independent of $E'$.
Since $x \mapsto \P(x \land (1+K E') < m')$ is non-increasing, we obtain
\begin{align}
\condP{M^{\sss(j)}_{k+1} < m'}{M^{\sss(j)}_{k_0}=m} \ge \P\Big( m \land \Big(1+ \frac{K}{k-k_0} E''\Big) \land (1+K E') < m' \Big).
\end{align}
Since $E'$ and $E''$ are independent, $\big(\frac{1}{k-k_0} E''\big) \land E'$ is equal in distribution to $\frac{1}{k+1-k_0} E$ for an exponential random variable $E$ of mean $1$ and the proof is complete.
\end{proof}

\subsection{First passage times and the IP cluster: Proof of \refthm{TfrScaling}}\lbsubsect{TfrScalingPf}

For notational convenience, given a constant $K<\infty$ to be fixed below, we will abbreviate
\begin{equation}\labelPartI{Tabbreviation}
T=T_{V^{\sss\BB,j}_{\floor{Ks_n}}}.
\end{equation}
for the remainder of this section.
We begin with the following preliminary lemma:
\begin{lemma}\lblemma{SumofUniforms}
Let $(U_i)_{i \in \N}$ be an i.i.d.\ sequence of uniform random variables on $(0,1)$.
There exist $\delta>0$ and $n_0 \in \N$ such that
\begin{equation}%\labelPartI{}
\P\Big(\sum_{i=1}^{\floor{s_n}} f_n(mU_i)\leq f_n(m)\Big)\geq \delta \qquad \text{for all }m \in \cointerval{1,\infty}, n \ge n_0.
\end{equation}
\end{lemma}
\begin{proof}
Choose $u\in(0,1)$ such that $u^{\epsilonCondition}\leq\tfrac{1}{4}\epsilonCondition$, where $\epsilonCondition$ was fixed below \refcond{boundfnExtended}.
Using the inequality $\P(U_i\leq u^{1/s_n}\forall i \le \floor{s_n})\ge u$, we can bound
$\P(\sum_{i=1}^{\floor{s_n}} f_n(mU_i)\leq f_n(m))\geq u\condP{\sum_{i=1}^{\floor{s_n}} f_n(mU_i)\leq f_n(m)}{U_i\leq u^{1/s_n}\forall i \le \floor{s_n}}$.
By \reflemma{Integralfn}, for sufficiently large $n$ and for all $m\ge 1$,
\begin{align}
&\condE{\textstyle{\sum_{i=1}^{\floor{s_n}}} f_n(mU_i) }{U_i\leq u^{1/s_n}\forall i \le \floor{s_n}}
\notag\\
&\quad\leq
s_n \int_0^1 f_n(u^{1/s_n}mx)dx
\leq f_n(m) \Big( \frac{u^{\epsilon_0}}{\epsilon_0} + \frac{f_n(1)}{f_n(m)} \frac{(1-\delta_0)^{\epsilon_0 s_n}s_n}{u^{1/s_n} m}\Big)
\leq
(\tfrac{1}{4}+o(1))f_n(m).
\end{align}
By Markov's inequality, $\condP{\sum_{i=1}^{\floor{s_n}} f_n(mU_i)> f_n(m)}{U_i\leq u^{1/s_n}\forall i\le \floor{s_n}}\leq (\tfrac{1}{4}+o(1))\leq \tfrac{1}{2}$, so that $\condP{\sum_{i=1}^{\floor{s_n}} f_n(mU_i)\leq f_n(m)}{U_i\leq u^{1/s_n}\forall i\le \floor{s_n}}\geq\tfrac{1}{2}$ for all $n$ sufficiently large.
We may therefore take $\delta=u/2$.
\end{proof}

\begin{lemma}\lblemma{FreezingByHeight}
Given $\delta>0$, there exist $K<\infty$ and $n_0 \in \N$ such that $T_\fr^{\sss(j)}\leq T_{V^{\sss\BB,j}_{\floor{K s_n}}}$ with probability at least $1-\delta$ for all $n \ge n_0$.
\end{lemma}
\begin{proof}
By \refthm{FrozenCluster}~\refitem{FrozenVolume} there exists $K_1$ such that $\P(\abs{\BP_\fr^{\sss(j)}} > K_1 s_n^2) <\delta/4$ for large $n$.
Recall from \eqrefPartI{Tabbreviation} that we write $T=T_{V^{\sss\BB,j}_{\floor{Ks_n}}}$.
We now show that there is a $K<\infty$ such that $\BP_T^{\sss (j)}$ contains more than $K_1 s_n^2$ vertices with probability at least $1-3\delta/4$ for large $n$.
This will then prove the lemma.
Set
\begin{equation}%\labelPartI{}
\mathcal{A}=\bigl\{M^{\sss(j)}_{\floor{s_n}} \leq 1+K_2/s_n\bigr\},
\end{equation}
where $K_2<\infty$ is chosen using \refprop{IPstructure}~\refitem{BBscaling} so that $\P(\mathcal{A}) > 1- \delta/4$ for large $n$.
By monotonicity, on $\mathcal{A}$, $M_k^{\sss (j)} \le 1+K_2/s_n$ for all $k \ge \floor{s_n}$.

Let $k \ge \floor{s_n}$ and write $\tau'_k\subset\tau_k$ for those vertices in the $k^\th$ off-backbone branch that lie within distance $s_n$ of $V^{\sss\BB,j}_k$.
By \refprop{IPstructure}~\refitem{BBdual}, given $M_k^{\sss(j)}$, $\tau_k$ is a Poisson Galton--Watson tree with mean $\hat{M}_k^{\sss(j)}=M_k^{\sss(j)}(1-\theta(M_k^{\sss(j)}))$.
When $M_k^{\sss(j)} \in \ocinterval{1,1+K_2/s_n}$, then \eqrefPartI{hatmandm} yields $K_3$ such that $\hat{M}_k^{\sss(j)} \in \cointerval{1-K_3/s_n,1}$ for sufficiently large $n$.
Hence, \refprop{PGWtrees}~\refitem{GWTotalProg-LargeConst} and \refitem{GWCut} yield $\delta_1>0$ such that $\P\condparenthesesreversed{\abs{\tau'_k}\geq s_n^2}{M^{\sss(j)}_k\leq 1+K_1/s_n} \geq\delta_1/s_n$ for sufficiently large $n$.

Consider $v\in\tau'_k$.
By \refprop{IPstructure}~\refitem{OffBBweights}, conditional on $M^{\sss(j)}_k$, the path from $V^{\sss\BB,j}_k$ to $v$ consists of at most $s_n$ edges whose PWIT edge weights $X_{wj}$ are independent and uniformly distributed on $[0,M^{\sss(j)}_k]$.
Write $\tau''_k$ for the collection of vertices $v\in\tau'_k$ such that $T_v-T_{V^{\sss\BB,j}_k} \leq f_n(M^{\sss(j)}_k)$.
By \reflemma{SumofUniforms}, there exists $\delta_2>0$ such that $\condP{v\in\tau''_k}{v\in\tau'_k, M_k^{\sss (j)}}\geq\delta_2$, so that \reflemma{ReverseMarkov} implies that $\condP{\abs{\tau''_k}\geq \tfrac{1}{2}\delta_2 s_n^2}{M^{\sss(j)}_k} \geq (\tfrac{1}{2}\delta_2)(\delta_1/s_n)$ on $\mathcal{A}$.

By \refprop{IPstructure}~\refitem{BBdual} and \refitem{OffBBweights}, the off-backbone branches $\tau_k,\tau'_k,\tau''_k$ are conditionally independent for different $k$ given $(M^{\sss(j)}_k)_{k=0}^\infty$.
Therefore, conditional on $(M^{\sss(j)}_k)_{k=0}^\infty$ and the event $\mathcal{A}$, we have for each $k\geq \floor{s_n}$ an independent chance, bounded below by $\tfrac{1}{2}\delta_2\delta_1/s_n$, that $\abs{\tau''_k}\geq \tfrac{1}{2}\delta_2 s_n^2$.
It follows that, by taking $K_2$ sufficiently large, we have $\condP{\sum_{k=\floor{s_n}+1}^{\floor{K_2 s_n}} \abs{\tau''_k} > K_1 s_n^2}{\mathcal{A}} > 1-\delta/4$.

Finally, let $\mathcal{A}'$ be the event that the unique edge with weight $M^{\sss(j)}_{\floor{K_2 s_n}}$ occurs at height at most $Ks_n$.
By \refprop{IPstructure}~\refitem{BBmaxAttained}, we may choose $K>K_2$ large enough that $\P(\mathcal{A}')> 1-\delta/4$.

Whenever $\mathcal{A}'$ occurs, it follows that $T-T_{V^{\sss\BB,j}_k} \geq f_n(M^{\sss(j)}_k)$ for each $k\leq K_2 s_n$, since the collection of first passage weights $f_n(X^{\sss\BB,j}_{k+1}),\dotsc,f_n(X^{\sss\BB,j}_{\floor{Ks_n}})$ along the path from $V^{\sss\BB,j}_k$ to $V^{\sss\BB,j}_{\floor{Ks_n}}$ must contain the edge weight $f_n(M^{\sss(j)}_k)$.
In particular, for any $v\in\tau''_k$ with $k\leq K_2 s_n$, we have $T_v\leq f_n(M^{\sss(j)}_k)+T_{V^{\sss\BB,j}_k} \leq T$ and therefore $v\in\BP^{\sss(j)}_T$.
Thus
\begin{align}
\P(T < T_\fr^{\sss(j)})
&\leq
\P\left( \abs{\BP_\fr^{\sss(j)}} > K_1 s_n^2 \right) + \P\left( T < T_\fr^{\sss(j)}\text{ and}\abs{\BP_\fr^{\sss(j)}} \leq K_1 s_n^2 \right)
\notag\\&
<
\delta/4 + \P\left( \abs{\BP_T^{\sss(j)}} \leq K_1 s_n^2 \right)
\notag\\&
\leq
\delta/4 + \P(\mathcal{A}^c)+\P((\mathcal{A}')^c)+\condP{\mathcal{A}'\cap \big\{\abs{\BP_T^{\sss(j)}} \leq K_1 s_n^2\big\}}{\mathcal{A}}
\notag\\&
<
\delta/4+\delta/4+\delta/4+\condP{\textstyle{\sum_{k=\floor{s_n}+1}^{\floor{K_2 s_n}}} \abs{\tau''_k} \le K_1 s_n^2}{\mathcal{A}}
<\delta
.\qedhere
\end{align}
\end{proof}
\begin{lemma}\lblemma{ExtraTimeToFreeze}
Given $m_0 \in (1,\infty)$ and $K<\infty$, there is a constant $K''<\infty$ such that
\begin{equation}%\labelPartI{}
\condE{T_{V^{\sss\BB,j}_{\floor{Ks_n}}}-T_{V^{\sss\BB,j}_{k_0}}}{M^{\sss(j)}_{k_0}=m}\leq K''f_n(m)
\end{equation}
for all $k_0\leq \floor{Ks_n}$ and all $1\leq m\leq m_0$.
\end{lemma}
\begin{proof}
We have $T_{V^{\sss\BB,j}_{\floor{Ks_n}}}
-T_{V^{\sss\BB,j}_{k_0}}=\sum_{k=k_0+1}^{\floor{Ks_n}} f_n(X^{\sss\BB,j}_k)$.
Using first \reflemma{BackboneWeightsBound} to find a constant $K_1<\infty$ and then \reflemma{Integralfn} and \reflemma{ExtendedImpliesWeak} and $M_{k-1}^{\sss (j)} \le M_{k_0}^{\sss (j)}$, we estimate uniformly in $m$
\begin{align}
&\sum_{k=k_0+1}^{\floor{Ks_n}}\condE{f_n(X^{\sss\BB,j}_k)}{M^{\sss(j)}_{k_0}=m}
\notag\\&\quad
\leq
\sum_{k=k_0+1}^{\floor{Ks_n}} \condE{K_1\left(M^{\sss(j)}_{k-1}-1 \right)f_n(M^{\sss(j)}_{k-1})+\int_0^1 f_n\big(M^{\sss(j)}_{k-1} x\big)dx}{M^{\sss(j)}_{k_0}=m}
\notag\\ & \quad
\leq
\sum_{k=k_0+1}^{\floor{Ks_n}} \condE{K_1\left( M^{\sss(j)}_{k-1}-1 \right) f_n(M^{\sss(j)}_{k-1})+\frac{f_n(M^{\sss(j)}_{k-1})+o(f_n(1))}{\epsilonCondition s_n}}{M^{\sss(j)}_{k_0}=m}
\notag\\&\quad
\leq
O(f_n(m))\sum_{k=k_0+1}^{\floor{Ks_n}} \condE{ \left(M^{\sss(j)}_{k-1}-1 \right) \left( \frac{M^{\sss(j)}_{k-1}}{m} \right)^{\epsilonCondition s_n} }{M^{\sss(j)}_{k_0}=m}
+O(f_n(m))+o(f_n(1))
.
\labelPartI{BackboneSumBound}
\end{align}
In the sum on the right-hand side of \eqrefPartI{BackboneSumBound}, the term $k=k_0+1$ contributes at most $m_0-1$ since $m \le m_0$.
The sum of contributions from the events $\set{M^{\sss(j)}_{k-1}\leq 1+1/s_n}$, $k=k_0+2,\dotsc,\floor{Ks_n}$, is $O(1)$ since $M_{k-1}^{\sss (j)}\le M_{k_0}^{\sss (j)}$.
We conclude that
\begin{align}
\labelPartI{BackboneSumBound2}
&\sum_{k=k_0+1}^{\floor{Ks_n}}\condE{f_n(X^{\sss\BB,j}_k)}{M^{\sss(j)}_{k_0}=m}
\\&\quad
\leq
O(f_n(m))+O(f_n(m))\sum_{k=k_0+2}^{\floor{Ks_n}} \condE{ \indicator{M^{\sss(j)}_{k-1}\geq 1+1/s_n} \left(M^{\sss(j)}_{k-1}-1 \right) \left( \frac{M^{\sss(j)}_{k-1}}{m} \right)^{\epsilonCondition s_n} }{M^{\sss(j)}_{k_0}=m}
\notag
.
\end{align}
In the expectation on the right-hand side of \eqrefPartI{BackboneSumBound2}, the integrand is increasing in $M^{\sss(j)}_{k-1}$, so as an upper bound we can use \reflemma{MStochasticBound} to replace $M^{\sss(j)}_{k-1}$ by $m \wedge (1+\frac{K'}{k-1-k_0}E)$ for some $K'<\infty$.
Letting $i=k-1-k_0$,
\begin{align}
&\sum_{k=k_0+1}^{\floor{Ks_n}}\condE{f_n(X^{\sss\BB,j}_k)}{M^{\sss(j)}_{k_0}=m} - O(f_n(m))
\notag\\&\quad
\leq
O(f_n(m))\sum_{i=1}^\infty \left[ (m-1) \P\left(\frac{K' E}{i} > m-1\right) + \E\left( \indicator{1/s_n \leq \frac{K' E}{i} \leq m-1} \frac{K' E}{i} \Big(\frac{1+K' E/i}{m}\Big)^{\epsilonCondition s_n} \right) \right]
\notag\\&\quad
\leq
O(f_n(m))\sum_{i=1}^\infty \left( (m-1)\e^{-(m-1)i/K'}+\int_{1/s_n}^{m-1} y \left( \frac{1+y}{m} \right)^{\epsilonCondition s_n} \frac{i}{K'} \e^{-iy/K'}dy \right)
\notag\\&\quad
\leq
O(f_n(m))\left( \frac{m-1}{1-\e^{-(m-1)/K'}} + \int_{1/s_n}^{m-1} \frac{y/K'}{(1-\e^{-y/K'})^2} \left( \frac{1+y}{m} \right)^{\epsilonCondition s_n} dy \right)
.
\labelPartI{BackboneSumBound3}
\end{align}
The ratio $y/(1-\e^{-y/K'})$ is uniformly bounded for $y\leq m_0-1$ and $1/(1-\e^{-y/K'}) \le 1/(1-\e^{-1/(K's_n)}) \sim K' s_n$ for $y \ge1/s_n$.
We conclude that
\begin{align}
&\sum_{k=k_0+1}^{\floor{Ks_n}}\condE{f_n(X^{\sss\BB,j}_k)}{M^{\sss(j)}_{k_0}=m}
\leq
O(f_n(m))+O(f_n(m))\int_{1/s_n}^{m-1} s_n \left( \frac{1+y}{m} \right)^{\epsilonCondition s_n} dy
,
\end{align}
and this upper bound is $O(f_n(m))$ as required.
\end{proof}

\reflemma{FreezingByHeight} and \reflemma{ExtraTimeToFreeze} now allow us to prove \refthm{TfrScaling}:

\begin{proof}[Proof of \refthm{TfrScaling}]
We need to show that $f_n^{-1}(T_\fr^{\sss(j)})\convp M^{\sss(j)}_0$.
We begin by arguing that $f_n^{-1}(T_\fr^{\sss(j)}) \geq M^{\sss(j)}_0$ {\whpdot}  Indeed, the subgraph $\BP^{\sss(j)}_{f_n(M^{\sss(j)}_0)^-}$ explored strictly before time $f_n(M^{\sss(j)}_0)$ is a subgraph of those vertices connected to the root by paths that use only edges of PWIT edge weight strictly less than $M^{\sss(j)}_0$.
This latter subgraph is a random subgraph of $\tree^{\sss(j)}$ that is finite a.s.\ (by \refprop{IPstructure}~\refitem{BBmaxAttained} and \refitem{BBdual}) and independent of $n$.
In particular, the subgraph $\BP^{\sss(j)}_{f_n(M^{\sss(j)}_0)^-}$ has size $O_\P(1)$.
To show that $T_\fr^{\sss(j)}\geq f_n(M^{\sss(j)}_0)$ \whp, it therefore suffices to show that $\abs{\cluster^{\sss(j)}_\fr}\convp\infty$.

Indeed, we shall show that $\cluster_\fr^{\sss(j)}$ must contain at least of order $s_n$ vertices.
To this end, it suffices by \reflemma{ExpectedUnfrozenChildren} to show that any vertex $v\in\cluster_\fr^{\sss(j)}$ can contribute at most $O(1)$ to $\int \e^{-\lambda_ny}d\mu_{n,\fr}^{\sss(j)}(y)$.
We use \reflemma{munDensityBounded} to compute
\begin{equation}\labelPartI{ContributionUniformBound}
\int \e^{-\lambda_ny} \mu_n(t-T_v+dy) \leq 1 + \int_{f_n(1)}^\infty \e^{-\lambda_ny} \frac{dy}{\epsilonCondition s_n f_n(1)}
\end{equation}
for $n$ large enough, uniformly in the age $t-T_v$, and the upper bound in \eqrefPartI{ContributionUniformBound} is $O(1)$ by \reflemma{lambdanAsymp}.
We have therefore shown that $f_n^{-1}(T_\fr^{\sss(j)})\geq M^{\sss(j)}_0$ {\whpdot}

For the corresponding upper bound, let $\epsilon>0$ be given and choose $m_0<\infty$ large enough that $\P(M^{\sss(j)}_0>m_0)<\epsilon$.
By \reflemma{FreezingByHeight} it is enough to show that for any fixed $K<\infty$, $T_{V^{\sss\BB,j}_{\floor{K s_n}}}\leq f_n(M^{\sss(j)}_0+\epsilon)$ \whp\ on $\set{M^{\sss(j)}_0\leq m_0}$.
Use \reflemma{ExtraTimeToFreeze} with $k_0=0$ together with Markov's inequality to find a constant $K'<\infty$ such that
\begin{equation}%\labelPartI{}
\P\left( T_{V^{\sss\BB,j}_{\floor{K s_n}}} \ge f_n(M^{\sss(j)}_0+\epsilon), M^{\sss(j)}_0\leq m_0\right)   \le K' \E\left( \frac{f_n(M_0^{\sss (j)})}{f_n(M_0^{\sss (j)}+\epsilon)} \indicator{M^{\sss(j)}_0\leq m_0} \right).
\end{equation}
By \reflemma{ExtendedImpliesWeak}, $f_n(M^{\sss(j)}_0+\epsilon)/f_n(M^{\sss(j)}_0) \convas\infty$, and the dominated convergence theorem completes the proof.
\end{proof}

\subsection{Diameter of the frozen cluster: Proof of \refthm{FrozenCluster}~\refitem{FrozenDiameter}}\lbsubsect{FrozenDiameterPf}

We next give an upper bound on the number of boundary edges of the IP cluster for which a distant descendant is explored before the freezing time.
Here we say that a vertex $v$ is explored by time $t$ if $v \in \BP_t$.
Vertices that belong to the IP cluster $\IP(\infty)$ are called {\em invaded.}
Write $N(k,K)$ for the number of vertices of $\boundary\tau_k$ for which some descendant at least $\ceiling{s_n}$ generations away belongs to $\BP_T^{\sss(j)}$.
(Recall from \eqrefPartI{Tabbreviation} the abbreviation $T=T_{V_{\floor{Ks_n}}^{\sss \BB,j}}$, and recall also that $\tau_k$ denotes the tree consisting of vertices that are descendants of $V^{\sss\BB,j}_k$ but not descendants of $V^{\sss\BB,j}_{k+1}$.)

\begin{lemma}\lblemma{ExpExplUninvChildrenBound}
Given $\delta>0$, $K<\infty$ and $m_0 \in (1,\infty)$, there is a constant $K'<\infty$ such that for all $k\in \N_0$
\begin{equation}\labelPartI{NkKbound}
\E\condparenthesesreversed{N(k,K)}{1+\delta/s_n\leq M^{\sss(j)}_k \leq m_0}\leq K'/s_n.
\end{equation}
\end{lemma}
\begin{proof}
Consider $v\in\tau_k$ and its children, which we write as $vi$, $i\in\N$.
Recall from \refprop{IPstructure}~\refitem{OffBBboundaryWeights} that, conditional on the invasion cluster, the PWIT weights $\set{X_{vi}\colon vi\notin\tau_k}$ (i.e., the weights corresponding to children that were \emph{not} invaded) form a Poisson point process on $(M^{\sss(j)}_k, \infty)$ with intensity 1.

A child $vi$ of $v$ will be explored by the FPP process by time $T$ if and only if $f_n(X_{vi})\leq T-T_v$.
Abbreviate $\tilde{X}=f_n^{-1}(T-T_{V^{\sss\BB,j}_k})$.
Since $T_v\geq T_{V^{\sss\BB,j}_k}$, we have as an upper bound that $vi$ can only be explored if $X_{vi}\leq \tilde{X}$.
In particular, the conditional expectation of the number of uninvaded children of $v$ that are explored by the FPP process by time $T$ obeys the bound
\begin{equation}\labelPartI{ExpExplUninvChildrenPerVertex}
\condE{\bigg.\abs{\set{\Big.i\in\N\colon vi\notin\tau_k, vi\in \BP_T^{\sss(j)}}}}{M^{\sss(j)}_k =m, \tilde{X}} \leq (\tilde{X}-m)\indicator{\tilde{X}\geq m}.
\end{equation}

For each such child $vi$, we may bound the probability that a descendant of $vi$ at least $\ceiling{s_n}$ generations away is explored by time $T$ %$T_{V^{\sss\BB,j}_{\floor{Ks_n}}}$
by the probability that a Poisson Galton--Watson branching process with mean $\tilde{X}$ survives for at least $s_n$ generations.
Because of \eqrefPartI{ExpExplUninvChildrenPerVertex} and since in \eqrefPartI{NkKbound} $M_k^{\sss(j)}\ge 1+\delta/s_n$, we may assume that $\tilde{X}\geq m\geq 1+\delta/s_n$.
By \refprop{PGWtrees}~\refitem{PGWSurvivalTime} and \refitem{PGWSurvivalAsymp}, this probability is $O(\tilde{X}-1)$.

On the other hand, conditional on $M^{\sss(j)}_k=m$, the expected size of $\tau_k$ is $O(1/(m-1))$, since, by \refprop{IPstructure}~\refitem{BBdual} and \eqrefPartI{hatmandm}, $\tau_k$ is distributed as a Poisson Galton--Watson tree with subcritical parameter $\hat{m}<1$ and $1-\hat{m}\sim m-1$.
We conclude that uniformly in $m \in \cointerval{1+\delta/s_n,\infty}$,
\begin{align}\labelPartI{ExpExplUninvChildren}
\E \condparenthesesreversed{ N(k,K) }{M^{\sss(j)}_k=m, \tilde{X}}
&\leq
O(1)\frac{(\tilde{X}-1)(\tilde{X}-m)}{m-1}\indicator{\tilde{X}\geq m}
\notag\\&
=O(1)\Big( \tilde{X}-m +\frac{(\tilde{X}-m)^2}{m-1} \Big) \indicator{\tilde{X}\geq m}.
\end{align}
In order to bound the right-hand side of \eqrefPartI{ExpExplUninvChildren}, we first
use \reflemma{ExtendedImpliesWeak} and \reflemma{ExtraTimeToFreeze} and Markov's inequality to obtain a constant $K'<\infty$ such that for all $m \in [1,m_0]$ and $x \ge 0$,
\begin{align}
&\condP{ \tilde{X}-m \geq x }{M^{\sss(j)}_k=m}
=
\condP{ T
-T_{V^{\sss\BB,j}_k} \geq f_n\left( m+x \right) }{M^{\sss(j)}_k=m}
\notag\\&\quad
\leq
\condP{ T
-T_{V^{\sss\BB,j}_k} \geq \left( 1+\frac{x}{m} \right)^{\epsilonCondition s_n} f_n(m) }{M^{\sss(j)}_k=m}
\leq
\frac{K'}{(1+x/m)^{\epsilonCondition s_n}}.
\end{align}
Thus, for $1 \le m\leq m_0$ and $p\in \set{1,2}$,
\begin{align}
&\condE{\left( \tilde{X}-m \right)^p\indicator{\tilde{X}\geq m}}{M^{\sss(j)}_k=m}
=
\int_0^\infty p x^{p-1}\condP{\tilde{X}-m\geq x}{M^{\sss(j)}_k=m} dx
\notag\\&\quad
\leq
K'\int_0^\infty \frac{p x^{p-1} dx}{(1+x/m)^{\epsilonCondition s_n}}
=
O(1)\int_1^\infty \frac{(t-1)^{p-1} dt}{t^{\epsilonCondition s_n}}
=
\begin{cases}
O(1/s_n), & p=1,\\
O(1/s_n^2), & p=2.
\end{cases}
\end{align}
For $m \ge 1+\delta/s_n$, we have $1/(m-1) \leq O(s_n)$ and we obtain that uniformly for $m\in [1+\delta/s_n,m_0]$,
	\begin{align}
	&\condE{\left( \tilde{X}-m +\frac{(\tilde{X}-m)^2}{m-1} \right)
	\indicator{\tilde{X}\geq m}}{M^{\sss(j)}_k=m}
	%\notag\\&\quad
	=
	%\condE{\left[ O(s_n)(\tilde{X}-m)^2+(\tilde{X}-m) \right]
	%\indicator{\tilde{X}\geq m}}{M^{\sss(j)}_k=m}
	O(1/s_n)+O(s_n)O(1/s_n^2)
	=
	O(1/s_n)
	.
	\labelPartI{hatXbound}
	\end{align}
Combining \eqrefPartI{ExpExplUninvChildren} with \eqrefPartI{hatXbound} completes the proof.
\end{proof}

\reflemma{ExpExplUninvChildrenBound} will allow us to control the effect of branches $\tau_k$ for small $k$.
For larger $k$, the following estimate will allow us to bound the probability that a very long path is explored between time $T_{V^{\sss\BB,j}_k}$ and the freezing time. Let $v\in \tau_k$.
Write $D_{v,r,C,\ell}$ for the collection of descendants of $v$ connected to $v$ by exactly $r$ edges of weight at most $1+C/s_n$, at most $\ell$ of which lie between $1-1/s_n$ and $1+C/s_n$.

\begin{lemma}\lblemma{NoFastLongPaths}
Given $C,\ell<\infty$, there is $C'<\infty$ such that $\P(D_{v,r,C,\ell}\neq \emptyset)\leq C'/r$ for all $r\in\N$ , $v \in \tree$ and $s_n\geq 1$.
\end{lemma}
\begin{proof}
On the event $\set{D_{v,r,C,\ell}\neq\emptyset}$, let $W$ be chosen uniformly from $D_{v,r,C,\ell}$, and let $L$ denote the collection of vertices $wi$ along the path from $v$ to $W$ for which $1-1/s_n\leq X_{wi} \leq 1+C/s_n$.
By the Poisson point process property, conditional on the occurence of $\set{D_{v,r,C,\ell}\neq \emptyset}$ and the values of $W$ and $L$, the edge weights $X_{wi}$, $wi\in L$, are uniformly distributed on $[1-1/s_n,1+C/s_n]$.
In particular, $1-1/s_n\leq X_{wi}\leq 1$ for each $wi\in L$ with conditional probability $(1/(C+1))^{\abs{L}}$, which is at least $(1/(C+1))^\ell$ by construction.
If this occurs then $v$ is connected by edges of weight at most 1 to a descendant at generation $r$.
Using the notation from \refprop{PGWtrees} and appealing to standard properties of critical Galton--Watson processes (see for example Lemma~I.1.10 in \cite{Har63}),
\begin{equation}
O(1/r) = \P_1(\abs{\tau^{\sss(\ge r)}}\ge 1)
\geq
\P(D_{v,r,C,\ell}\neq\emptyset) \cdot (1/(C+1))^\ell
,
\end{equation}
completing the proof.
\end{proof}
\begin{proof}[Proof of \refthm{FrozenCluster}~\refitem{FrozenDiameter}]
We have to show that the diameter of $\cluster_\fr^{\sss(j)}$ is $O_\P(s_n)$.
Let $\epsilon>0$ be given.
Use \reflemma{FreezingByHeight} to choose $K<\infty$ such that $T_\fr^{\sss(j)}\leq T$
with probability at least $1-\epsilon$ for large $n$.
Use \refprop{IPstructure}~\refitem{BBscaling} to choose $\delta>0$ and $m_0<\infty$ so that $1+\delta/s_n\leq M_{\floor{Ks_n}}^{\sss(j)} \leq M_0^{\sss(j)}\leq m_0$ with probability at least $1-\epsilon$.
By \reflemma{ExpExplUninvChildrenBound}, we may choose $\eta\in(0,1)$ sufficiently small that $\E\left( \indicator{1+\delta/s_n\leq M_{\ceiling{s_n}}^{\sss(j)} \leq M_0^{\sss(j)}\leq m_0} \sum_{k=0}^{\floor{\eta s_n}} N(k,K) \right) <\epsilon$, so that $N(k,K)=0$ for each $k\leq \eta s_n$ with probability at least $1-2\epsilon$ for large $n$.

Use \refprop{IPstructure}~\refitem{BBscaling} to choose $C'<\infty$ such that $M_{\ceiling{\eta s_n}}^{\sss(j)}\leq 1+C'/s_n$ with probability at least $1-\epsilon$.
Assuming that this event occurs, \reflemma{ExtraTimeToFreeze} and Markov's inequality give $C''<\infty$ such that $T-T_{V^{\sss\BB,j}_{\ceiling{\eta s_n}}}\leq C'' f_n(1+C'/s_n)$ with conditional probability at least $1-\epsilon$.
By \refcond{scalingfn}, it follows that we may choose $C,\ell<\infty$ such that
\begin{equation}\labelPartI{ExtraTimeSpecific}
T-T_{V^{\sss\BB,j}_{\ceiling{\eta s_n}}}\leq f_n(1+C/s_n)\leq \ell f_n(1-1/s_n)
\end{equation}
with probability at least $1-\epsilon$ for all sufficiently large $n$.

With these preliminaries, we can now complete the proof.
Assume that all the above events occur, making an error of at most $6\epsilon$.
By \refprop{IPstructure}~\refitem{IPdiameter}, the part of the IP cluster that does not descend from $V^{\sss \BB,j}_{\floor{Ks_n}}$ has diameter $O_\P(s_n)$.
Hence it suffices to show that vertices of $\BP_T^{\sss(j)}$ are within distance $O_\P(s_n)$ from the IP cluster.
The assumption $N(k,K)=0$, $k\leq \eta s_n$, verifies this for the beginning of the backbone.

For $k\geq \eta s_n$, \eqrefPartI{ExpExplUninvChildrenPerVertex}, \eqrefPartI{ExtraTimeSpecific} and the bound $\E\condparenthesesreversed{\abs{\tau_k}}{M_k^{\sss(j)}}=O(1/(M_k^{\sss(j)}-1))=O(s_n)$ imply that the (conditionally) expected number of vertices $vi\in\BP_T^{\sss(j)}\setminus\tau_k$ for which $v\in\tau_k$, for some $\eta s_n\leq k\leq Ks_n$, is at most $(C/s_n)\cdot O(s_n) \cdot (Ks_n) = O(s_n)$.
If such a vertex $vi$ has a descendant $w\in\BP_T^{\sss(j)}$ at distance $r$, then by \eqrefPartI{ExtraTimeSpecific} the path from $vi$ to $w$ can contain no edges of PWIT weight greater than $1+C/s_n$ and at most $\ell$ edges of PWIT weight greater than $1-1/s_n$.
In other words, necessarily $D_{vi,r,C,\ell}\neq\emptyset$.
This event is conditionally independent of the possible $vi$, so by \reflemma{NoFastLongPaths} the conditional probability of any $D_{vi,r,C,\ell}$ occuring is at most $O(s_n/r)$.
This can be made smaller than $\epsilon$ by taking $r=C''' s_n$ for $C'''$ large enough, and this completes the proof.
\end{proof}

\section{A strong disorder result: Proof of \refthm{IPWeightForFPP}}\lbsect{IPWeightProof}

\begin{proof}[Proof of \refthm{IPWeightForFPP}]
The proof of \eqrefPartI{WeightRescalesToIP_gen} proceeds via stochastic upper and lower bounds on $W_n$ based on couplings with IP where we use two different exploration processes.
For the lower bound, consider the minimal-rule exploration process $\explore=(\explore_k)_{k\in \N_0}$ given by IP on $\tree$ that alternates between an invasion step in $\tree^{\sss(1)}$ and an invasion step in $\tree^{\sss(2)}$ as explained below \refdefn{MinimalRule}.
By \refthm{CouplingExpl} the edge-weights $Y_e^{\sss(K_n)}=f_n(X_e^{\sss(K_n)})$ derived from this exploration process as in \eqrefPartI{EdgeWeightCoupling} are i.i.d.\ with distribution function $\FY$ and it suffices to consider the FPP problem on $K_n$ with these edge weights.
Let $\fstpond^{\sss(j)}$ denote the set of vertices in the invasion cluster on $\tree^{\sss(j)}$ explored before the edge of weight $M_0^{\sss(j)}$ is invaded.
(In the language of \cite{AddGriKan12,DamSap11,Good12,SteNew95}, $\fstpond^{(j)}$ is the first pond, not including the first outlet.)
Let $\fstpond^{\sss(j)}_+$ consist of $\fstpond^{\sss(j)}$ together with all adjacent vertices connected by an edge of weight at most $M_0^{\sss(1)}\vee M_0^{\sss(2)}$.

Let $\cA_n$ be the event that none of the vertices in $\fstpond^{\sss(1)}_+\union \fstpond^{\sss(2)}_+$ are thinned, and that the exponential variable $X_{\set{1,2}}^{\sss(K_n)}=E_{\set{1,2}}$ from \eqrefPartI{EdgeWeightCoupling} satisfies $X_{\set{1,2}}^{\sss(K_n)}\geq \tfrac{1}{n}(M_0^{\sss(1)}\vee M_0^{\sss(2)})$.
Since $\fstpond^{\sss(j)}_+$ is finite and the vertices in $\fstpond^{\sss(j)}$ are explored after a finite number of steps (not depending on $n$), $\cA_n$ holds with high probability.

On $\cA_n$, let $\mathscr{W}^{\sss(j)}$ denote the image in $[n]$ of $\fstpond^{\sss(j)}$ under the mapping $\pi_M\colon  v\mapsto M_v$ from \refdefn{InducedGraph}.
Then every edge $e$ between $\mathscr{W}^{\sss(j)}$ and $[n]\setminus(\mathscr{W}^{\sss(1)}\union \mathscr{W}^{\sss(2)})$ satisfies $Y_e^{\sss(K_n)}\geq f_n(M_0^{\sss(j)})$, and every edge $e$ between $\mathscr{W}^{\sss(1)}$ and $\mathscr{W}^{\sss(2)}$ satisfies $Y_{e}^{\sss(K_n)}\geq f_n(M_0^{\sss(1)}\vee M_0^{\sss(2)})$.
Since every path between vertices $1$ and $2$ has to leave $\mathscr{W}^{\sss(1)}$ and $\mathscr{W}^{\sss(2)}$, this therefore proves that $W_n\ge f_n(M_0^{\sss(1)} \vee M_0^{\sss(2)})$ on $\cA_n$, i.e., whp.

For the upper bound, let $\epsilon>0$ and let $N, N_1\in\N$ denote constants to be chosen later.
Modify the minimal-rule exploration process above by stopping after $N$ steps in each subtree $\tree^{\sss(j)}$, i.e., set $\explore'_k=\explore_{k\wedge 2N}$, and couple the edge weights according to \eqrefPartI{EdgeWeightCoupling}.
Denote by $X_{\max}^{\sss(j)}$ the largest edge weight in $\tree^{\sss(j)}$ so invaded, so that $X_{\max}^{\sss(j)}\leq M_0^{\sss(j)}$ by definition.

Let $\mathscr{U}^{\sss(j)}=\set{v\in\boundary \explore_{2N_1}\intersect\tree^{\sss(j)}\colon X_{\max}^{\sss(j)}<X_v < X_{\max}^{\sss(j)}+\epsilon}$ denote the collection of boundary vertices joined to invaded vertices by an edge of weight at most $X_{\max}^{\sss(j)}+\epsilon$.
Conditional on $X_{\max}^{\sss(1)},X_{\max}^{\sss(2)}$ and $\explore'$, the number $\abs{\mathscr{U}^{\sss(j)}}$ of such boundary vertices is Poisson with mean $\epsilon N$, independently for $j\in\set{1,2}$.
(This holds because the event that the exploration process $\explore'$ explores a given sequence of vertices $v_1,\dotsc,v_{2N_1}$ can be expressed solely in terms of the numbers $\abs{\set{vw\in\boundary \explore'_k\colon X_{vw}<X_{v_i}}}$ of boundary edges of \emph{smaller} weight, over all $k,i=1,\dotsc,2N$.)

Let $\mathcal{A}'_n$ denote the event that none of the vertices in $\explore'_{2N}\union\mathscr{U}^{\sss(1)}\union\mathscr{U}^{\sss(2)}$ have the same mark.
Condition on the occurrence of $\mathcal{A}'_n$ and the disjoint vertex sets $\pi_M(\explore'_{2N}), \pi_M(\mathscr{U}^{\sss(1)}), \pi_M(\mathscr{U}^{\sss(2)})$.
Consider the induced subgraph $K'_{n-2N-2}$ of $K_n$ obtained by excluding the $2N+2$ explored vertices in $\pi_M(\explore'_{2N})$.
Since no other vertices are explored, the edge weights in this induced subgraph are the independent exponential random variables $E_e$ from \eqrefPartI{EdgeWeightCoupling}.

The random subgraph $G'_{n-2N-2}=\set{e\in E(K'_{n-2N-2})\colon E_e\leq \tfrac{1}{n}(1+\epsilon)}$ has the (conditional) law of the Erd\H{o}s-R\'enyi random graph $G(n-2N-2,p)$ with $p=\P(E_e\leq \tfrac{1}{n}(1+\epsilon))\sim \tfrac{1}{n}(1+\epsilon)$ as $n\to\infty$.
As is well known, in the supercritical regime, the giant component has diameter $O_\P(\log n)$ and contains a positive asymptotic fraction of vertices (see e.g., \cite{FerRam04}).
Suppose $U^{\sss(1)},U^{\sss(2)}$ are two disjoint subsets of vertices in $G'_{n-2N-2}$ (possibly random but independent of the randomness in $G'_{n-2N-2}$).
If $U^{\sss(1)}$ and $U^{\sss(2)}$ are sufficiently large, each of them is likely to contain at least one vertex from the giant component.
Hence we may choose $N_1\in\N$ such that, given the event $\set{\abs{U^{\sss(1)}},\abs{U^{\sss(2)}}\geq N_1}$, there will exist a pair of vertices $u_1\in U^{\sss(1)},u_2\in U^{\sss(2)}$ connected by a path in $G'_{n-2N-2}$ of length at most $N_1 \log n$, with conditional probability at least $1-\epsilon$ for $n$ sufficiently large.

Since the sizes $\abs{\mathscr{U}^{\sss(j)}}$ are independent Poisson random variables with mean $\epsilon N$, we may choose $N$ large enough that $\abs{\mathscr{U}^{\sss(j)}}\geq N_1$ with probability at least $1-\epsilon$.
Moreover, since $\explore'_{2N},\mathscr{U}^{\sss(1)},\mathscr{U}^{\sss(2)}$ are finite and do not depend on $n$, it follows that $\mathcal{A}_n$ occurs with high probability and we can choose $N_2$ large enough that the diameters of $\explore'_{2N},\mathscr{U}^{\sss(1)},\mathscr{U}^{\sss(2)}$ are at most $N_2$ with probability at least $1-\epsilon$.

Because of conditional independence, we may choose the vertex sets $U^{\sss(j)}=\pi_M(\mathscr{U}^{\sss(j)})$ in the preceding discussion.
Taking the intersection of all the events above, it follows that, with probability at least $1-2\epsilon-o(1)$, there is a path in $K_n$ between vertices 1 and 2 consisting of at most $N_2$ edges of FPP weight at most $f_n(X_{\max}^{\sss(1)})$; a single edge of weight at most $f_n(X_{\max}^{\sss(1)}+\epsilon)$; at most $N_1\log n$ edges of weight at most $g(\tfrac{1}{n}(1+\epsilon))=f_n(1+\epsilon)$; a single edge of weight at most $f_n(X_{\max}^{\sss(2)}+\epsilon)$; and at most $N_2$ edges of weight at most $f_n(X_{\max}^{\sss(2)})$.
Therefore, with probability at least $1-2\epsilon-o(1)$,
\begin{align}
W_n &\leq (2N_2+2+N_1 \log n) f_n\left( (X_{\max}^{\sss(1)} \vee X_{\max}^{\sss(2)} \vee 1) + \epsilon \right)
\notag\\&
\leq (2N_2+2+N_1 \log n) f_n\left( (M_0^{\sss(1)} \vee M_0^{\sss(2)})+\epsilon \right)
.
\labelPartI{WnStochUpperBound}
\end{align}

To complete the proof, it suffices to show that the right-hand side of \eqrefPartI{WnStochUpperBound} is at most $f_n((M_0^{\sss(1)} \vee M_0^{\sss(2)})+2\epsilon )$ with high probability.
Since $M_0^{\sss(1)}$ and $M_0^{\sss(2)}$ do not depend on $n$ and satisfy $M_0^{\sss(j)}\geq 1$ a.s., it suffices to show that, for each fixed $x\geq 1$, we have $(2N_2+2+N_1 \log n)f_n(x) \leq f_n(x+\epsilon)$ for $n$ sufficiently large.
Taking $R=x+\epsilon$ in \refcond{boundfn}~\refitem{boundfnR}, we can mimic the proof of \reflemma{ExtendedImpliesWeak} to find that
\begin{align}
\frac{f_n(x+\epsilon)}{(2N_2+2+N_1 \log n)f_n(x)}
&\geq \frac{1}{(2N_2+2+N_1 \log n)} \left( \frac{x+\epsilon}{x} \right)^{\epsilonConditiona s_n}
\notag\\&
=\exp\left( \epsilonConditiona s_n\log(1+\epsilon/x)-\log\log n - O(1) \right).
\labelPartI{fnRatioLowerBound}
\end{align}
Since $s_n/\log\log n\to\infty$, the right-hand side of \eqrefPartI{fnRatioLowerBound} diverges to infinity and therefore exceeds 1 for $n$ sufficiently large.
\end{proof}
\vskip0.5cm

\noindent
{\bf{Acknowledgements:}} A substantial part of this work has been done at Eurandom and Eindhoven University of Technology.
ME and JG are grateful to both institutions for their hospitality.\\
The work of JG was carried out in part while at Leiden University (supported in part by the European Research Council grant VARIS 267356), the Technion, and the University of Auckland (supported by the Marsden Fund, administered by the Royal Society of New Zealand) and JG thanks his hosts at those institutions for their hospitality.
The work of RvdH is supported by the Netherlands
Organisation for Scientific Research (NWO) through VICI grant 639.033.806
and the Gravitation {\sc Networks} grant 024.002.003.

\vskip1cm

\labelPartI{notation}
\paragraph{A short guide to notation:}
\begin{itemize}
\item $\SWT^{\sss(v)}_t$ is SWT from vertex $v$.
We abbreviate $\SWT^{\sss(\emptyset_j)}=\SWT^{\sss(j)}$, $j\in \set{1,2}$.
\item$\cS^{\sss(j)}_t$ is the SWT from vertex $j\in \set{1,2}$ such that $\cS^{\sss(1)}$ and $\cS^{\sss(2)}$ cannot merge and with an appropriate freezing procedure.
\item $\cS_t=\cS_t^{\sss(1)}\cup \cS_t^{\sss(2)}$
\item $\BP^{\sss(j)}$ is branching process copy number $j$ where $j\in \set{1,2}$, without freezing.
\item $\cluster^{\sss(j)}$ is branching process copy number $j$ where $j\in \set{1,2}$, with freezing.
\item $\cluster_t$ is the union of 2 CTBPs with the appropriate freezing of one of them.
\item $\thinnedcluster_t$ is the union of 2 CTBPs with the appropriate freezing of one of them, and the resulting thinning.
Thus, $\thinnedcluster_t$ has the same law as the frozen $\cS_t$.
\item $f_n$ is the function with $Y_e^{\sss(K_n)}\overset{d}{=} f_n(nE)$, where $E$ is exponential with mean $1$.
\item $\mu_n$ is the image of the Lebesgue measure on $(0,\infty)$ under $f_n$
\item $\lambda_n$ is the exponential growth rate of the CTBP, cf.
\eqrefPartI{lambdanDefn}.
\end{itemize}

{\small \bibliographystyle{plain}
\bibliography{FPP}}

\end{document}